\documentclass[11pt]{article}

\usepackage{mathtools}
\usepackage{amsmath, amsthm, amsfonts,amssymb}
\usepackage{enumitem}
\usepackage{graphicx}

\usepackage{colortbl}
\usepackage[utf8]{inputenc}
\usepackage{amsmath}
\usepackage{amsfonts}
\usepackage{amssymb}
\usepackage{esint}
\usepackage{mathrsfs}
\usepackage{subfig,float}
\usepackage[T1]{fontenc}
\usepackage{amsmath, amsthm, amsfonts, amssymb}
\usepackage{mathrsfs}  
\usepackage{bbm} 
\usepackage{enumitem}
\usepackage{dsfont}
\newcommand{\set}[1]{\left\{#1\right\}}

\newcommand{\abs}[1]{\left\vert#1\right\vert}

\newcommand{\norm}[1]{\left\Vert#1\right\Vert}

\def\R{\mathbb{R}}

\usepackage[utf8]{inputenc}

\usepackage[T1]{fontenc}

\usepackage{amsmath, amsthm, amsfonts, amssymb}
\usepackage{mathrsfs}           

\usepackage[colorlinks=true,linkcolor=blue,citecolor=blue,urlcolor=blue,breaklinks]{hyperref}

\usepackage{bbm} 

\usepackage{dsfont}

\usepackage{enumitem}

\usepackage{hyperref}
\usepackage{breakurl}
\usepackage[square,sort,comma,numbers]{natbib}
\usepackage{url}
\usepackage[square,sort,comma,numbers]{natbib}
\def\R{\mathbb{R}}

\def\NN{\mathbb{N}}
\def\Q{\mathbb{Q}}

\def\d{\,\mathrm{d}}
\def\dx{\,\mathrm{d}x}
\def\dt{\,\mathrm{d}t}

\newcommand{\en}{\mathcal{H}}

\DeclareMathOperator{\divergence}{div}
\newcommand{\dv}[1]{\divergence \left(#1\right)}

\newcommand{\OmT}{\Omega\times (0,T)}
\let\pa\partial
\let\na\nabla

\newcommand{\cD}{\mathcal{D}}

\newcommand{\rrhoA}{\sqrt{\rho_A}}
\newcommand{\rrhoB}{\sqrt{\rho_B}}
\newcommand{\rrhoAB}{\sqrt{\rho_A\rho_B}}

\newtheorem{thm}{Theorem}[section]

\newtheorem{lem}[thm]{Lemma}
\newtheorem{prp}[thm]{Proposition}
\newtheorem{claim}[thm]{Claim}

\theoremstyle{definition}
\newtheorem{dfn}[thm]{Definition}
\theoremstyle{remark}
\newtheorem{remark}[thm]{Remark}

\hypersetup{pdftitle={Gang Dynamics}}
\hypersetup{pdfauthor={Alethea B. T. Barbaro  \& Nancy Rodriguez  \& Havva Yoldaş \& Nicola Zamponi }}

\author{
Alethea B. T. Barbaro\footnote{Department of Mathematics, Applied Mathematics \& Statistics, Case Western Reserve University, 10900 Euclid Avenue, Yost Hall, Cleveland, Ohio 44106-7058, USA. abb71@case.edu} 
\and  Nancy Rodriguez \footnote{Engineering Center, ECOT 225, 526 UCB, Boulder, CO 80309-0526 nrod@unc.edu} 
\and Havva Yolda\c{s} \footnote{Camille Jordan Institute, Claude Bernard University of Lyon 1, Batiment Jean Braconnier, 21 Avenue Claude Bernard, 69622 Villeurbanne Cedex France. yoldas@math.univ-lyon1.fr} 
\and Nicola Zamponi \footnote{
University of Mannheim, School of Business Informatics and Mathematics, B6, 28, 68159 Mannheim, Germany. nzamponi@mail.uni-mannheim.de}}

\title{Analysis of a cross-diffusion model for rival gangs interaction in a city}

\makeindex

\begin{document}

\maketitle

\begin{abstract}
\noindent 
We study a two-species cross-diffusion model that is inspired by a system of convection-diffusion equations derived from an agent-based model on a two-dimensional discrete lattice. The latter model has been proposed to simulate gang territorial development through the use of graffiti markings. We find two energy functionals for the system that allow us to prove a weak-stability result and identify equilibrium solutions. We show that under the natural definition of weak solutions, obtained from the weak-stability result, the system does not allow segregated solutions. 
Moreover, we present a result on the long-term behavior of solutions in the case when the product of the masses of the densities are smaller than a critical value. This result is complemented with numerical experiments. 
\end{abstract}

\tableofcontents

\section{Introduction} \label{S:intro}
This article is devoted to the study of a two-population model with cross-diffusion:
\begin{equation} \label{E:system}
\begin{cases}
\pa_t \rho_A(t,x,y) = \frac{1}{4} \nabla \cdot \left( \nabla \rho_A(t,x,y) + 2 \beta c \rho_A(t,x,y) \nabla \rho_B(t,x,y) \right), \qquad x, y \in \Omega, \, t>0,\\
\pa_t \rho_B(t,x,y) = \frac{1}{4} \nabla \cdot \left( \nabla \rho_B(t,x,y) + 2 \beta c \rho_B(t,x,y) \nabla \rho_A(t,x,y) \right),  \qquad  x, y\in \Omega, \, t>0,
\end{cases}
\end{equation}
complemented with the initial data 
\begin{align} \label{E:1_IC}
\rho_A(0,x,y) = \rho_A^{in}(x,y) \text{ and } \rho_B(0,x,y) = \rho_B^{in}(x,y), \qquad  x, y \in \Omega,
\end{align}
and the homogeneous Neumann boundary conditions
\begin{align} \label{E:1_BC}
\pa_\nu \rho_A(t,x,y) =\pa_\nu \rho_B (t,x,y)=0 \qquad x, y \in \pa\Omega, \, t >0.
\end{align}
In system \eqref{E:system}-\eqref{E:1_BC}, $\beta$ and $c$ are positive parameters and $\Omega\subset \R^2$ a bounded domain. 
Such a system can arise, for example, by considering the following two-species segregation model involving two densities of agents, $\rho_A$ and $\rho_B$, along with respective marking densities $g_A$ and $g_B,$ introduced in \cite{AB18}:
\begin{align}\label{E:4system} 
\begin{cases}
\pa_t g_A(t,x,y) = c\rho_A(t,x,y) - g_A(t,x,y), & \qquad x,y \in \Omega, \, t>0,\\
\pa_t g_B(t,x,y) = c\rho_B(t,x,y) - g_B(t,x,y), & \qquad x,y \in \Omega, \, t>0,\\
\pa_t \rho_A(t,x,y) = \frac{1}{4} \nabla \cdot \left( \nabla \rho_A(t,x,y) + 2 \beta \rho_A(t,x,y) \nabla g_B(t,x,y) \right), & \qquad x,y \in \Omega, \, t>0,\\
\pa_t \rho_B(t,x,y) = \frac{1}{4} \nabla \cdot \left( \nabla \rho_B(t,x,y) + 2 \beta  \rho_B(t,x,y) \nabla g_A(t,x,y) \right), &\qquad x,y \in \Omega, \, t>0,
\end{cases}
\end{align} with homogeneous Neumann boundary conditions.
System \eqref{E:4system} models the dynamics of two competing groups that mark their territory, {\it e.g.} with graffiti, and whose movement strategies
is a combination of passive diffusion and directed movement towards the gradients of the marking densities of the competing groups.
To arrive at the reduced system \eqref{E:system}-\eqref{E:1_BC} from system \eqref{E:4system}, we assume that the marking densities equilibrate much more rapidly than the population densities. Hence we assume
\[ \pa_t g_A(t,x,y) = \pa_t g_B(t,x,y) = 0.\]
However, we remark that system \eqref{E:system}-\eqref{E:1_BC} can also be seen as a more general model of cross diffusion system where the inter-specific interactions can lead to segregation.  

The notion of cross-diffusion was initially motivated by Morisita's theory of environmental density \cite{M52, M71}, which brings to the 
forefront the influence that a population pressure has on the dispersal of a population due to the interface between individuals. In \cite{SKT79}, Shigesada, Kawasaki, and Teramoto introduced a behavioral model for the movement of individuals based on Morisita's observations.  
According to Morisita's theory, the movement of individuals is influenced by the following three factors: (i) random movement; 
(ii) population pressure due to mutual interference between individuals; and 
(iii) movement toward favorable places.  The population pressure due to the competing population leads to the cross-diffusion.  
To see this from a mathematical point of view, we consider our two populations, $\rho_A$ and $\rho_B$. 
Under the assumption of Fickian diffusion, we obtain a system of two equations: 
\begin{equation*} 
\begin{cases}
\partial_t \rho_A = \nabla J_A(\rho_A,\rho_B),\\
\partial_t \rho_B = \nabla J_B(\rho_A,\rho_B), 
\end{cases}
\end{equation*}
where $J_A, J_B$ are the flows of the populations $A$ and $B$, respectively.  The flow proposed by Shigesada, Kawasaki, and Teramoto has the form: 
\[ J_i = \nabla \left(\left(a_{iA} \rho_A + a_{iB} \rho_B + c_i\right) \rho_i\right), \text{ for } i \in \{A,B\}\]
with $a_{iA}, \, a_{iB}, \, c_i \geq 0.$

Soon after, Busenberg and Travis introduced some epidemic models with migration that also include cross-diffusion in \cite{BT83}. 
In their model, the authors assume that the population flow $J_i$, $i\in \set{A,B},$ are proportional to the gradient of a potential function, $\Psi,$ that only depends on 
the total population $P=\rho_A+\rho_B.$ The proportion is assumed to be the portion of the subpopulation $i$, which leads to the following form of the flow:
\[J_i = a\frac{\rho_i}{P}\nabla \Psi(P), \text{ for } i\in \set{A,B}.\]  
In \cite{GP84}, Gurtin and Pipkin introduced the potential function $\Psi(s) = s^2/2,$ which yields that: 
\begin{align}\label{BT}
J_i = a \rho_i\nabla(\rho_A+\rho_B), \text{ for }i\in \set{A,B}. 
\end{align}
In \cite{GS14}, Galiano and Selgas considered a more general version of the Gurtin and Pipkin model, where potential function depends on a general linear combination of the population densities, with the addition of random movement
and environmental effects.  
The most general version of the system they consider is:
$$
J_i(\rho_A, \rho_B) = \rho_i\nabla(a_{iA} \rho_A + a_{iB}\rho_B+b_iV)+c_i\nabla \rho_i,
$$
where $V$ is the environmental potential.  In \cite{GS14}, the existence of weak solutions for non-negative initial data in $L^\infty$ was proved in two parameter cases.  The first was under the condition that 
$$
4a_{AA}a_{BB}-(a_{AB}+a_{BA})^2>a_0,
$$
for some $a_0>0$.  This condition implies an ellipticity condition on the matrix $(a_{ij})_{i,j\in \set{A,B}}$ and can be relaxed. The authors were also be able to prove the existence of solutions in the case when
$a=a_{ij}$ for all $i,j\in \set{A,B}$, with $a>0$.

System \eqref{E:system}-\eqref{E:1_BC}, which we consider here, is a special case of this general model, where only passive diffusion and cross-diffusion are considered. Thus, the populations does not take into account the population pressure due to their own group. In particular, we assume that $a_{AA}, a_{BB}, b_A, b_B$ are all equal to zero. Thus, this case falls outside of the two cases considered in \cite{GS14}.  
It is worth noting that systems with local self- and cross-diffusion have found many applications, for example in, illicit trade of drugs \cite{E97}; 
epidemic models with diffusion of polymorphic populations \cite{BT83}; \vspace{3pt}
models for overcrowding effect with nonuniform ease of dispersal for different individuals \cite{GP84};
opinion dynamics \cite{SWS00}; and biochemical reactions \cite{VE09}.  

Many analytical results for cross-diffusion systems are available in the literature. For example, in \cite{CDJ18} Chen et al.~consider a reaction-cross-diffusion model for an arbitrary number of competing populations which, in the case of linear transition rates, extends the two-species SKT model presented in \cite{SKT79}. Existence of  global-in-time weak solutions to the model in a bounded domain with homogeneous Neumann boundary conditions is shown via an entropy method and an approximation scheme. Crucial conditions on the diffusion matrix are either weak cross-diffusion or detailed balance. Another cross-diffusion system where the diffusion matrix depends linearly on the densities is the two-species ion transport model through narrow membrane channels studied by Burger et al.~in \cite{BFPS10}. There the authors prove global existence of weak solutions to the equations in a bounded domain with no-flux boundary conditions via an entropy method, as well as global existence of strong solutions near the equilibrium. The result is generalized in \cite{GJ18} to the case of an arbitrary number of species with different specific electrical charges and mixed Dirichlet-Neumann boundary condition. The systems considered in \cite{BFPS10} and \cite{GJ18} also present a degeneracy in the entropy structure, that is, some gradient estimates are lost when the densities approach some critical region. Other degenerate cross-diffusion models have been recently studied in literature; for example, a large class of population models with degenerate cross-diffusion was analyzed in \cite{JZ17}, while the combination of degenerate cross-diffusion and nonlocal interaction is a feature of both models considered in \cite{BBP17} and \cite{FEF18}. The problem of degenerate cross-diffusion in a moving domain was considered in \cite{BE15}, while the interplay between singularity and degeneracy was a major feature of the model studied in 
\cite{DMZ19}. Reaction-cross-diffusion systems with Laplacian structure have been considered by Desvillettes et al.~in \cite{DLM14, DLMT15, DT15}.

Cross-diffusion equations can be seen as a large class of nonlinear, strongly coupled evolution PDEs with the structure: 
\begin{align}\label{CD}
\pa_t \rho = \dv{A(\rho)\nabla \rho} \equiv \sum_{i=1}^{n}\pa_{x_i}( A(\rho)\pa_{x_i}\rho ),\quad \; x\in \R^n, \; t>0.
\end{align}
The unkown of the system, $\rho = \rho(t,x)\in\R^n$, usually represents a vector of densities or concentrations. Therefore, it should be nonnegative to be consistent with the physics; sometimes it is also required to be (uniformly) bounded for the same reason. 
Quite often, in cross-diffusion systems coming from the applied sciences the matrix $A(\rho)\in \R^{n\times n}$, the so-called \emph{diffusion matrix}, is neither symmetric nor positive semidefinite, which means that the standard coercivity-based approaches to the analysis of \eqref{CD} are ineffective. Moreover maximum/minimun principles are also usually unavailable due to the fact that $A$ is full and lacks a suitable structure. For these reasons, the analytical study of cross-diffusion equations is in general quite challenging.

A useful method, the so-called \emph{boundedness-by-entropy method}, in the analysis of (reaction-)cross-diffusion systems has been developed by J\"ungel and collaborators (see e.g.~\cite{J15} for a comprehensive review) after an idea found in Burger et al.~\cite{BFPS10}. This method is suitable for systems of evolution PDEs presenting a formal gradient flow structure, or an \emph{entropy structure}.  Specifically, it works for systems which can be written in the following form:
\begin{align}\label{gr.flow.0}
\pa_t\rho = \dv{\mathbb{M}\nabla\frac{\delta\en[\rho]}{\delta\rho}},\quad t>0,
\end{align}
where $\mathbb M$ is a positive semidefinite (often also symmetric) matrix and
$\frac{\delta\en[\rho]}{\delta\rho}$ is the Frech\'et derivative of the convex functional $\en$,
which is called the \emph{mathematical entropy} of the system. 
In many cases $\en[\rho]$ has the form 
\[\en[\rho] = \int_\Omega h(\rho) \mathrm{d} x,\] where $h$ a scalar convex function called \emph{entropy density}, then the object
$\frac{\delta\en[\rho]}{\delta\rho}$ can be identified, via Riesz representation theorem, with the gradient of $h$: 

\[\frac{\delta\en[\rho]}{\delta\rho}\simeq Dh(\rho), \]
which is referred to as \emph{entropy variable}. A first consequence of this formulation is that the functional $\en$ is a Lyapounov functional for \eqref{gr.flow.0}, that is;
it is nonincreasing in time along the solutions of \eqref{gr.flow.0}:
\begin{align*}
\frac{\d}{\d t}\en[\rho(t)] = - \int_\Omega \nabla Dh(\rho)\cdot \mathbb{M}\nabla Dh(\rho) \mathrm{d}x \leq 0,\quad t>0,
\end{align*}
since $\mathbb{M}$ is positive semidefinite by assumption\footnote{Here we assumed homogeneous Neumann boundary conditions and no reaction terms, that is the right-hand side of \eqref{CD} is zero. If any of these conditions are not verified, the quantity $\frac{\d}{\d t}\en[\rho(t)]$ might be positive. However, suitable compatibility conditions usually ensure that the entropy $\en[\rho(t)]$ remains at every time upper bounded via a Gronwall argument.}. Furthermore, if $Dh : \mathcal{D}\to\R^n$ is a globally invertible mapping, than the physical variable $\rho$ can be written in terms of the entropy variable $w = Dh(\rho)$ via $\rho = (Dh)^{-1}(w)$.
As a consequence $\rho\in\mathcal{D}$ whenever $w\in\R^n$. So, if \eqref{gr.flow.0} can be written and solved in terms of $w$, then the constraint $\rho(x,t)\in\mathcal{D}$ will hold whenever $w(x,t)$ is finite (that is, for a.e.~$x,t$, provided that $w$ is integrable). In particular, if $\mathcal{D}\subset\R^n$ is bounded, then $\rho\in L^\infty$ with bounds that only depend on $\mathcal{D}$; similarly, if $\mathcal{D}\subset\R^n_+$, then $\rho$ has nonnegative components. 
These ideas can be exploited to formulate an existence argument which proceeds roughly in three steps: (i) writing an approximate scheme which yields a sequence of approximate solutions to \eqref{gr.flow.0}; (ii) deriving an entropy balance inequality which yields gradient estimates for the approximate solution; (iii) showing via suitable compactness result that the approximate sequence has a converging subsequence and taking the limit in the approximate system to recover a weak solution to \eqref{gr.flow.0}.

Unfortunately, this program cannot be straightforwardly carried out when studying \eqref{E:system}-\eqref{E:1_BC} because of its extremely degenerate structure. Indeed, a standard entropy (formal gradient flow) structure requires the existence of a {\em convex} entropy functional, which cannot be the case for \eqref{E:system}-\eqref{E:1_BC}. Precisely, a necessary condition for a cross-diffusion system \eqref{CD} to admit a convex entropy is the normal ellipticity of the differential operator $\rho\mapsto \dv{A(\rho)\nabla\rho}$, that is,
the property that the real part of every eigenvalue of $A(\rho)$ is nonnegative \cite[Lemma 3.2]{J17}. This property is not verified by \eqref{E:system}-\eqref{E:1_BC}; as a matter of fact, the diffusion matrix $A(\rho)$ in \eqref{E:system}-\eqref{E:1_BC} has one positive and one negative eigenvalue in the region $\{(\rho_A,\rho_B)\in\R^2_+ \, \mid \, \rho_A\rho_B>1\}$. Consistently with this fact, the only Lyapounov functional that is known for \eqref{E:system}-\eqref{E:1_BC} is nonconvex. Furthermore, the property that the mapping $Dh : \R^2_+\to\R^2$ being invertible also fails for \eqref{E:system}-\eqref{E:1_BC}, since such a mapping is not even one-to-one. 


A search for a workaround to counter these formidable difficulties and obtain nonetheless some global-in-time existence result for \eqref{E:system}-\eqref{E:1_BC} has been unsuccessful, and only a local-in-time existence result is available for \eqref{E:system}-\eqref{E:1_BC}, which comes from Amann's theory \cite[Thr.~3.1]{J17} and holds under the assumption that the initial datum is $W^{1,p}$ with $p>d=2$ and takes values in the region $\{(\rho_A,\rho_B)\in\R^2_+ \, \mid \, \rho_A\rho_B<1\}$.
However, we are also able to provide a weak-stability result, which holds (in spite of the very degenerate structure of the system) for generic weak solutions taking values in $\R^2_+$. Such result is a key step in the proof of the global well-posedness and provides evidence that the system is not likely to be ill-posed. 
We believe this to be a remarkable result, given the strongly degenerate structure of the system,
the critical loss of normal ellipticity properties for the right-hand side of \eqref{E:system} in the region $\{(\rho_A,\rho_B)\in\R^2_+\, \mid \, \rho_A\rho_B>1\}$,
the lack of a globally convex energy functional. To the best of our knowledge, this is the first result of this kind dealing with such a strongly degenerate system.
Unfortunately, finding an approximation to \eqref{E:system}-\eqref{E:1_BC}
for which we can prove existence and then apply the weak stability result has been a challenge and remains an open problem. 


The paper is organized as follows: In Section \ref{S:energy}, we give two energy functionals that the system \eqref{E:system}-\eqref{E:1_BC} attains. These energy functionals are the key tools that help us obtain complementary estimates on the solutions. We also present these complementary a-priori estimates in Section \ref{S:energy}. A Maxwell-Boltzmann entropy functional holds under some constraints on the solutions, mainly that $\rho_A\rho_B<1$. Moreover, in Section \ref{S:ss}, we are interested in understanding the stationary states of \eqref{E:system}-\eqref{E:1_BC},  as they outline the possible long-term behavior of the evolution problem. Section \ref{S:existence} is dedicated to the existence analysis. An important question to consider is whether segregated steady states, which are physical in many situations, arise.  
Based on the natural definition of weak solutions, obtained from the weak stability result, we show that all steady states must be constant.  This implies that perhaps the use of the entropy structure is not suitable to study segregated solutions.  An alternative is that the model actually does not capture the physical property of segregation. 	Our final result is on the long-term behavior of solutions to \eqref{E:system}-\eqref{E:1_BC}, in the case when the product of the population densities $\rho_A,\rho_B$ are small; see Section \ref{S:long_time}. We complement these results with some numerical simulations in Section \ref{S:NR}.

\section{Energy functionals and a-priori estimates} \label{S:energy}
For the sake of simplicity we assume in this section that $2\beta c = 1$ and neglect the prefactor $1/4$ in front of the divergence in \eqref{E:system}. Note that this does not influence the existence analysis as it follows from a simple rescaling of the system.
In the following we denote $\rho = (\rho_A, \rho_B)$, and $\Omega\subset\R^2$ is an open, bounded set with Lipschitz boundary.

\subsection{Two energy functionals}
In this section we present two energy functionals that will be useful in obtaining bounds for $\rho_A$ and $\rho_B$.  
Let us define the first energy functional $\mathcal{H}[\rho]$:
\begin{align}
\mathcal{H}[\rho] = \int_\Omega h(\rho_A(x),\rho_B(x)) \mathrm{d} x,
\label{E:energy}
\end{align} where 
\[ h(\rho_A,\rho_B)=\rho_A \log \rho_A - \rho_A +
\rho_B\log \rho_B - \rho_B + \rho_A\rho_B .\]

\noindent System \eqref{E:system}-\eqref{E:1_BC} is a formal gradient flow with respect to $\mathcal{H}$:
\begin{align}\label{gr.flow}
\pa_t\rho = \dv{\mathbb{M}\nabla\frac{\delta\en[\rho]}{\delta\rho}},\quad
\text{with  }\mathbb{M} = \begin{pmatrix}
\rho_A & 0\\ 0 & \rho_B	
\end{pmatrix},
\end{align}
where $\frac{\delta\en[\rho]}{\delta\rho}$ is the Frech\'et derivative of $\en$, which can be identified, via Riesz representation theorem, with the gradient of $h$: 
$$
\frac{\delta\en[\rho]}{\delta\rho}\simeq Dh(\rho_A,\rho_B) = 
(\log\rho_A + \rho_B, \log\rho_B + \rho_A)^\top .
$$
The matrix $\mathbb M$ is positive semidefinite in $\R^2_+ \equiv [0,\infty)^2$.
Testing \eqref{gr.flow} against $Dh(\rho_A,\rho_B)$
yields the energy balance equation:
\begin{equation}
\frac{\mathrm{d}\mathcal{H}[\rho]}{\mathrm{d} t} + \int_\Omega \left \{ \rho_A \abs{ \nabla \left(\log \rho_A + \rho_B \right)}^2 + \rho_B \abs{\nabla \left(\log \rho_B + \rho_A \right)}^2 \right \}  \mathrm{d}x = 0.\label{ei.A}
\end{equation}
On the other hand, system \eqref{E:system}-\eqref{E:1_BC} also admits another gradient-flow structure under some restrictions on its solution.
Let us define the open set:
$$
\cD = \big \{ (\rho_A,\rho_B)\in \R^2_+ ~ \mid ~ \rho_A\rho_B<1  \big \} ,
$$
and the Maxwell-Boltzmann entropy functional:
\[\mathcal H_{MB}[\rho] = \int_\Omega h_{MB}(\rho_A(x),\rho_B(x)) \mathrm{d}x, \]
where \[ h_{MB}(\rho_A,\rho_B) = \rho_A\log\rho_A - \rho_A 
+ \rho_B\log\rho_B - \rho_B. \]
Then \eqref{E:system}-\eqref{E:1_BC} can be rewritten as
\begin{align}\label{gr.flow.2}
\pa_t\rho = \dv{ \mathbb{M}' \nabla \frac{\delta \en_{MB}[\rho]}{\delta\rho}},\quad \text{with  } \mathbb{M}' = \begin{pmatrix}
	\rho_A & \rho_A\rho_B\\
	\rho_A\rho_B & \rho_B
\end{pmatrix},
\end{align}
where
$$
\frac{\delta\en_{MB}[\rho]}{\delta\rho}\simeq Dh_{MB}(\rho_A,\rho_B) = 
(\log\rho_A, \log\rho_B )^\top .
$$
We remark that $\mathbb M'$ is positive semi-definite on $\overline\cD$.
Testing \eqref{gr.flow.2} against $Dh_{MB}(\rho_A,\rho_B)$ yields the balance equation for $\en_{MB}[\rho]$:
\begin{align*}
	\frac{\mathrm{d}}{\mathrm{d}t}\en_{MB}[\rho] &= \int_\Omega (\log\rho_A \pa_t\rho_A + \log\rho_B\pa_t \rho_B ) \mathrm{d}x\\
	&= - \int_\Omega ( \rho_A^{-1}\na\rho_A\cdot( \na\rho_A + \rho_A\na\rho_B ) + \rho_B^{-1}\na\rho_B\cdot( \na\rho_B + \rho_B\na\rho_A ))  \mathrm{d}x\\
	&= -\int_\Omega ( \rho_A^{-1}|\na\rho_A|^2 + \rho_B^{-1}|\na\rho_B|^2 + 2\na\rho_A\cdot\na\rho_B ) \mathrm{d}x\\
	&= -4\int_\Omega ( |\na\sqrt{\rho_A}|^2 + |\na\sqrt{\rho_B}|^2 + 2\sqrt{\rho_A\rho_B}\na\sqrt{\rho_A}\cdot\na\sqrt{\rho_B} )  \mathrm{d}x\\
	&= -4\int_\Omega ( \sqrt{\rho_A\rho_B}|\na(\sqrt{\rho_A}+\sqrt{\rho_B})|^2 + (1-\sqrt{\rho_A\rho_B})(|\na\sqrt{\rho_A}|^2 + |\na\sqrt{\rho_B}|^2))  \mathrm{d}x .
\end{align*}
Summarizing up gives the following:
\begin{align}\label{ei.B}
	\frac{\mathrm{d}}{\mathrm{d}t}\mathcal H_{MB}(\rho) + 4\int_\Omega ( \sqrt{\rho_A\rho_B}|\na(\sqrt{\rho_A}+\sqrt{\rho_B})|^2 + (1-\sqrt{\rho_A\rho_B})(|\na\sqrt{\rho_A}|^2 + |\na\sqrt{\rho_B}|^2)) \d x = 0.
\end{align}
\begin{remark}
	We point out that \eqref{ei.B} is only useful if $\rho_A\rho_B\leq 1$, otherwise we obtain terms we cannot control.
\end{remark}


\subsection{A-priori estimates}
In this section, we give a-priori estimates on the agent densities $\rho_A$ and $\rho_B$. The estimates are obtained from energy balance equations \eqref{ei.A} and \eqref{ei.B}.

Throughout the section we assume  that the initial datum $\rho^{in}\in L^2(\Omega)$, where $\Omega \in \R^2$ is an open, bounded domain with Lipschitz boundary. As a consequence $\en_{MB}[\rho^{in}]\leq \en[\rho^{in}] < \infty$. Also, we denote $\Omega_T\equiv \Omega\times (0,T)$ for every $T>0$.

\begin{lem}[mass conservation]
System \eqref{E:system}-\eqref{E:1_BC} conserves mass. In particular we have the following estimate:
	\begin{equation}
	\label{rhoL1Linf}
	\|\rho_i\|_{L^\infty(0,T; L^1(\Omega))} = \|\rho_{i}^{in}\|_{L^1(\Omega)},\quad i\in \{A,B\}.
	\end{equation}
\end{lem}
\begin{proof}
	Integrating \eqref{E:system}-\eqref{E:1_BC} in $\Omega$ yields 
	\begin{equation}
	\label{mass.cons}
	\int_\Omega\rho_i(t) \mathrm{d}x = \int_\Omega\rho_{i}^{in} \mathrm{d}x,\quad i \in \{A,B\},~~t>0.
	\end{equation} Thus \eqref{rhoL1Linf} holds.
\end{proof}
\subsubsection{Estimates from ``natural'' energy balance equation \eqref{ei.A}.}
\begin{lem} We obtain the following estimates for $\rho$:
	\begin{align}
	\label{est.1b}
	\norm{ \left( 1 - \sqrt{\rho_A \rho_B} \right) \nabla 
		\sqrt{\rho_i}}_{L^2(\Omega_T)} &\leq C,\quad i \in \{A,B\},\\
	\label{est.1a}
	\norm{ \left( 1 + \sqrt{\rho_A \rho_B} \right) \nabla \left( \sqrt{\rho_A}+\sqrt{\rho_B} \right) }_{L^2(\Omega_T)} &\leq \tilde{C},
	\\	\label{est.na.p} \|(\sqrt{\rho_A\rho_B}-1)^{2}\|_{L^{4/3}(0,T; W^{1,4/3}(\Omega))} &\leq C_T,
	\end{align} where $C, \tilde{C}, C_T >0$ are some constants, $C_T$ depending on $T>0$. Moreover, the following estimates hold true for $\sqrt{\rho_A} + \sqrt{\rho_B}$:
	\begin{align}
	\label{est.na.sum}
	\|\sqrt{\rho_A} + \sqrt{\rho_B}\|_{L^2(0,T; H^1(\Omega))}\leq \tilde{C}_T, \\
	\label{est.sum.L4}
	\|\sqrt{\rho_A}+\sqrt{\rho_B}\|_{L^4(\Omega_T)}\leq 
	\hat{C}_T,
	\end{align}where $\tilde{C}, \hat{C}_T >0$ some constant depending on $T>0$.
\end{lem}
\begin{proof}
	 Integrating \eqref{ei.A} in the time interval $[0,T]$ with $T>0$ arbitrary leads to
		\begin{align}\label{ei.A.int}
		\en[\rho(T)] + 
		4\int_{\Omega_T}\left( \abs{ \nabla \sqrt{\rho_A} + \sqrt{ \rho_A \rho_B} \, \nabla \sqrt{\rho_B} }^2 + \abs{ \nabla \sqrt{\rho_B} + \sqrt{ \rho_A \rho_B} \, \nabla \sqrt{\rho_A} }^2
		\right) \mathrm{d}x  \mathrm{d}t \leq \en[\rho^{in}] .
		\end{align}
		However, since $2(x^2 + y^2)\geq (x\pm y)^2$ for every $x,y\in\R$, we deduce
		\begin{align*}
			\left |\nabla \sqrt{\rho_A} + \sqrt{\rho_A \rho_B}\,  \nabla \sqrt{\rho_B} \right |^2 + |\nabla \sqrt{\rho_B} + \sqrt{\rho_A \rho_B} \, \nabla \sqrt{\rho_A}|^2  \geq \frac{1}{2} \left | (1 + \sqrt{\rho_A \rho_B}) \nabla (\sqrt{\rho_A} + \sqrt{\rho_B}) \right |^2
		\end{align*} which, together with \eqref{ei.A.int}, yield \eqref{est.1a}. On the other hand, given that $1 + x \geq |1-x|$ for $x \geq 0$, the above inequality yields
	\begin{align*}
		\left |\nabla \sqrt{\rho_A} + \sqrt{\rho_A \rho_B} \, \nabla \sqrt{\rho_B} \right |^2 + |\nabla \sqrt{\rho_B} + \sqrt{\rho_A \rho_B} \, \nabla \sqrt{\rho_A}|^2  \geq \frac{1}{2} \left | (1 - \sqrt{\rho_A \rho_B}) \nabla (\sqrt{\rho_A} + \sqrt{\rho_B}) \right |^2
	\end{align*} while (since $2(x^2 + y^2)\geq (x\pm y)^2$ for every $x,y\in\R$) the following inequality is also true
	\begin{align*}
		\left |\nabla \sqrt{\rho_A} + \sqrt{\rho_A \rho_B} \, \nabla \sqrt{\rho_B} \right |^2 + |\nabla \sqrt{\rho_B} + \sqrt{\rho_A \rho_B} \, \nabla \sqrt{\rho_A}|^2  \geq \frac{1}{2} \left | (1 - \sqrt{\rho_A \rho_B}) \nabla (\sqrt{\rho_A} - \sqrt{\rho_B}) \right |^2
	\end{align*} By summing the two previous inequalities and exploiting the elementary property $2(x^2 + y^2) \geq (x \pm y)^2$ for every $x,y \in \R$ as well as \eqref{ei.A.int} we obtain \eqref{est.1b}.
		
The definition of $\en$ and \eqref{ei.A.int} lead to
\begin{equation}
\label{est.L1logL}
\|\rho_A\log\rho_A\|_{L^\infty(0,T; L^1(\Omega))} + 
\|\rho_B\log\rho_B\|_{L^\infty(0,T; L^1(\Omega))}\leq C.
\end{equation}Then \eqref{est.na.sum} follows. 
The following Gagliardo-Nirenberg inequality holds since $\Omega\subset\R^2$:
\begin{align}\label{GN}
\|u\|_{L^4(\Omega)}\leq C_{GN}\|u\|_{L^2(\Omega)}^{1/2}\|u\|_{H^1(\Omega)}^{1/2}, \text{  for all  } u\in H^1(\Omega) .
\end{align}
Choosing $u=\sqrt{\rho_A}+\sqrt{\rho_B}$ in the above inequality and integrating it in time lead to
$$
\int_0^T\|\sqrt{\rho_A}+\sqrt{\rho_B}\|_{L^4(\Omega)}^4 \mathrm{d} t\leq C_{GN}^4 \left( \sup_{t\in [0,T]}\|\sqrt{\rho_A}(t)+\sqrt{\rho_B}(t)\|_{L^2(\Omega)}^{2}\right)
\int_0^T\|\sqrt{\rho_A}(t)+\sqrt{\rho_B}(t)\|_{H^1(\Omega)}^{2} \mathrm{d}t,
$$
which, thanks to \eqref{rhoL1Linf}, \eqref{est.na.sum}, leads to \eqref{est.sum.L4}.
		

From \eqref{est.1b}, \eqref{est.sum.L4} and the identity
\begin{align*}
\frac{1}{2}\nabla [(\sqrt{\rho_A\rho_B}-1)^{2}] &= 
(\sqrt{\rho_A\rho_B}-1)\nabla\sqrt{\rho_A\rho_B} \\
&= \sqrt{\rho_B}(\sqrt{\rho_A\rho_B}-1)\nabla\sqrt{\rho_A} + \sqrt{\rho_A}(\sqrt{\rho_A\rho_B}-1)\nabla\sqrt{\rho_B} 
\end{align*}
we deduce via H\"{o}lder inequality that 
\begin{align*}
	\frac{1}{2} &\|\nabla [(\sqrt{\rho_A \rho_B}-1)^2]\|_{L^{4/3}(\Omega_T)}\\
	&\leq \|\sqrt{\rho_B}\|_{L^4(\Omega_T)} \|(\sqrt{\rho_A \rho_B} -1) \nabla \sqrt{\rho_A}\|_{L^2(\Omega_T)} + \|\sqrt{\rho_A}\|_{L^4(\Omega_T)} \|(\sqrt{\rho_A \rho_B} -1) \nabla \sqrt{\rho_B}\|_{L^2(\Omega_T)} 
\end{align*} so \eqref{est.1b}, \eqref{est.sum.L4} lead to 
$$
\|\nabla [(\sqrt{\rho_A\rho_B}-1)^{2}] \|_{L^{4/3}(\Omega_T)}\leq C_T .
$$
Since $(\sqrt{\rho_A\rho_B}-1)^{2}\leq C(1+\rho_A^2+\rho_B^2)$,
from the above estimate and \eqref{est.sum.L4}, as well as Poincar\'e's Lemma, we obtain \eqref{est.na.p}.

\end{proof}
\begin{lem}[estimate on $\rho_A \rho_B$]
	We have the following estimate for the product of the agent densities $\rho_A$ and $\rho_B$
	\begin{equation}
	\label{est.Hm1.b}
	\|\rho_A\rho_B\|_{L^{3/2}(\Omega_T)}\leq C_T,
	\end{equation}
	where $C_T > 0$ is a constant depending on $T>0$.
\end{lem}
\begin{proof}
	We give the proof by using the so-called $H^{-1}$ method,
	i.e.~by testing \eqref{E:system}-\eqref{E:1_BC} against $\psi\approx (-\Delta)^{-1}(\rho_A + \rho_B)$. For $t\in (0,T)$ define the function $\psi(t)$ as the only solution to
	\begin{align} \label{dual}
	\begin{cases} 
	-\Delta\psi(t) = \rho_A(t) + \rho_B(t) - \langle( \rho_A(t) + \rho_B(t))\rangle, \quad &\text{in }\Omega, \\
	\pa_\nu\psi(t) = 0\quad &\text{on }\pa\Omega, 
	\end{cases}
	\end{align}
	and 
	\[\int_\Omega\psi(t) \mathrm{d}x = 0,\]
	where 
	\[\langle\rho\rangle\equiv \int_{\Omega} \frac{\rho }{|\Omega|} \mathrm{d}x. \]
	We remark that $\langle\rho(t)\rangle = \langle\rho^{in}\rangle$ is constant in time thanks to \eqref{mass.cons}.
	Let us first compute the following: 
	\begin{align*}
	&\frac{\mathrm{d}}{\mathrm{d}t}\int_\Omega\frac{1}{2}|\nabla\psi|^2 \mathrm{d}x = 
	\int_\Omega \nabla\psi\cdot\nabla\pa_t\psi \mathrm{d}x = 
	\int_\Omega \psi\pa_t(-\Delta\psi) \mathrm{d}x = 
	\int_\Omega \psi\pa_t (\rho_A + \rho_B) \mathrm{d}x .
	\end{align*}
	Therefore, testing each equation in \eqref{E:system}-\eqref{E:1_BC} against $\psi$ and summing the equations lead to
	\begin{align*}
	\frac{\mathrm{d}}{\mathrm{d}t}\int_\Omega\frac{1}{2}|\nabla\psi|^2 \mathrm{d}x &= 
	-\int_\Omega\left( 
	\nabla\rho_A + \rho_A\nabla\rho_B + 
	\nabla\rho_B + \rho_B\nabla\rho_A
	\right)\cdot\nabla\psi \mathrm{d}x\\
	&= -\int_\Omega\nabla\left( 
	\rho_A + \rho_B + \rho_A\rho_B
	\right)\cdot\nabla\psi \mathrm{d}x\\
	&= -\int_\Omega\left( 
	\rho_A + \rho_B + \rho_A\rho_B
	\right)(\rho_A + \rho_B - \langle\rho_A + \rho_B\rangle) \mathrm{d}x .
	\end{align*}
	Thanks to the mass conservation \eqref{mass.cons}, the energy balance \eqref{ei.A} and the fact that $\rho_A\rho_B\leq C + h(\rho_A,\rho_B)$,
	we deduce \begin{align*}
	\frac{\mathrm{d}}{\mathrm{d}t}\int_\Omega\frac{1}{2}|\nabla\psi|^2 \mathrm{d}x +
	\int_\Omega\left( 
	\rho_A + \rho_B + \rho_A\rho_B
	\right)(\rho_A + \rho_B) \mathrm{d}x \leq C,
	\end{align*}
	and integrating the above inequality in $[0,T]$ yields \begin{align*}
	\int_\Omega|\nabla\psi(T)|^2  \mathrm{d}x + 
	\int_0^T\int_\Omega\left( 
	\rho_A + \rho_B + \rho_A\rho_B
	\right)(\rho_A + \rho_B) \mathrm{d}x  \mathrm{d}t \leq C_T + 
	\int_\Omega|\nabla\psi(0)|^2  \mathrm{d}x .
	\end{align*} 
	Testing \eqref{dual} against $\psi(t)$ and exploiting Poincar\'e's Lemma (remember that $\int_\Omega\psi(t) \mathrm{d}x = 0$) lead to \begin{align*}
	\|\nabla\psi(t)\|_{L^2(\Omega)}^2 &\leq \int_\Omega (\rho_A(t)+\rho_B(t))\psi(t) \mathrm{d}x\leq
	\|\rho_A(t)+\rho_B(t)\|_{L^2(\Omega)}\|\psi(t)\|_{L^2(\Omega)}\\
	&\leq C_P
	\|\rho_A(t)+\rho_B(t)\|_{L^2(\Omega)}\|\nabla\psi(t)\|_{L^2(\Omega)}
	\end{align*}
	which means $$
	\|\nabla\psi(t)\|_{L^2(\Omega)}\leq C_P
	\|\rho_A(t)+\rho_B(t)\|_{L^2(\Omega)} ,\quad t\in [0,T].
	$$
	In particular, since $\rho^{in}\in L^2(\Omega)$ by assumption, it follows that $\|\nabla\psi(0)\|_{L^2(\Omega)}\leq C$, so we conclude that
	\begin{align*}
	\int_\Omega|\nabla\psi(T)|^2 \mathrm{d}x + 
	\int_0^T\int_\Omega\left( 
	\rho_A + \rho_B + \rho_A\rho_B
	\right)(\rho_A + \rho_B) \mathrm{d}x  \mathrm{d}t \leq C_T .
	\end{align*}It follows
	\begin{equation*}
	\int_{\Omega_T} (\rho_A + \rho_B)\rho_A\rho_B \, \mathrm{d}x  \mathrm{d}t\leq C_T,
	\end{equation*}
	which, by Young's inequality,  
	$\rho_A + \rho_B\geq 2\sqrt{\rho_A\rho_B}$, leads to \eqref{est.Hm1.b}. 
\end{proof}

\begin{lem}[estimate on the fluxes] We have the following estimate for the fluxes 
	\begin{equation}
	\label{est.flux}
	\|\na\rho_A + \rho_A\na\rho_B\|_{L^{4/3}(\Omega_T)} + 
	\|\na\rho_B + \rho_B\na\rho_A\|_{L^{4/3}(\Omega_T)}\leq C_T.
	\end{equation} where $C_T$ is a constant depending on $T>0$.
\end{lem}
\begin{proof}
Since
\begin{align*}
\na\rho_A + \rho_A\na\rho_B &= 2\sqrt{\rho_A}\left( \na\sqrt{\rho_A} + \sqrt{\rho_A\rho_B}\na\sqrt{\rho_B} \right)\\
&= 2\sqrt{\rho_A}\left( \na(\sqrt{\rho_A}+\sqrt{\rho_B}) + (\sqrt{\rho_A\rho_B}-1)\na\sqrt{\rho_B} \right),
\end{align*}
from \eqref{est.1a}, \eqref{est.1b}, \eqref{est.sum.L4} it follows
\begin{align*}
\|\na\rho_A + \rho_A\na\rho_B\|_{L^{4/3}(\Omega_T)} &\leq 
2\|\sqrt{\rho_A}\|_{L^{4}(\Omega_T)}\left( \|\na(\sqrt{\rho_A}+\sqrt{\rho_B})\|_{L^{2}(\Omega_T)} + \|(\sqrt{\rho_A\rho_B}-1)\na\sqrt{\rho_B}\|_{L^{2}(\Omega_T)} \right)\\
&\leq C_T.
\end{align*}
Since a similar argument can be done for $\na\rho_B + \rho_B\na\rho_A$, we obtain \eqref{est.flux}.
\end{proof}
\begin{lem}[estimate on $\pa_t \rho_A$ and $\pa_t \rho_B$] We have the following estimate on the time derivative of $\rho_A$ and $\rho_B$:
	\begin{align}
	\|\pa_t\rho_A\|_{L^{4/3}(0,T; W^{1,4}(\Omega)')} + \|\pa_t\rho_B\|_{L^{4/3}(0,T; W^{1,4}(\Omega)')} \leq C_T . \label{est.rhot}
	\end{align}
	where $C_T$ is a constant depending on $T>0$.
\end{lem}
\begin{proof}
Given any test function $\psi$, bound \eqref{est.flux} yields
\begin{align*}
	\langle\pa_t\rho_A , \psi\rangle &= -\int_0^T\int_\Omega \na\psi\cdot\left( \na\rho_A + \rho_A\na\rho_B \right) \mathrm{d}x  \mathrm{d}t\leq
	C\|\na\psi\|_{L^4(\Omega_T)},
\end{align*}
which means that $\pa_t\rho_A$ is bounded in $L^{4/3}(0,T; W^{1,4}(\Omega)')$. In the same way one proves the same bound for $\pa_t\rho_B$ and obtain \eqref{est.rhot}.
\end{proof}
\subsubsection{Additional estimates from Maxwell-Boltzmann energy balance equation \eqref{ei.B}.}
In this subsection we assume that the solution $\rho$ to \eqref{E:system}-\eqref{E:1_BC} satisfies $\rho\in\cD$ a.e.~in $\Omega_T$. This means that \eqref{ei.B} holds. On the other hand, we wish to point out that $\rho$ fulfills also \eqref{ei.A}, which implies that the estimates derived in the previous subsection are additionaly satisfied.
\begin{lem} We have the following estimates on $\rho$:
	\begin{align}
	\label{est.1}
 \|(1-\sqrt{\rho_A\rho_B})^{1/2}\nabla\sqrt{\rho_i}\|_{L^2(0,T; L^2(\Omega))} & \leq C_T,\qquad i\in \{A,B\}, \\
	\label{est.na.p.2}
	\|(1-\sqrt{\rho_A\rho_B})^{3/2}\|_{L^{4/3}(0,T; W^{1,4/3}(\Omega))} &\leq \tilde{C}_T, 
	\end{align} where $C_T, \tilde{C}_T >0$ are some constants depending on $T>0$.
%
	
\end{lem}
\begin{proof}
Integrating \eqref{ei.B} in the time interval $[0,T]$ leads to the following
		\begin{align}\label{ei.B.int}
		\mathcal H_{MB}[\rho(T)] + 4\int_{\Omega_T} ( \sqrt{\rho_A\rho_B}|\na(\sqrt{\rho_A}+\sqrt{\rho_B})|^2 + (1-\sqrt{\rho_A\rho_B})(|\na\sqrt{\rho_A}|^2 + |\na\sqrt{\rho_B}|^2)) \mathrm{d}x \mathrm{d}t = \mathcal H_{MB}[\rho^{in}] ,
		\end{align} and thus \eqref{est.1}.
	From the identity
	\begin{align*}
		-\frac{2}{3}\nabla [(1-\sqrt{\rho_A\rho_B})^{3/2}] &= 
		(1-\sqrt{\rho_A\rho_B})^{1/2}\nabla\sqrt{\rho_A\rho_B} \\
		&=\sqrt{\rho_B} (1-\sqrt{\rho_A\rho_B})^{1/2}\nabla\sqrt{\rho_A} + \sqrt{\rho_A} (1-\sqrt{\rho_A\rho_B})^{1/2}\nabla\sqrt{\rho_B} 
	\end{align*}
	we deduce via H\"older inequality that
	\begin{align*}
		&\frac{2}{3} \|\nabla [(1-\sqrt{\rho_A\rho_B})^{3/2}]\|_{L^{4/3}(\Omega_T)} \\
		&\leq \|\sqrt{\rho_B}\|_{L^4(\Omega_T)} \|(1-\sqrt{\rho_A\rho_B})^{\frac{1}{2}}\nabla\sqrt{\rho_A}\|_{L^2(\Omega_T)} + \|\sqrt{\rho_A}\|_{L^4(\Omega_T)} \|(1-\sqrt{\rho_A\rho_B})^{\frac{1}{2}}\nabla\sqrt{\rho_B}\|_{L^2(\Omega_T)}
	\end{align*}
	so \eqref{est.sum.L4}, \eqref{est.1} imply that 
	$$
	\|\nabla [(1-\sqrt{\rho_A\rho_B})^{3/2}] \|_{L^{4/3}(\Omega_T)}\leq C_T ,
	$$
		From the above estimate and the bound $0 \leq (1- \sqrt{\rho_A \rho_B})^{3/2} \leq 1$ coming from the assumption $\rho_A \rho_B \leq 1$  a.e.~in $\Omega$ we obtain via Poincaré's Lemma \eqref{est.na.p.2}.
\end{proof}

\begin{remark}
	Estimate \eqref{est.1} is an improvement of \eqref{est.1b} (since the gradient of $\sqrt{\rho_i}$ is less degenerate in the region $\{\rho_A\rho_B=1\}$). Similarly, estimate \eqref{est.na.p.2} is better than \eqref{est.na.p} as the bounds for $\nabla\sqrt{\rho_A}$, $\nabla\sqrt{\rho_B}$ are less degenerate.
\end{remark}

\section{Existence analysis}\label{S:existence}
In this section we provide results on local-in time existence of strong solutions, define the notion of weak solutions and perform a weak stability analysis.

We consider the scaled equations with homogeneous Neumann boundary conditions ~in a bounded, open $\Omega\subset\R^2$ with Lipschitz boundary:
\begin{align}
\begin{cases}\label{eq1}
\pa_t\rho_A = \dv{\na\rho_A + \rho_A\na\rho_B},\quad &\mbox{in }\Omega\times (0,\infty), \\
\rho_A(0) =\rho_A^{in},\quad  &\mbox{in }\Omega.
\end{cases}\\
\begin{cases}\label{eq2}
\pa_t\rho_B = \dv{\na\rho_B + \rho_B\na\rho_A},\quad &\mbox{in }\Omega\times (0,\infty), \\
\rho_B(0)=\rho_B^{in},\quad  &\mbox{in }\Omega,
\end{cases} 
\end{align} with 
\begin{align} \label{bc}
\pa_\nu \rho_A =\pa_\nu \rho_B=0\qquad \mbox{on }\pa\Omega\times (0,\infty).
\end{align}

An analytical study of \eqref{eq1}--\eqref{bc} is the content of the next part.
\subsection{Local-in-time existence of strong solutions.}
The diffusion matrix of \eqref{eq1} and \eqref{eq2}  is given by
$$
A(\rho) = \begin{pmatrix}
1 & \rho_A\\ \rho_B & 1
\end{pmatrix},\qquad \rho_A, \rho_B\geq 0,
$$
has eigenvalues $\lambda_\pm(\rho) = 1\pm\sqrt{\rho_A\rho_B}$. 
Therefore, $\lambda_+(\rho)\geq\lambda_-(\rho)>0$ for $\rho = (\rho_A,\rho_B)\in \mathcal{D}$, where
$$
\mathcal{D}= \left\{ \rho \in\R^2_+ \, \mid \, \rho_A\rho_B < 1 \right\}.
$$
Applying \cite[Thr.~3.1]{J17} to \eqref{eq1}, \eqref{eq2} yields the following
\begin{lem}[Local-in-time existence] Let $(\rho_A^{in}, \rho_B^{in})\in W^{1,p}(\Omega; \R^2)$ for some $p>2$. Assume that there exists $\epsilon_0>0$ such that
$$
\min \left \{ \rho_A^{in}(x), \rho_B^{in}(x), 1-\rho_A^{in}(x)\rho_B^{in}(x) \right \}
\geq \epsilon_0\quad\mbox{a.e.~}x\in\Omega .
$$
Then there exists a unique maximal solution $\rho$ to \eqref{eq1}--\eqref{bc} satisfying
$\rho\in C^0([0,T^*), W^{1,p}(\Omega; \R^2))\cap C^\infty( (0,T^*); \R^2)$, with $0<T^*\leq\infty$, and there exists $\epsilon_1>0$ such that
$$
\min \left \{ \rho_A(t,x), \rho_B(t,x), 1-\rho_A(t,x)\rho_B(t,x) \right \}
\geq \epsilon_1\quad x\in\Omega ,~~t\in (0,T^*) .
$$
\end{lem}
This means that the solution exists as long as its value remins far away from the border of the region $\mathcal{D}$. Unfortunately, it is not clear how to guarantee such property for arbitrary large times.

\subsection{Weak solutions}

We first give a definition of a weak solution to \eqref{eq1}--\eqref{bc}. Let us first define the classes of functions
\begin{align}
	\label{class.X}
	\begin{split}
		X = \{ f\in C^1(\R^2_+):~~\exists C>0:~~ 
		&|f(u_A,u_B)|\leq C(1+u_A+u_B)^2, \\
		&|Df(u_A,u_B)|\leq C|u_A u_B-1|,~~\forall u_A,u_B\geq 0\},
	\end{split}
\end{align}
\begin{align}
	\label{class.Y}
	\begin{split}
		Y = \{
		\Psi\in C^1(\R_+):~~\exists C>0:~~ |\Psi(s)|\leq C\min\{|s-1|,1\},
		~~ |\Psi'(s)|\leq C,~~\forall s\geq 0\},
	\end{split}
\end{align}
\begin{align}
	\label{class.Z}
	\begin{split}
		Z = \{
		\Phi\in C^1(\R_+):~~\exists C>0:~~ |\Phi(s)|\leq Cs,
		~~ |\Phi'(s)|\leq C\min\{s,|s-1|,1\},~~\forall s\geq 0\}.
	\end{split}
\end{align}


\begin{dfn}[Weak solution]\label{def.weaksol}
A Lebesgue-measurable function $\rho : \OmT\to\R^2_+$ is called a 
weak solution to \eqref{eq1}--\eqref{bc} if (and only if) the following properties are satisfied.
\begin{enumerate}[label=(\roman*)]
	\item It has the regularity:
	\begin{align}
		\label{reg.1}
		\rho_A , \rho_B, \text{ and } \rho_A\rho_B \in L^\infty(0,T; L^1(\Omega)),\quad
		\pa_t\rho_A , \pa_t\rho_B\in 
		L^{4/3}(0,T; W^{1,4}(\Omega)'), 
	\end{align}
	\begin{align}
		\label{reg.2}
		\sqrt{\rho_A} + \sqrt{\rho_B}\in L^2(0,T; H^1(\Omega)),\quad
		(1 + \rrhoAB)\nabla[\sqrt{\rho_A} + \sqrt{\rho_B}]\in L^2(\Omega_T),
	\end{align}
	\begin{align}
		\label{reg.3}
		\begin{split}
			f(\rrhoA , \rrhoB) \in &L^2(0,T; H^{1}(\Omega)),~~ \text{ for all } f\in X.
		\end{split}
	\end{align}

\item The following weak formulation of \eqref{eq1}--\eqref{bc} holds for all $T>0$:
\begin{align}
\label{w.A}
\int_0^T\langle\pa_t\rho_A , \phi\rangle \mathrm{d}t + 2\int_0^T\int_\Omega\na\phi\cdot
\sqrt{\rho_A}\cdot\zeta_A  \dx\dt= 0\quad \text{for all } \phi\in L^4(0,T; W^{1,4}(\Omega)),\\
\label{w.B}
\int_0^T\langle\pa_t\rho_B , \phi\rangle \d t 
+ 2\int_0^T\int_\Omega\na\phi\cdot
\sqrt{\rho_B}\cdot\zeta_B \dx\dt = 0\quad\text{for all }\phi\in L^4(0,T; W^{1,4}(\Omega)),
\end{align}
\begin{align}
\label{w.ic}
\rho_A(t)\to \rho_A^{in},\quad
\rho_B(t)\to \rho_B^{in} \quad \mbox{ strongly in }
W^{1,4}(\Omega)' \mbox{ as }t\to 0, 
\end{align}
where the quantities $\zeta_A, \zeta_B\in L^2(\Omega\times (0,T))$ are identified by the relations
\begin{align}
	\label{zeta.1}
	\zeta_A + \zeta_B = (1+\sqrt{\rho_A\rho_B})\nabla(\sqrt{\rho_A}+\sqrt{\rho_B})\quad
	\mbox{a.e.~in }\Omega_T,
\end{align}
\begin{align}
	\label{zeta.2}
	\begin{split}
		\int_{\Omega_T} & \Psi(\sqrt{\rho_A\rho_B}) ( \zeta_A - \zeta_B )\cdot\phi \dx \d t \\
		&=  -\int_{\Omega_T}
		\left[\Psi(\sqrt{\rho_A\rho_B}) (1-\sqrt{\rho_A\rho_B}) ( \sqrt{\rho_A} - \sqrt{\rho_B} ) \right]\dv{\phi} \dx \dt \\ 
		&- \int_{\Omega_T}( \sqrt{\rho_A} - \sqrt{\rho_B} )\nabla\left[ 
		\Psi(\sqrt{\rho_A\rho_B}) (1-\sqrt{\rho_A\rho_B}) \right]\cdot\phi \dx \dt ,
	\end{split}
\end{align}
for every $\phi\in C^1_c(\Omega_T; \R^2)$ and $\Psi \in Y$,
\begin{align}
	\label{zeta.3}
	\begin{split}
		\int_{\Omega_T} &
		2 \Phi(\rrhoAB)\left(
		\frac{\zeta_A}{\rrhoA} - \frac{\zeta_B}{\rrhoB}
		\right)\cdot\phi \dx \dt\\ 
		=&
		-\int_{\Omega_T}\left[ \Phi(\rrhoAB)(\log(\rho_A/\rho_B) + \rho_B-\rho_A) \right]\dv{\phi}\dx \dt\\
		&- \int_{\Omega_T}\frac{1}{2}(\log(\rho_A/\rho_B) + \rho_B-\rho_A)
		\frac{\Phi'(\rrhoAB)}{\rrhoAB - 1}\nabla [(\rrhoAB - 1)^2]\cdot\phi \dx \dt,
	\end{split}
\end{align}
for every $\phi\in C^1_c(\Omega_T)$, $\Phi\in Z$.
\item The mass of each species is conserved:
\begin{align}
\label{mass}
\int_\Omega\rho_A(t) \mathrm{d}x = \int_\Omega\rho_A^{in} \mathrm{d} x ,\quad
\int_\Omega\rho_B(t) \mathrm{d}x = \int_\Omega\rho_B^{in}  \mathrm{d}x ,\quad
t>0,
\end{align}
and the {\em integrated energy balance} is satisfied:
\begin{align}\label{ieb.zeta}
\mathcal{H}[\rho(T)] + 4\int_{\Omega_T}( |\zeta_A|^2 + |\zeta_B|^2 ) \, \mathrm{d}x  \mathrm{d}t \leq \mathcal{H}[\rho^{in}], \quad \text{ for all } T>0.
\end{align}
\end{enumerate}
\end{dfn}

\begin{remark}
	A more standard weak formulation of \eqref{eq1}--\eqref{bc} would simply require 
	$\zeta_A = \nabla\rrhoA + \rrhoAB \, \nabla\rrhoB$, 
	$\zeta_B = \nabla\rrhoB + \rrhoAB \,\nabla\rrhoA$ a.e.~in $\Omega_T$, or equivalently $\zeta_A + \zeta_B = (1+\rrhoAB)\nabla(\rrhoA+\rrhoB)$, $\zeta_A - \zeta_B = (1-\rrhoAB)\nabla(\rrhoA-\rrhoB)$ a.e.~in $\Omega_T$.
	However, while the equality $\zeta_A + \zeta_B = (1+\rrhoAB)\nabla(\rrhoA+\rrhoB)$ in \eqref{zeta.1} makes sense as an identity between $L^1_{loc}(\Omega_T)$ functions because $\rrhoA+\rrhoB\in L^2(0,T; H^1(\Omega))$, the (formal) relation
	$\zeta_A - \zeta_B = (1-\rrhoAB)\nabla(\rrhoA-\rrhoB)$
	cannot be intended as an identity between $L^1_{loc}(\Omega_T)$ functions since the distributional gradient of $\rrhoA-\rrhoB$ is not in $L^1_{loc}(\Omega_T)$. This is due to the degenerate factor $(1-\rrhoAB)$ which prevents us from deriving a bound for $\nabla(\rrhoA-\rrhoB)$ in the region $\{\rrhoAB=1\}$. Def.~\ref{def.weaksol} provides a workaround to this issue by stating in \eqref{zeta.1}--\eqref{zeta.3} a ``renormalized'' formulation of the identities
	$\zeta_A = \nabla\rrhoA + \rrhoAB \, \nabla\rrhoB$, 
	$\zeta_B = \nabla\rrhoB + \rrhoAB \,\nabla\rrhoA$.
	It is straightforward to see (via a density argument) that if $\nabla(\rrhoA-\rrhoB)\in L^1_{loc}(\Omega_T)$ then \eqref{zeta.1}--\eqref{zeta.3} yield 
	$\zeta_A = \nabla\rrhoA + \rrhoAB \, \nabla\rrhoB$, 
	$\zeta_B = \nabla\rrhoB + \rrhoAB \,\nabla\rrhoA$ a.e.~in $\Omega_T$.
\end{remark}
%


\subsection{Weak stability analysis} \label{sec:weak_stab}
In this section, we prove the following result. 
\begin{lem}[Weak stability]\label{lem.compact}
Let $\rho^{in} = (\rho_{A}^{in},\rho_{B}^{in}): \Omega\to\R^2_+$
such that $\rho_A^{in}$, $\rho_B^{in}\in L^2(\Omega)$.
Moreover, let $\rho^n = (\rho_A^n,\rho_B^n)$ be a sequence of 
weak solutions to \eqref{eq1}--\eqref{bc} having $\rho^{in}$ as initial datum according to Definition \ref{def.weaksol}. Assume furthermore that $\sqrt{\rho_A^n}, \sqrt{\rho_B^n}\in L^2(0,T; H^1(\Omega))$ for every $n\in\NN$. 
Then $\rho^n$ converges (up to subsequences) strongly in $L^1(\OmT)$ for every $T>0$ to a weak solution $\rho=(\rho_A,\rho_B) : \Omega\times (0,\infty)\to \R^2_+$ to \eqref{eq1}--\eqref{bc} in the sense of Definition \ref{def.weaksol}.
\end{lem}
\begin{remark}
The Lemma implies that the weak solutions described in Definition \ref{def.weaksol} are limit points of standard, {\em nondegenerate} weak solutions to the system: notice that the assumption
$\sqrt{\rho_A^n}, \sqrt{\rho_B^n}\in L^2(0,T; H^1(\Omega))$ for every $n\in\NN$, which is not true in general for weak solutions as for Def.~\ref{def.weaksol}. 
\end{remark}
\begin{remark} Another perspective into Lemma \ref{lem.compact} is the notion that the entropy structure of the system is robust, that is, the estimates provided by the entropy balance inequality are sufficient to show compactness of a suitable sequence of approximated solutions and prove that a limit point of such approximating sequence is a weak solution to the system in the sense of Def.~\ref{def.weaksol}. Unfortunately, no approximating sequence with the required regularity is known, this is why the global-in-time existence of weak solutions to the system is an open problem.  
\end{remark}
\begin{proof}({\bf Proof of Lemma \ref{lem.compact}})
By assumption, for every $T>0$, the approximate solution $\rho^n$ satisfies 
\begin{align}
\label{w.A.n}
\int_0^T\langle\pa_t\rho_A^n , \phi\rangle \dt
+ 2\int_0^T\int_\Omega\na\phi\cdot
\sqrt{\rho_A^n}\, (\na\sqrt{\rho_A^n} + \sqrt{\rho_A^n\rho_B^n}\,\na\sqrt{\rho_B^n}) \dx \dt &= 0, \text{ for all }\phi\in C^1_c(\Omega_T),\\
\label{w.B.n}
\int_0^T\langle\pa_t\rho_B^n , \phi\rangle \mathrm{d}t 
+ 2\int_0^T\int_\Omega\na\phi\cdot
\sqrt{\rho_B^n}\, (\na\sqrt{\rho_B^n} + \sqrt{\rho_A^n\rho_B^n}\,\na\sqrt{\rho_A^n}) \dx \dt&= 0,\text{ for all }\phi\in C^1_c(\Omega_T),
\end{align}
\begin{align}
\label{w.ic.n}
\rho_A^n(t)\to \rho_A^{in},\quad
\rho_B^n(t)\to \rho_B^{in}\quad \mbox{ strongly in } W^{-1,4/3}(\Omega)\mbox{ as }t\to 0,
\end{align}
as well as \eqref{ei.A.int}, and therefore also \eqref{est.1a}--\eqref{est.rhot}.\medskip\\
{\bf Notation.} Given a sequence $f_n$ in some Banach space $X$, that is weakly (or weak-*) convergent in $X$, we denote with $\overline{f_n}$ the weak (or weak-*) limit of $f_n$. Moreover, we define $\R_+ = [0,\infty)$.\medskip\\
The proof is divided into four steps.\medskip\\
{\bf Step 1: strong convergence of $\rho_A^n + \rho_B^n$.}\\
Let $f\in W^{1,\infty}(\R_+,\R_+)$ function such that 
$|f'(x)|\leq C|1-x|$ for $x\geq 0$.
Let us define the vector fields
\begin{align}\label{Un}
\begin{split}
U_A^n &= (\rho_A^n , -\nabla\rho_A^n - \rho_A^n\nabla\rho_B^n), \\
U_B^n &= (\rho_B^n , -\nabla\rho_B^n - \rho_B^n\nabla\rho_A^n),\\
V^n &= (f(\sqrt{\rho_A^n\rho_B^n}), 0,0,0).
\end{split}
\end{align}
From \eqref{est.sum.L4}, \eqref{est.flux} we deduce that for $i \in \{A,B\}$, $U_i^n$ is bounded in $L^{4/3}(\Omega_T)$, while \eqref{E:system}-\eqref{E:1_BC} means that $\divergence_{(t,x)}U_i^n = 0$ (a fortiori $\divergence_{(t,x)}U_i^n$ is relatively compact in $W^{-1,r}(\Omega_T)$ for every $r>1$).
On the other hand $V^n$ is bounded in $L^\infty(\Omega_T)$ and the antisymmetric part $\mbox{curl}_{(t,x)}V^n$ of its Jacobian can be estimated as
\begin{align*}
|\mbox{curl}_{(t,x)}V^n| \leq C |\na f(\sqrt{\rho_A^n\rho_B^n})| 
&\leq
C  |f'(\sqrt{\rho_A^n\rho_B^n})| (\sqrt{\rho_A^n}|\na\sqrt{\rho_B^n}|
+\sqrt{\rho_B^n}|\na\sqrt{\rho_A^n}|)\\
&\leq C(\sqrt{\rho_A^n} |\sqrt{\rho_A^n\rho_B^n}-1| |\na\sqrt{\rho_B^n}|
+\sqrt{\rho_B}|\sqrt{\rho_A^n\rho_B^n}-1| |\na\sqrt{\rho_A^n}|),
\end{align*}
which means, thanks to \eqref{est.1b}, \eqref{est.sum.L4}, that $\mbox{curl}_{(t,x)}V^n$ is bounded in $L^{4/3}(\Omega_T)$, and a fortiori relatively compact in $W^{-1,r}(\Omega_T)$ for some $r>1$.
Therefore, the Div-Curl Lemma [Theorem 10.21 in \cite{FN17}] implies that
\begin{align*}
\overline{U^n\cdot V^n} = \overline{U^n}\cdot\overline{V^n} \quad\mbox{a.e.~in }\Omega_T.
\end{align*}
Therefore, for $i \in  \{A,B\}$ and for every $f\in W^{1,\infty}(\R_+,\R_+)$, 
\begin{align}
	\label{wl.1}
	\begin{split}
\overline{\rho_i^n f(\sqrt{\rho_A^n\rho_B^n})} = 
\overline{\rho_i^n}\,
\overline{f(\sqrt{\rho_A^n\rho_B^n})},\mbox{ a.e.~in }\Omega_T, \text{ such that } |f'(x)|\leq C|1-x| \text{ for } x\geq 0.
\end{split}
\end{align}
Let us consider \eqref{wl.1} with $f(x) = \min\{k, F(x)\}$ where $k\geq 1$ is arbitrary and $F\in L^1_{loc}(\R_+,\R_+)$ such that $F'\in L^\infty(\R_+)$, $F(x)\leq C(1+x)$, 
$ |F'(x)|\leq C|1-x|$ for $x\geq 0$.
Let us now estimate the quantity
\begin{multline*}
\| \overline{\rho_i^n ( \min\{k, F(\sqrt{\rho_A^n\rho_B^n})\} - F(\sqrt{\rho_A^n\rho_B^n}) )}\|_{L^1(\Omega_T)} 
\\ \leq \liminf_{n\to\infty}\|\rho_i^n ( \min\{k, F(\sqrt{\rho_A^n\rho_B^n})\} - F(\sqrt{\rho_A^n\rho_B^n}) )\|_{L^1(\Omega_T)}\\
\leq \sup_{n\in\NN}\|\rho_i^n\|_{L^2(\Omega_T)}
\|\min\{k, F(\sqrt{\rho_A^n\rho_B^n})\} - F(\sqrt{\rho_A^n\rho_B^n})\|_{L^2(\Omega_T)},
\end{multline*}
where we used Fatou's Lemma in the first inequality. From \eqref{est.sum.L4} it follows
\begin{align*}
\|  \overline{\rho_i^n ( \min\{k, F(\sqrt{\rho_A^n\rho_B^n})\} - F(\sqrt{\rho_A^n\rho_B^n}) )}\|_{L^1(\Omega_T)}^2
&\leq C \sup_{n\in\NN}\int_{\Omega_T\cap \{F(\sqrt{\rho_A^n\rho_B^n})>k\} }\left|
F(\sqrt{\rho_A^n\rho_B^n})\right|^2 \mathrm{d}x \mathrm{d}t\\
&\leq \frac{C}{k} \sup_{n\in\NN}\int_{\Omega_T\cap \{F(\sqrt{\rho_A^n\rho_B^n})>k\} }\left|
F(\sqrt{\rho_A^n\rho_B^n})\right|^3 \mathrm{d}x \mathrm{d}t\\
&\leq \frac{C}{k} \sup_{n\in\NN}\int_{\Omega_T}
\left(1 + \sqrt{\rho_A^n\rho_B^n}\right)^3 \mathrm{d}x \mathrm{d}t ,
\end{align*}
which, thanks to \eqref{est.Hm1.b}, implies
\begin{align*}
\| & \overline{\rho_i^n ( \min\{k, F(\sqrt{\rho_A^n\rho_B^n})\} - F(\sqrt{\rho_A^n\rho_B^n}) )}\|_{L^1(\Omega_T)}\leq 
\frac{C}{\sqrt k} ,\quad i \in \{A,B\},\quad k\geq 1.
\end{align*}
In a similar way one shows that 
\begin{align*}
	\| & \overline{\rho_i^n}\, \overline{( \min\{k, F(\sqrt{\rho_A^n\rho_B^n})\} - F(\sqrt{\rho_A^n\rho_B^n}) )}\|_{L^1(\Omega_T)}\leq 
	\frac{C}{\sqrt k} ,\quad i\in \{A,B\},\quad k\geq 1.
\end{align*}
From the above inequalities and \eqref{wl.1} we deduce
\begin{multline*}
\|\overline{\rho_i^n F(\sqrt{\rho_A^n\rho_B^n})} - 
\overline{\rho_i^n }\,
\overline{F(\sqrt{\rho_A^n\rho_B^n})}\|_{L^1(\Omega_T)}\\
\leq \| \overline{\rho_i^n ( \min\{k, F(\sqrt{\rho_A^n\rho_B^n})\} - F(\sqrt{\rho_A^n\rho_B^n}) )}\|_{L^1(\Omega_T)} + 
\| \overline{\rho_i^n}\, \overline{( \min\{k, F(\sqrt{\rho_A^n\rho_B^n})\} - F(\sqrt{\rho_A^n\rho_B^n}) )}\|_{L^1(\Omega_T)}\\
\leq \frac{C}{\sqrt k},\quad i \in \{A,B\},\quad k\geq 1,
\end{multline*}
implying that for $i \in \{A,B\}$,
\begin{align}
	\label{wl.2}
	\overline{\rho_i^n F(\sqrt{\rho_A^n\rho_B^n})} = 
	\overline{\rho_i^n}\,
	\overline{F(\sqrt{\rho_A^n\rho_B^n})}, \quad \mbox{ a.e.~in }\Omega_T,
	\end{align} and for every such $F\in L^1_{loc}(\R_+,\R_+)$,
\begin{align} \label{wl.2.2}
	F'\in L^\infty(\R_+), \, F(x) \leq C(1+x),  \, 
		|F'(x)|\leq C|1-x| \text{ for } x\geq 0.
\end{align}
Let us now choose $F=F_\delta$ in \eqref{wl.2}-\eqref{wl.2.2}, where $0<\delta<1$ and
\begin{align*}
F_\delta(s) = \begin{cases}
0, & s\leq 1,\\
(s-1)^2, & 1<s\leq 1+\delta,\\
s + \delta^2 -1-\delta, & s>1+\delta.
\end{cases}
\end{align*}
Let us estimate (similar idea as before)
\begin{align*}
\|\overline{\rho_i^n(F_\delta(\sqrt{\rho_A^n\rho_B^n})-(\sqrt{\rho_A^n\rho_B^n}-1)_+)}\|_{L^1(\Omega_T)}^2
& \leq \sup_{n\in\NN}\|\rho_i^n\|_{L^2(\Omega_T)}^2
\|F_\delta(\sqrt{\rho_A^n\rho_B^n})-(\sqrt{\rho_A^n\rho_B^n}-1)_+\|^2_{L^2(\Omega_T)}\\
&\leq C\delta ,
\end{align*}
where the last step comes from the fact that 
$|F_\delta(s)-(s-1)_+|\leq C\delta$ for every $s\geq 0$. We can deduce that \eqref{wl.2}-\eqref{wl.2.2} holds with $F(s)=(s-1)_+$, that is
\begin{equation}
	\label{wl.3a}
\overline{\rho_i^n (\sqrt{\rho_A^n\rho_B^n}-1)_+} = 
	\overline{\rho_i^n}\,
	\overline{(\sqrt{\rho_A^n\rho_B^n}-1)_+},\quad i\in \{A,B\}\quad\mbox{a.e.~in }\Omega_T .
\end{equation}
In a similar way, by writing \eqref{wl.2}-\eqref{wl.2.2} with $F=G_\delta$, $0<\delta<1$,
\begin{align*}
G_\delta = \begin{cases}
s - \delta^2 - 1 + \delta, & 0\leq s \leq 1-\delta,\\
-(s-1)^2, & 1-\delta < s \leq 1,\\
0, & s>1.
\end{cases}
\end{align*}
One deduces that
\begin{equation}
	\label{wl.3b}
	\overline{\rho_i^n (\sqrt{\rho_A^n\rho_B^n}-1)_-} = 
	\overline{\rho_i^n}\,
	\overline{(\sqrt{\rho_A^n\rho_B^n}-1)_-},\quad i \in \{A,B\},\quad\mbox{a.e.~in }\Omega_T .
\end{equation}
Summing \eqref{wl.3a} and \eqref{wl.3b} allows us to conclude
\begin{equation}
	\label{wl.4}
	\overline{\rho_i^n \sqrt{\rho_A^n\rho_B^n}} = 
	\overline{\rho_i^n}\,
	\overline{\sqrt{\rho_A^n\rho_B^n}},\quad i\in \{A,B\},\quad\mbox{a.e.~in }\Omega_T .
\end{equation}
Let $k\in\NN$ arbitrary. Define the vector field 
$$
Z^n = (\min\{(\sqrt{\rho_A^n}+\sqrt{\rho_B^n})^2, k^2\},0,0,0),\quad \mbox{ for } n \in\NN .
$$
Clearly $Z^n$ is bounded in $L^\infty(\Omega_T)$, while 
$$
|\mbox{curl}_{(t,x)} Z^n|\leq C |\nabla [\min\{(\sqrt{\rho_A^n}+\sqrt{\rho_B^n})^2,k^2\}]|\leq C k
|\nabla(\sqrt{\rho_A^n}+\sqrt{\rho_B^n})| ,
$$
so thanks to \eqref{est.1a} $\mbox{curl}_{(t,x)} Z^n$ is bounded in $L^2(\Omega_T)$ and therefore relatively compact in $W^{-1,r}(\Omega_T)$
for some $r>1$. The Div-Curl Lemma allows us once again to deduce
\begin{align*}
\overline{U_i^n\cdot Z^n} = 
\overline{U_i^n}\cdot \overline{Z^n}\quad i \in \{A,B\},\quad\mbox{a.e.~in }\Omega_T ,
\end{align*}
which is equivalent to
\begin{align}\label{wl2.1}
\overline{\rho_i^n\min\{(\sqrt{\rho_A^n}+\sqrt{\rho_B^n})^2, k^2\}} = 
\overline{\rho_i^n}
\,\overline{\min\{(\sqrt{\rho_A^n}+\sqrt{\rho_B^n})^2, k^2\}}
\quad i \in \{A,B\},\quad\mbox{a.e.~in }\Omega_T .
\end{align}
Let us define $v^{n,k} = \sqrt{\rho_A^n}+\sqrt{\rho_B^n}
-\min\{\sqrt{\rho_A^n}+\sqrt{\rho_B^n}, k\}
= ( \sqrt{\rho_A^n}+\sqrt{\rho_B^n} - k )_+ $. 
For $t\in [0,T]$ and $k\geq 2$ let us estimate
\begin{align*}
\|v^{n,k}(t)\|_{L^2(\Omega)}^2 &= \int_\Omega \left(\sqrt{\rho_A^n(t)}+\sqrt{\rho_B^n(t)} - k \right)_+^2 \d x\\
& \leq \int_{\Omega\cap\{ \sqrt{\rho_A^n(t)}+\sqrt{\rho_B^n(t)} 
	> k \}  }
\left(\sqrt{\rho_A^n(t)}+\sqrt{\rho_B^n(t)}\right)^2 \d x\\
&\leq \frac{1}{\log k}\int_{\Omega\cap\{ \sqrt{\rho_A^n(t)}+\sqrt{\rho_B^n(t)} > k \}  }
\left(\sqrt{\rho_A^n(t)}+\sqrt{\rho_B^n(t)}\right)^2
\log\left(\sqrt{\rho_A^n(t)}+\sqrt{\rho_B^n(t)}\right) \d x\\
&\leq \frac{C}{\log k}\int_\Omega (\rho_A^n(t) + \rho_B^n(t))(1 + \log(\rho_A^n(t) + \rho_B^n(t)) ) \d x .
\end{align*} where the last inequality comes from the elementary property $(x+y)^2\leq 2(x^2 + y^2)$ for $x, y\in\R$, since
\begin{align*}
	(x+y)^2\log(x+y)\leq 2(x^2+y^2)\log(x+y) = (x^2+y^2)\log[(x+y)^2]
	\leq (x^2+y^2)\log[2(x^2+y^2)] .
\end{align*}
From \eqref{rhoL1Linf}, \eqref{est.L1logL} it follows
\begin{equation}
\label{est.v.1}
\|v^{n,k}\|_{L^\infty(0,T; L^2(\Omega))}\leq \frac{C}{\sqrt{\log k}},\quad n,k\in\NN,~~~ k\geq 2.
\end{equation}
From Gagliardo-Nirenberg inequality \eqref{GN} applied with $u=v^{n,k}$
it follows
\begin{align*}
\int_0^T\|v^{n,k}\|_{L^4(\Omega)}^4\mathrm{d}x \mathrm{d}t \leq 
C\left( \sup_{t\in [0,T]}\|v^{n,k}\|_{L^2(\Omega)}^2 \right)
\int_0^T\|v^{n,k}\|_{H^1(\Omega)}^2 \mathrm{d}x \mathrm{d}t ,
\end{align*}
which, thanks to \eqref{est.v.1}, leads to
\begin{align*}
\|v^{n,k}\|_{L^4(\Omega_T)}^4 &\leq \frac{C}{\log k}\|\sqrt{\rho_A^n}+\sqrt{\rho_B^n}\|_{L^2(0,T; H^1(\Omega))}^2 .
\end{align*}
Bound \eqref{est.na.sum} and the definition of $v^{n,k}$ allow us to deduce
\begin{align*}
\|\sqrt{\rho_A^n}+\sqrt{\rho_B^n}
-\min\{\sqrt{\rho_A^n}+\sqrt{\rho_B^n}, k\}\|_{L^4(\Omega_T)}\leq \frac{C}{(\log k)^{1/4}},\quad n,k\in\NN,\quad k\geq 2,
\end{align*}
which, together with Cauchy-Schwartz inequality
\begin{align*}
&\|(\sqrt{\rho_A^n}+\sqrt{\rho_B^n})^2
-(\min\{\sqrt{\rho_A^n}+\sqrt{\rho_B^n}, k\})^2\|_{L^2(\Omega_T)}\\
&\qquad\leq
\|\sqrt{\rho_A^n}+\sqrt{\rho_B^n}
-\min\{\sqrt{\rho_A^n}+\sqrt{\rho_B^n}, k\}\|_{L^4(\Omega_T)}
\|\sqrt{\rho_A^n}+\sqrt{\rho_B^n}
+\min\{\sqrt{\rho_A^n}+\sqrt{\rho_B^n}, k\}\|_{L^4(\Omega_T)}
\end{align*}
and \eqref{est.sum.L4}, allows us to conclude
\begin{align}\label{wl2.2}
&\|(\sqrt{\rho_A^n}+\sqrt{\rho_B^n})^2
-(\min\{\sqrt{\rho_A^n}+\sqrt{\rho_B^n}, k\})^2\|_{L^2(\Omega_T)}
\leq \frac{C}{(\log k)^{1/4}},\quad n,k\in\NN,\quad k\geq 2 .
\end{align}
Let $\mathcal{M}(\overline{\Omega_T}) = C(\overline{\Omega_T})'$ the space of Radon measures on $\overline{\Omega_T}$. Since $(\rho_A^n + \rho_B^n) (\sqrt{\rho_A^n}+\sqrt{\rho_B^n})^2$ is bounded in $L^1(\Omega_T)$ (thanks to \eqref{est.sum.L4}), then (up to subsequences) it is weak-* convergent in
$\mathcal{M}(\overline{\Omega_T})$. Since \eqref{wl2.1} holds, we can write
\begin{multline}  \label{wl2.est.1}
\| \overline{(\rho_A^n + \rho_B^n) (\sqrt{\rho_A^n}+\sqrt{\rho_B^n})^2}
- \overline{(\rho_A^n + \rho_B^n)}\,
\overline{(\sqrt{\rho_A^n}+\sqrt{\rho_B^n})^2}
 \|_{\mathcal{M}(\overline{\Omega_T})}\\
\leq 
\| \overline{(\rho_A^n + \rho_B^n)[  (\sqrt{\rho_A^n}+\sqrt{\rho_B^n})^2
	-(\min\{\sqrt{\rho_A^n}+\sqrt{\rho_B^n}, k\})^2]}\|_{\mathcal{M}(\overline{\Omega_T})} \\
 \qquad + \| \overline{(\rho_A^n + \rho_B^n)}\,\overline{[  (\sqrt{\rho_A^n}+\sqrt{\rho_B^n})^2
	-(\min\{\sqrt{\rho_A^n}+\sqrt{\rho_B^n}, k\})^2]}\|_{\mathcal{M}(\overline{\Omega_T})} =: J_1 + J_2 .
\end{multline}
Let us bound the terms $J_1$, $J_2$. Since the norm in $\mathcal M(\overline{\Omega_T})$ is weak-* lower semicontinuous, it follows
\begin{align*}
J_1 &\leq \liminf_{n\to\infty} 
\| (\rho_A^n + \rho_B^n)[  (\sqrt{\rho_A^n}+\sqrt{\rho_B^n})^2
	-(\min\{\sqrt{\rho_A^n}+\sqrt{\rho_B^n}, k\})^2]\|_{\mathcal{M}(\overline{\Omega_T})}\\
	&\leq\sup_{n\in\NN}
	\| (\rho_A^n + \rho_B^n)[  (\sqrt{\rho_A^n}+\sqrt{\rho_B^n})^2
	-(\min\{\sqrt{\rho_A^n}+\sqrt{\rho_B^n}, k\})^2]\|_{\mathcal{M}(\overline{\Omega_T})}.
\end{align*}
Since $L^1(\Omega_T)\hookrightarrow \mathcal{M}(\overline{\Omega_T})$,
it follows
\begin{align*}
J_1 &\leq C\sup_{n\in\NN}
\| (\rho_A^n + \rho_B^n)[(\sqrt{\rho_A^n}+\sqrt{\rho_B^n})^2
-(\min\{\sqrt{\rho_A^n}+\sqrt{\rho_B^n}, k\})^2]\|_{L^1(\Omega_T)}\\
&\leq C\sup_{n\in\NN}\|\rho_A^n + \rho_B^n\|_{L^2(\Omega_T)}
\|(\sqrt{\rho_A^n}+\sqrt{\rho_B^n})^2
-(\min\{\sqrt{\rho_A^n}+\sqrt{\rho_B^n}, k\})^2\|_{L^2(\Omega_T)} .
\end{align*}
Since \eqref{est.sum.L4}, \eqref{wl2.2} hold, we obtain
\begin{align*}
J_1 \leq \frac{C}{(\log k)^{1/4}},\quad n,k\in\NN,\quad k\geq 2 .
\end{align*}
In a similar way one can show
\begin{align*}
J_2 \leq \frac{C}{(\log k)^{1/4}},\quad n,k\in\NN,\quad k\geq 2 .
\end{align*}
From the previous two bounds and \eqref{wl2.est.1} one concludes
\begin{align*}
\| \overline{(\rho_A^n + \rho_B^n) (\sqrt{\rho_A^n}+\sqrt{\rho_B^n})^2}
- \overline{(\rho_A^n + \rho_B^n)}\,
\overline{(\sqrt{\rho_A^n}+\sqrt{\rho_B^n})^2}
\|_{\mathcal{M}(\overline{\Omega_T})}\leq \frac{C}{(\log k)^{1/4}},\quad k\geq 2,
\end{align*}
which means
\begin{align*}
\overline{(\rho_A^n + \rho_B^n) (\sqrt{\rho_A^n}+\sqrt{\rho_B^n})^2}
= \overline{(\rho_A^n + \rho_B^n)}\,
\overline{ (\sqrt{\rho_A^n}+\sqrt{\rho_B^n})^2}\quad
\mbox{in }\mathcal{M}(\overline{\Omega_T}).
\end{align*}
For every $f\in \mathcal{M}(\overline{\Omega_T})$, $\phi\in C(\overline{\Omega_T})$, let $\langle f, \phi\rangle$ be the dual product between $f$, $\phi$ (i.e.~$\langle f, \phi\rangle$ is the result of the application of the linear, bounded functional $f$ to $\phi$). It follows
\begin{align} \label{wl2.lim1}
\begin{split}
\lim_{n\to\infty}  \int_{\Omega_T}(\rho_A^n + \rho_B^n) (\sqrt{\rho_A^n}+\sqrt{\rho_B^n})^2 \mathrm{d}x \mathrm{d}t &= 
\langle \overline{(\rho_A^n + \rho_B^n) (\sqrt{\rho_A^n}+\sqrt{\rho_B^n})^2} , 1 \rangle\\
&=\int_{\Omega_T}\overline{(\rho_A^n + \rho_B^n)}\,
\overline{(\sqrt{\rho_A^n}+\sqrt{\rho_B^n})^2} \d x \d t.
\end{split}
\end{align}
On the other hand, summing \eqref{wl.4} in $i \in \{A,B\}$, multiplying it with $2$ and integrating it in $\Omega_T$ lead to
\begin{equation}
\label{wl2.lim2}
\lim_{n\to\infty} \int_{\Omega_T}(\rho_A^n + \rho_B^n) 
2\sqrt{\rho_A^n\rho_B^n} \, \mathrm{d}x \mathrm{d}t = \int_{\Omega_T}
\overline{(\rho_A^n+\rho_B^n)}\,
\overline{2\sqrt{\rho_A^n\rho_B^n}} \, \mathrm{d}x \mathrm{d} t .
\end{equation}
Taking the difference between \eqref{wl2.lim1} and \eqref{wl2.lim2} yields
\begin{equation*}
\lim_{n\to\infty} \int_{\Omega_T}(\rho_A^n + \rho_B^n)^2 \, \mathrm{d}x \mathrm{d}t = 
\int_{\Omega_T}\left(\,\overline{\rho_A^n + \rho_B^n}\,\right)^2 \, \mathrm{d}x \mathrm{d}t ,
\end{equation*}
which means (thanks to \cite[Thr.~10.20]{FN17} ) that $\rho_A^n + \rho_B^n$ is strongly convergent in $L^2(\Omega_T)$.\medskip\\
{\bf Step 2: strong convergence of $\rho_A^n\rho_B^n$.}
Let
\begin{align*}
\phi(s) = \begin{cases}
e^{1 + 1/(s^2 - 1)} & 0\leq s < 1\\
0 & s\geq 1
\end{cases} .
\end{align*}
For every $r>0$, $u\in \R^2_+$, let us define the function
\begin{align*}
f_{(r,u)}(\rho) = \phi\left( \frac{|\rho - u|}{r} \right),\quad\rho\in\R^2 .
\end{align*}
We define also
\begin{align*}
\Gamma_{cr} &= \{ (\rho_1,\rho_2)\in\R^2_+ \, \mid \, \rho_1\rho_2 = 1\},\\
\mathscr{F} &= \left\{ 
f_{(r,u)} \, \mid \,  r\in\Q\cap (0,\infty),~ u\in (\Q\cap [0,\infty))^2 , ~
\overline{B_r(u)}\cap \Gamma_{cr} = \emptyset
 \right\} .
\end{align*}
We point out that $\mathscr F$ is a countable family of $C^\infty_c(\R^2)$ functions whose gradient vanishes in a neighbourhood of $\Gamma_{cr}$. This fact, together with \eqref{est.1b}, easily imply that
\begin{equation}
\label{fru.L2H1}
\|f_{(r,u)}(\rho^n)\|_{L^2(0,T; H^1(\Omega))}\leq C(r,u,T) .
\end{equation}
Moreover,
\begin{align*}
\frac{1}{2}&\pa_t f_{(r,u)}(\rho^n) = 
\frac{1}{2}\frac{\pa f_{(r,u)}(\rho^n)}{\pa\rho_A}\pa_t\rho_A^n + 
\frac{1}{2}\frac{\pa f_{(r,u)}(\rho^n)}{\pa\rho_B}\pa_t\rho_B^n\\
&= \dv{
\frac{\pa f_{(r,u)}(\rho^n)}{\pa\rho_A}\sqrt{\rho_A^n}(
\nabla\sqrt{\rho_A^n} + \sqrt{\rho_A^n\rho_B^n}\nabla\sqrt{\rho_B^n}) + 
\frac{\pa f_{(r,u)}(\rho^n)}{\pa\rho_B}\sqrt{\rho_B^n}(
\nabla\sqrt{\rho_B^n} + \sqrt{\rho_A^n\rho_B^n}\nabla\sqrt{\rho_A^n})}\\
&-\nabla\frac{\pa f_{(r,u)}(\rho^n)}{\pa\rho_A}\cdot
\sqrt{\rho_A^n}(
\nabla\sqrt{\rho_A^n} + \sqrt{\rho_A^n\rho_B^n}\nabla\sqrt{\rho_B^n}) 
-\nabla\frac{\pa f_{(r,u)}(\rho^n)}{\pa\rho_B}\cdot
\sqrt{\rho_B^n}(
\nabla\sqrt{\rho_B^n} + \sqrt{\rho_A^n\rho_B^n}\nabla\sqrt{\rho_A^n})\\
&=: \dv{\mathcal{J}_{(r,u)}^n} + \Xi_{(r,u)}^n .
\end{align*}
Once again, the assumptions on $f_{(r,u)}$ as well as \eqref{est.1b}
imply that 
\begin{align*}
\|\mathcal{J}_{(r,u)}^n\|_{L^2(\Omega_T)} + \|\Xi_{(r,u)}^n\|_{L^1(\Omega_T)}\leq C(r,u,T) ,
\end{align*}
which leads to
\begin{align}
\label{frudt}
\|\pa_t f_{(r,u)}(\rho^n)\|_{L^1(0,T; W^{-1,1}(\Omega))}\leq C(r,u,T).
\end{align}
We are therefore allowed to apply Aubin-Lions Lemma (and the uniform $L^\infty(\Omega_T)$ bound for $f_{(r,u)}(\rho^n)$) to deduce the strong convergence of $f_{(r,u)}(\rho^n)$ in $L^q(\Omega_T)$ for every $q<\infty$. In particular,
\begin{align*}
\text{ For all } f_{(r,u)}\in\mathscr F ,  \text{ there exists } (\rho^{n_k(r,u)})_{k\in\NN}\subset (\rho^n)_{n\in\NN}  \text{ such that }
f_{(r,u)}(\rho^{n_k(r,u)})\mbox{ is a.e.~convergent in }\Omega_T .
\end{align*}
However, since $\mathscr F$ is countable, a Cantor diagonal argument allows us to find a subsequence (not relabeled) of $\rho^n$ such that
\begin{align*}
\text{ For all } f_{(r,u)}\in\mathscr F ,  \quad
f_{(r,u)}(\rho^{n})\to\xi_{(r,u)}\quad\mbox{ a.e.~in }\Omega_T .
\end{align*}
Let $(x,t)\in \Omega_T$ be a point where such convergence holds true.
Since (from Step 1) $\rho_A^n + \rho_B^n$ is (up to subsequences) strongly convergent in $L^2(\Omega_T)$, we can assume w.l.o.g.~that
$\rho^n(x,t)$ is bounded in $\R^2$.
There are two cases.\medskip\\
{\em \textbf{Case 1}: There exists $f_{(r_1,u_1)}\in\mathscr F$ such that  $\xi_{(r_1,u_1)}(x,t)>0$.} \\
Since $f_{(r,u)}(\rho) = \phi(|u-\rho|/r)$, and $\phi\mid_{[0,1]}$ is one-to-one and continuous, then $|\rho^n(x,t) - u_1|\to d_1\in [0,r)$.
By continuity there exist $f_{(r_2,u_2)}$, $f_{(r_3,u_3)}\in\mathscr F$ such that $u_1$, $u_2$, $u_3$ are not aligned and $\xi_{(r_2,u_2)}(x,t)$, $\xi_{(r_3,u_3)}(x,t)>0$. Since $f_{(r,u)}(\rho)$ is a one-to-one function of $|\rho-u|$ for $|\rho-u|\leq 1$, the limits 
$d_i = \lim_{n\to\infty}|\rho^n(x,t)-u_i|$, $i=1,2,3$, exist finite.
As a consequence, each accumulation point of $\rho^n(x,t)$ will fall at the intersection of three circles ($\pa B_{d_i}(u_i)$, $i=1,2,3$) with mutually not aligned centers, which can only consist of at most one point. This means that all accumulation points concide with this point, i.e.~the sequence $\rho^n(x,t)$ is convergent to that point. A fortiori $\rho_A^n(x,t)\rho_B^n(x,t)$ is also convergent.\medskip\\
{\em \textbf{Case 2}: For all  $f_{(r,u)}\in\mathscr F$ it holds $\xi_{(r,u)}(x,t)=0$.} 
\\Consider a generic subsequence $\rho^{n_m}(x,t)$ of $\rho^n(x,t)$.
Since it is bounded, is has a sub-subsequence $\rho^{n_{m_k}}(x,t)$ that is convergent to some limit $\ell\in\R^2_+$. However, since $f_{(r,u)}(\ell)=0$ for every $f_{(r,u)}\in\mathscr F$, the only possibility is that $\ell\in\Gamma_{cr}$. In particular
$\lim_{k\to\infty}\rho_A^{n_{m_k}}\rho_B^{n_{m_k}}(x,t) = 1$. 
The subsequence $\rho^{n_m}(x,t)$ being abitrary, this means that
$\rho_A^{n}\rho_B^{n}(x,t)\to 1$ as $n\to\infty$.\medskip\\
Summarizing up, we have proved that $\rho_A^n\rho_B^n$ is a.e.~convergent in $\Omega_T$. Bound \eqref{est.Hm1.b} implies that
$\rho_A^n\rho_B^n$ is strongly convergent in $L^{3/2-\eta}(\Omega_T)$
for every $\eta\in (0,1/2]$. Furthermore, we have also showed that 
$\rho_A^n\to\rho_A$, $\rho_B^n\to\rho_B$ a.e.~in $E$, where
$$ E = \left\{ (x,t)\in\Omega_T \, \mid \, \lim_{n\to\infty}\rho_A^n(x,t)\rho_B^n(x,t)\neq 1 \right\} . $$
This also implies (together with \eqref{est.sum.L4}) that 
$\rho_A^n\to\rho_A$, $\rho_B^n\to\rho_B$ strongly in $L^{2-\eta}(E)$, for every $\eta\in (0,1]$. We also point out that 
$|\sqrt{\rho_A^n} \pm \sqrt{\rho_B^n}| = \sqrt{\rho_A^n + \rho_B^n \pm 2\sqrt{\rho_A^n\rho_B^n}}$ is a.e.~convergent in $\Omega_T$, and so
is $|\rho_A^n - \rho_B^n| = |\sqrt{\rho_A^n} + \sqrt{\rho_B^n}||\sqrt{\rho_A^n} - \sqrt{\rho_B^n}|$.
\medskip\\
{\bf Step 3: strong convergence of $\rho_A^n$, $\rho_B^n$.}
Now we must prove that $\rho_A^n$, $\rho_B^n$ are a.e.~convergent also in $E^c=\big\{ \lim_{n\to\infty}\rho_A^n\rho_B^n = 1\big\}$. To this aim, let $\psi\in C^2(\R)$ be a cutoff such that 
\begin{equation}
\label{def.psi}
\psi(s) = 
\begin{cases}
0 & s<\frac{1}{3}, \\
1 & s>\frac{2}{3},\\
\mbox{nondecreasing} & \frac{1}{3}\leq s\leq\frac{2}{3}.
\end{cases}
\end{equation}
Moreover define 
\begin{align}\label{def.g}
g(s) = \frac{1}{1+s^4},\qquad s\geq 1,
\end{align}
\begin{align}
\label{def.fA}
f_A(\rho^n) &= \psi(\sqrt{\rho_A^n\rho_B^n})
g(\sqrt{\rho_A^n}+\sqrt{\rho_B^n})(\log\rho_A^n + \rho_B^n),\\
\label{def.fB}
f_B(\rho^n) &=\psi(\sqrt{\rho_A^n\rho_B^n})
g(\sqrt{\rho_A^n}+\sqrt{\rho_B^n})(\log\rho_B^n + \rho_A^n).
\end{align}
Since $\psi(s)s^{-\alpha}$ is bounded for every $\alpha\geq 0$
and $|\log s|\leq C(s^{-1/8} + s)$, it holds
\begin{align}\label{fA.Linf}
\begin{split}
|f_A(\rho^n)| &\leq \frac{\psi(\sqrt{\rho_A^n\rho_B^n})
	g(\sqrt{\rho_A^n}+\sqrt{\rho_B^n})}{(\rho^n_A\rho_B^n)^{1/4}}
((\rho_B^n)^{1/4}(\rho^n_A)^{1/4}|\log\rho_A^n| + (\rho^n_A)^{1/4} (\rho_B^n)^{5/4})\\
&\leq C\frac{(\rho_B^n)^{1/4}[ (\rho_A^n)^{1/8} + (\rho_A^n)^{5/4} ] + (\rho^n_A)^{1/4} (\rho_B^n)^{5/4}}{1+(\rho_A^n)^2 + (\rho_B^n)^2} \\
&\leq C,
\end{split}
\end{align}
which means that $f_A(\rho^n)$ is bounded in $L^{\infty}(\OmT)$.
Similarly one shows the same bound for $f_B(\rho^n)$.
Let us now consider
\[\na f_A(\rho^n) =  I_1 + I_2 + I_3, \]
where 
\begin{align*} 
I_1 := \psi(\sqrt{\rho_A^n\rho_B^n})
g(\sqrt{\rho_A^n}+\sqrt{\rho_B^n})\nabla (\log\rho_A^n + \rho_B^n) \\
I_2 := \psi(\sqrt{\rho_A^n\rho_B^n})
(\log\rho_A^n + \rho_B^n)\nabla g(\sqrt{\rho_A^n}+\sqrt{\rho_B^n}), \\
I_3 := g(\sqrt{\rho_A^n}+\sqrt{\rho_B^n})(\log\rho_A^n + \rho_B^n)\nabla  \psi(\sqrt{\rho_A^n\rho_B^n}).
\end{align*}
First we give an estimate to $I_1$. From \eqref{ei.A.int} it follows
\begin{align*}
\|I_1\|_{L^{2}(\Omega_T)} &= 2
\left\|
\frac{\psi(\sqrt{\rho_A^n\rho_B^n})}{\sqrt{\rho_A^n\rho_B^n}}
g(\sqrt{\rho_A^n}+\sqrt{\rho_B^n})\sqrt{\rho_B^n}
( \nabla\sqrt{\rho_A^n} + \sqrt{\rho_A^n\rho_B^n}\nabla\sqrt{\rho_B^n} )
\right\|_{L^{2}(\Omega_T)}\\
&\leq
 2
 \left\|\frac{\psi(\sqrt{\rho_A^n\rho_B^n})}{\sqrt{\rho_A^n\rho_B^n}} \right\|_{L^\infty(\Omega_T)}
 \|g(\sqrt{\rho_A^n}+\sqrt{\rho_B^n})\sqrt{\rho_B^n}
 \|_{L^\infty(\Omega_T)}
\|\nabla\sqrt{\rho_A^n} + \sqrt{\rho_A^n\rho_B^n}\nabla\sqrt{\rho_B^n}\|_{L^2(\Omega_T)}\\
&\leq C_T .
\end{align*}
Now, we consider $I_2$. Since $g'/g$ is bounded in $\R_+$ while 
$f_A(\rho^n)$ is bounded in $L^\infty(\Omega_T)$,
\begin{align*}
&\|I_2\|_{L^2(\Omega_T)} = 
\left\| f_A(\rho^n)\frac{g'(\sqrt{\rho_A^n}+\sqrt{\rho_B^n})}{
	g(\sqrt{\rho_A^n}+\sqrt{\rho_B^n})}\nabla (\sqrt{\rho_A^n}+\sqrt{\rho_B^n})\right\|_{L^2(\Omega)}\leq
C \|\nabla(\sqrt{\rho_A^n}+\sqrt{\rho_B^n})\|_{L^2(\Omega_T)}\leq C_T
\end{align*}
where the last inequality comes from \eqref{est.1a}. \\
Finally, let us consider $I_3$. Since $|\psi'(s)|\leq C s^{1/4}|1-s|$ for $s\geq 0$ one obtains
\begin{multline*}
\| I_3\|_{L^2(\Omega_T)} = \left\|
g(\sqrt{\rho_A^n}+\sqrt{\rho_B^n})(\log\rho_A^n + \rho_B^n) \psi'(\sqrt{\rho_A^n\rho_B^n})\nabla (\sqrt{\rho_A^n\rho_B^n})
\right\|_{L^{2}(\Omega_T)}\\
\leq \left\|
g(\sqrt{\rho_A^n}+\sqrt{\rho_B^n})(\log\rho_A^n + \rho_B^n) (\sqrt{\rho_A^n\rho_B^n})^{1/4}|1-\sqrt{\rho_A^n\rho_B^n}|
(\sqrt{\rho_A^n}|\nabla\sqrt{\rho_B^n}| + \sqrt{\rho_B^n}|\nabla\sqrt{\rho_A^n}| )\right\|_{L^{2}(\Omega_T)}\\
\leq \left\|
g(\sqrt{\rho_A^n}+\sqrt{\rho_B^n})(\sqrt{\rho_A^n}+\sqrt{\rho_B^n})
(\log\rho_A^n + \rho_B^n) (\rho_A^n\rho_B^n)^{1/8}
\right\|_{L^{\infty}(\Omega_T)} \\ \times\left\|
(1-\sqrt{\rho_A^n\rho_B^n})(|\nabla\sqrt{\rho_A^n}| + |\nabla\sqrt{\rho_B^n}|)
\right\|_{L^{2}(\Omega_T)}  \leq C_T ,
\end{multline*}
where we used \eqref{est.1b}, \eqref{est.sum.L4}. We conclude that
$\na f_A(\rho^n)$ is bounded in $L^2(\Omega_T)$.
Similarly, one can show that $\nabla f_B(\rho^n)$ is bounded in $L^{2}(\OmT)$, too. This means that the vector fields
\begin{align*}
Y_i^n = (f_i(\rho^n),0,0,0),\quad i \in \{A,B\},
\end{align*}
are bounded in $L^\infty(\Omega)$ and the antisymmetric part of their Jacobian is bounded in $L^2(\Omega_T)$, thus  relatively compact in $W^{-1,r}(\Omega_T)$ for some $r>1$.
Once again, the Div-Curl Lemma (applied to $U^n_A - U^n_B$ defined in \eqref{Un} and $Y^n_A - Y^n_B$) lead to
$$
\overline{(\rho_A^n-\rho_B^n)(f_A(\rho^n)-f_B(\rho^n))} = 
\overline{(\rho_A^n-\rho_B^n)}\,
\overline{(f_A(\rho^n)-f_B(\rho^n))}\quad\mbox{a.e.~in }\Omega_T,
$$
which means
\begin{align}\label{rlr}
\overline{ \frac{(\rho_A^n-\rho_B^n)\psi(\rho_A^n\rho_B^n)}{1+(\sqrt{\rho_A^n} + \sqrt{\rho_B^n})^4} 
\left( 
\log\frac{\rho_A^n}{\rho_B^n} + \rho_B^n - \rho_A^n\right)} = 
\overline{(\rho_A^n-\rho_B^n)}\,
\overline{ \frac{\psi(\rho_A^n\rho_B^n)}{1+(\sqrt{\rho_A^n} + \sqrt{\rho_B^n})^4} \left( 
	\log\frac{\rho_A^n}{\rho_B^n} + \rho_B^n - \rho_A^n\right)} .
\end{align}
Let 
\begin{align*}
\mu^n &= \min\{\rho_A^n,\rho_B^n\} = \frac{1}{2}(\rho_A^n+\rho_B^n) - \frac{1}{2}|\rho_A^n-\rho_B^n|,\\
M^n &= \max\{\rho_A^n,\rho_B^n\} = \frac{1}{2}(\rho_A^n+\rho_B^n) + \frac{1}{2}|\rho_A^n-\rho_B^n|,
\end{align*} and 
\begin{align*}
\sigma^n &= \begin{cases}
1, & \rho_A^n > \rho_B^n, \\
0, & \rho_A^n = \rho_B^n,\\
-1, & \rho_A^n < \rho_B^n.
\end{cases}
\end{align*}
Since both $\rho_A^n+\rho_B^n$ and $|\rho_A^n-\rho_B^n|$ are a.e.~convergent in $\OmT$, then also $\mu^n$, $M^n$ are a.e.~convergent in $\OmT$ (and strongly convergent in $L^{2-\eta}(\OmT)$ for every $\eta>0$) towards some nonnegative functions $\mu$, $M$ (respectively). On the other hand $|\sigma^n|\leq 1$ a.e.~in $\OmT$ and so $\sigma^n\rightharpoonup^*\sigma$ weakly* in $L^\infty(\OmT)$.

Let us point out that 
$$
(\rho_A^n-\rho_B^n)\left( 
\log\frac{\rho_A^n}{\rho_B^n} + \rho_B^n - \rho_A^n\right) = 
(M^n-\mu^n)\left( 
\log\frac{M^n}{\mu^n} + \mu^n - M^n\right)
$$
which implies that $(\rho_A^n-\rho_B^n)\left( 
\log\frac{\rho_A^n}{\rho_B^n} + \rho_B^n - \rho_A^n\right)$ is a.e.~convergent on $\{\mu>0\}$. On the other hand, since $\psi(\rho_A^n\rho_B^n)>0$ only on $\{\rho_A^n\rho_B^n\geq 1/3\} = \{ 1/\mu^n\leq 3 M^n \}$, it follows (via Lagrange's theorem) that for a suitable point $\lambda^n \in [\mu^n,M^n]$
\begin{align*}
\psi(\sqrt{M^n\mu^n})(M^n-\mu^n)\log\frac{M^n}{\mu^n} =
\psi(\sqrt{M^n\mu^n}) \frac{(M^n-\mu^n)^2}{\lambda^n}\leq 
\psi(\sqrt{M^n\mu^n}) \frac{(M^n-\mu^n)^2}{\mu^n}\leq C |M^n|^3 .
\end{align*}
Since $M^n$, $\mu^n$ are bounded in $L^2(\OmT)$, it follows (by dominated convergence) that 
\begin{multline*}
\frac{(\rho_A^n-\rho_B^n)\psi(\rho_A^n\rho_B^n)}{1+(\sqrt{\rho_A^n} + \sqrt{\rho_B^n})^4} \left( 
\log\frac{\rho_A^n}{\rho_B^n} + \rho_B^n - \rho_A^n\right) \\ \to
\frac{(M-\mu)\psi(M\mu)}{1+(\sqrt{M} + \sqrt{\mu})^4}\left( 
\log\frac{M}{\mu} + \mu - M\right) \quad 
\mbox{ strongly in }L^1(\Omega_T).
\end{multline*}
A similar argument allows us to obtain
\begin{multline*}
\frac{\psi(M^n\mu^n)}{1+(\sqrt{M^n} + \sqrt{\mu^n})^4}\left( 
\log\frac{M^n}{\mu^n} + \mu^n - M^n\right) \\ \to
\frac{\psi(M\mu)}{1+(\sqrt{M} + \sqrt{\mu})^4}\left( 
\log\frac{M}{\mu} + \mu - M\right) \quad 
\mbox{ strongly in }L^1(\Omega_T).
\end{multline*}
As a consequence,
\begin{align*}
\rho_A^n-\rho_B^n = \sigma^n(M^n - \mu^n)\rightharpoonup 
\sigma(M - \mu) \quad \mbox{ weakly in }L^1(\OmT),
\end{align*}
\begin{align*}
\frac{\psi(\rho_A^n\rho_B^n)}{1+(\sqrt{\rho_A^n} + \sqrt{\rho_B^n})^4} &\left( 
\log\frac{\rho_A^n}{\rho_B^n} + \rho_B^n - \rho_A^n\right) = \sigma^n\frac{\psi(M^n\mu^n)}{1+(\sqrt{M^n} + \sqrt{\mu^n})^4}\left( 
\log\frac{M^n}{\mu^n} + \mu^n - M^n\right) \\ &\rightharpoonup 
\sigma\frac{\psi(M\mu)}{1+(\sqrt{M} + \sqrt{\mu})^4}\left( 
\log\frac{M}{\mu} + \mu - M\right) \quad 
\mbox{ weakly in }L^1(\OmT).
\end{align*}
From the above relations and \eqref{rlr} 
as well as the fact that
$$
M\mu = 1,\quad \log\frac{M}{\mu} + \mu - M = 0 \quad
\Leftrightarrow\quad M=\mu=1,
$$
we deduce that
$|\sigma|^2 = 1$ in $\{M>\mu ,~ M\mu=1\}$. However, since $|\sigma^n|\leq 1$ a.e.~in $\Omega_T$, $n\in\NN$, this means that 
$$
\overline{(\sigma^n)^2}\leq 1 = |\sigma|^2 = 
( \overline{\sigma^n} )^2\quad \mbox{a.e.~in }\{M>\mu ,~ M\mu=1\}.
$$
However, being $x\mapsto x^2$ strictly convex, it follows from \cite[Thr.~10.20]{FN17} that $\sigma^n\to\sigma$ a.e.~in $\{M>\mu ,~ M\mu=1\}$.
This allows us to deduce that $\rho_A^n-\rho_B^n = \sigma^n(M^n - \mu^n)$ is a.e.~convergent in $\{M\mu=1\} = E^c$, and so are
$\rho_A^n = \frac{1}{2}(\rho_A^n + \rho_B^n) 
+ \frac{1}{2}(\rho_A^n - \rho_B^n)$ and 
$\rho_B^n = \frac{1}{2}(\rho_A^n + \rho_B^n) 
- \frac{1}{2}(\rho_A^n - \rho_B^n)$ (because we already know that $\rho_A^n + \rho_B^n$ is a.e.~convergent in $\Omega_T$). Since we already knew that $\rho_A^n$, $\rho_B^n$ are a.e.~convergent in $E$, we conclude that 
$\rho_A^n$, $\rho_B^n$ are a.e.~convergent in $\Omega_T$ and therefore
by dominated convergence (and \eqref{est.sum.L4}) $\rho_A^n$, $\rho_B^n$ are also strongly convergent in $L^{2-\delta}(\Omega_T)$
for every $\delta\in (0,1]$.\medskip\\
{\bf Step 4: limit in the equations.}
Now we show that \eqref{w.A}, \eqref{w.B} hold for the limit functions $\rho_A$, $\rho_B$. We first study the convergence of the expressions
	$$
	\zeta_A^n = \na\sqrt{\rho_A^n} + \sqrt{\rho_A^n\rho_B^n}\na\sqrt{\rho_B^n}, \quad \text{ and } \quad 
	\zeta_B^n = \na\sqrt{\rho_B^n} + \sqrt{\rho_A^n\rho_B^n}\na\sqrt{\rho_A^n}.
	$$
	Since $\zeta_A^n$, $\zeta_B^n$ are bounded in $L^2(\Omega_T)$, it holds (up to subsequences)
	\begin{align*}
	\zeta_A^n\rightharpoonup \zeta_A,\quad \zeta_B^n\rightharpoonup \zeta_B\quad 
	\mbox{weakly in }L^2(\Omega_T).
	\end{align*}
	We want to show that \eqref{zeta.1}-\eqref{zeta.2} hold.
	Let us consider
	\begin{align*}
	\zeta_A^n + \zeta_B^n &=
		(\sqrt{\rho_A^n\rho_B^n}+1)\na(\sqrt{\rho_A^n}+\sqrt{\rho_B^n}).
	\end{align*}
	We know that $\na(\sqrt{\rho_A^n}+\sqrt{\rho_B^n})$ is bounded in $L^2(\OmT)$ and that $\sqrt{\rho_A^n}+\sqrt{\rho_B^n}\to \sqrt{\rho_A} + \sqrt{\rho_B}$ strongly in $L^2(\OmT)$ (given that $\rho_A^n, \rho_B^n$ are strongly convergent in $L^{2-\eta}(\Omega_T)$ for $\eta \in (0,1]$, as showed in Step 3), therefore
	$$
	\na(\sqrt{\rho_A^n}+\sqrt{\rho_B^n}) \rightharpoonup \na(\sqrt{\rho_A}+\sqrt{\rho_B})\quad \mbox{ weakly in }L^2(\Omega_T).
	$$
	Moreover, thanks to \eqref{est.1a}, $J_A^n + J_B^n$ is bounded in $L^2(\Omega_T)$, while $\sqrt{\rho_A^n\rho_B^n}\to\sqrt{\rho_A\rho_B}$ strongly in $L^{3-\delta}(\Omega_T)$ for every $\delta\in (0,2]$ thanks
	to the property $\rho_A^n\to\rho_A$, $\rho_B^n\to\rho_B$
	a.e.~in $\Omega_T$ (proved in Step 3) and \eqref{est.Hm1.b}. Therefore
	\begin{align*}
	\zeta_A^n + \zeta_B^n &\rightharpoonup
	(\sqrt{\rho_A\rho_B}+1)\na(\sqrt{\rho_A}+\sqrt{\rho_B}) \quad 
	\mbox{ weakly in }L^{2}(\Omega_T).
	\end{align*}
Therefore, \eqref{zeta.1} holds.
Let us now turn our attention to
$$
\zeta_A^n - \zeta_B^n = (1-\sqrt{\rho_A^n\rho_B^n})\na (\sqrt{\rho_A^n} - \sqrt{\rho_B^n}) .
$$
Let us consider a generic function $\Psi\in Y$ (defined in \eqref{class.Y}). By construction and \eqref{est.na.p} it follows that $\Psi(\sqrt{\rho_A^n\rho_B^n})(1-\sqrt{\rho_A^n\rho_B^n})$ is bounded in $L^{4/3}(0,T; W^{1,4/3}(\Omega))$.
 It holds
	\begin{align*}
	\Psi(\sqrt{\rho_A^n\rho_B^n})(\zeta_A^n - \zeta_B^n) 
	&= \na[ \Psi(\sqrt{\rho_A^n\rho_B^n})
	(1-\sqrt{\rho_A^n\rho_B^n})(\sqrt{\rho_A^n}-\sqrt{\rho_B^n}) ]\\
	&\quad - (\sqrt{\rho_A^n}-\sqrt{\rho_B^n})\na [
	(1-\sqrt{\rho_A^n\rho_B^n})\Psi(\sqrt{\rho_A^n\rho_B^n})] .
	\end{align*}
However, due to \eqref{est.sum.L4}, \eqref{est.Hm1.b} and the strong convergence of $\rho^n$ in $L^1(\Omega_T)$ we have
\begin{align*}
\Psi(\sqrt{\rho_A^n\rho_B^n})
(1-\sqrt{\rho_A^n\rho_B^n})(\sqrt{\rho_A^n}-\sqrt{\rho_B^n})
\rightharpoonup
\Psi(\sqrt{\rho_A\rho_B})
(1-\sqrt{\rho_A\rho_B})(\sqrt{\rho_A}-\sqrt{\rho_B}) \quad 
\mbox { weakly in }L^{4/3}(\Omega_T),
\end{align*}
while, since  $\Psi(\sqrt{\rho_A^n\rho_B^n})(1-\sqrt{\rho_A^n\rho_B^n})$ is bounded in $L^{4/3}(0,T; W^{1,4/3}(\Omega))$, it holds
\begin{align*}
	\na [(1-\sqrt{\rho_A^n\rho_B^n})\Psi(\sqrt{\rho_A^n\rho_B^n})]
	\rightharpoonup
	\na [(1-\sqrt{\rho_A\rho_B})\Psi(\sqrt{\rho_A\rho_B})] \quad 
	\mbox{ weakly in }L^{4/3}(\Omega_T),
\end{align*}
so it follows that, for every $\phi\in C^1_c(\Omega_T;\R^2)$,
\begin{align*}
	\int_{\Omega_T} & \Psi(\sqrt{\rho_A}, \sqrt{\rho_B}) (\zeta_A - \zeta_B)\cdot\phi \, \mathrm{d}x \mathrm{d}t\\
	&= \lim_{n\to\infty}\int_{\Omega_T}\Psi(\sqrt{\rho_A^n}, \sqrt{\rho_B^n})(\zeta_A^n - \zeta_B^n)\cdot\phi \, \mathrm{d}x \mathrm{d}t\\ 
	&= -\int_{\Omega_T}\Psi(\sqrt{\rho_A}, \sqrt{\rho_B})
	(1-\sqrt{\rho_A\rho_B})(\sqrt{\rho_A}-\sqrt{\rho_B})\dv{\phi}\,
	\mathrm{d}x \mathrm{d}t\\ 
	&\qquad - \int_{\Omega_T}(\sqrt{\rho_A}-\sqrt{\rho_B})
	\na [(1-\sqrt{\rho_A\rho_B})\Psi(\sqrt{\rho_A}, \sqrt{\rho_B})]
	\cdot\phi \, \mathrm{d}x \mathrm{d}t,
\end{align*}
implying that \eqref{zeta.2} holds. Eq.~\eqref{zeta.3} is derived in a similar way by choosing $\Phi\in Z$ (defined in \eqref{class.Z}) and $\phi\in C^1_c(\Omega,\R^2)$, computing the expression $\int_{\Omega_T}\Phi(\sqrt{\rho_A^n\rho_B^n})\left(
\frac{\zeta_A}{\sqrt{\rho_A^n}}-\frac{\zeta_B}{\sqrt{\rho_B^n}}
\right)\cdot\phi \, \mathrm{d} x \mathrm{d} t$, integrating by parts and taking the limit $n\to\infty$.
%
%

From \eqref{est.rhot} it follows
\begin{align*}
	\pa_t\rho_i^n\rightharpoonup \pa_t\rho_i\quad\mbox{weakly in }L^{4/3}(0,T; W^{-1,4/3}(\Omega)),\quad i \in \{A,B\},
\end{align*}
and via the compact Sobolev embedding $W^{1,4/3}(0,T; W^{1,4}(\Omega)')\hookrightarrow C_{weak}([0,T], W^{1,4}(\Omega)')$
\begin{align}
	\rho_i^n\to \rho_i\quad\mbox{in }C_{weak}([0,T], W^{1,4}(\Omega)'),\quad
	i \in \{A,B\}.\label{rho.C0}
\end{align}
Since $\zeta_i^n\rightharpoonup \zeta_i$ weakly in $L^2(\Omega_T)$, $i \in \{A,B\}$, while $\sqrt{\rho_A^n}\to\sqrt{\rho_A}$, $\sqrt{\rho_B^n}\to\sqrt{\rho_B}$
strongly in $L^{4-\delta}(\Omega_T)$ for every $\delta\in (0,3]$ (thanks to Step 3), it follows
that one can take the limit in \eqref{w.A.n}, \eqref{w.B.n} and conclude that
$\rho\equiv (\rho_A,\rho_B)$ satisfies \eqref{w.A}, \eqref{w.B}.

Let us now make sure that \eqref{w.ic} is satisfied by $\rho$.
Given any constant in time $\phi\in W^{1,4}(\Omega)$ and any $t\in (0,T)$, from \eqref{est.rhot} (and the fundamental theorem of calculus, which holds since $\rho^n\in W^{1,4/3}(0,T; W^{1,4}(\Omega)')$ and therefore $t\mapsto \langle\rho^n(t),\phi\rangle$ is in $W^{1,4/3}(0,T)$) it follows
\begin{align*}
	\left|\int_\Omega\rho_i^n(t)\phi \, \mathrm{d}x - 
	\int_\Omega\rho_i^{in}\phi \, \mathrm{d}x\right| &\leq
	\int_0^t |\langle\pa_t\rho_i^n(t') , \phi\rangle | \, \mathrm{d}t' \leq
	\|\pa_t\rho_i^n\|_{L^{4/3}(0,t; W^{1,4}(\Omega)')}
	\|\phi\|_{L^{4}(0,t; W^{1,4}(\Omega))}\\
	&\leq C t^{1/4}\|\phi\|_{W^{1,4}(\Omega)}.
\end{align*}
Since \eqref{rho.C0} holds, it follows that 
$\int_\Omega\rho_i^n(t)\phi \d x\to \int_\Omega\rho_i(t)\phi \, \mathrm{d}x$
as $n\to\infty$, so
\begin{align*}
	\left|\int_\Omega\rho_i(t)\phi \, \mathrm{d}x - 
	\int_\Omega\rho_i^{in}\phi \, \mathrm{d}x\right| &\leq C t^{1/4}\|\phi\|_{W^{1,4}(\Omega)},
	\quad \text{for all }\phi\in W^{1,4}(\Omega),
\end{align*}
which means
\begin{align*}
	\|\rho_i(t)-\rho_i^{in}\|_{W^{1,4}(\Omega)'}\leq C t^{1/4} .
\end{align*}
Being $t\in (0,T)$ arbitrary and $C>0$ independent of $t$, we deduce
that \eqref{w.ic} holds. Therefore, $\rho$
is a weak solution to \eqref{eq1}--\eqref{bc} according to Definition \ref{def.weaksol}.
This finishes the proof.
\end{proof}

\section{Stationary states}\label{S:ss}
In this section, we study the steady states of \eqref{E:system}-\eqref{E:1_BC}, i.e.~the constant-in-time solutions to \eqref{E:system}-\eqref{E:1_BC}. 
First, we provide a definition.
\begin{dfn}[Steady state] A {\em steady state} of \eqref{E:system}-\eqref{E:1_BC}
	is a weak solution to \eqref{E:system}-\eqref{E:1_BC} in the sense of Def.~\ref{def.weaksol} that is constant in time.
\end{dfn}
In the following we prove that the only allowed steady states for the system are constant.
\begin{prp}\label{prop.ss} Let $\Omega\, \in \R^2$ open, bounded and connected with Lipschitz boundary.
	Every steady state of \eqref{E:system}-\eqref{E:1_BC} is constant in $\Omega$.
\end{prp}
\begin{proof}
According to the definition of a weak solution, the integrated entropy balance \eqref{ieb.zeta} must hold. Being a steady state constant in time, this implies that $\zeta_A = \zeta_B = 0$ a.e.~in $\Omega$.
From \eqref{zeta.1} it follows immediately that $\rrhoA + \rrhoB = k_1$ is constant in $\Omega$. In particular, $\rho_A, \rho_B\in L^\infty(\Omega)$.
From \eqref{zeta.2} we deduce
\begin{align*}
	\int_{\Omega} &
	\left[\Psi(\sqrt{\rho_A\rho_B}) (1-\sqrt{\rho_A\rho_B}) ( \sqrt{\rho_A} - \sqrt{\rho_B} ) \right]\dv{\phi}  \, \mathrm{d}x  \\ 
	&+ \int_{\Omega}( \sqrt{\rho_A} - \sqrt{\rho_B} )\nabla\left[ 
	\Psi(\sqrt{\rho_A\rho_B}) (1-\sqrt{\rho_A\rho_B}) \right]\cdot\phi \, \mathrm{d}x  = 0
\end{align*}
for every $\phi\in C^1_c(\Omega;\R^2)$, $\Psi\in Y$ defined in \eqref{class.Y}. 
Since $\rrhoA + \rrhoB$ is constant it also holds trivially
\begin{align*}
	\int_{\Omega} &
	\left[\Psi(\sqrt{\rho_A\rho_B})(1-\sqrt{\rho_A\rho_B}) ( \sqrt{\rho_A} + \sqrt{\rho_B} ) \right]\dv{\phi} \dx  \\ 
	&+ \int_{\Omega}( \sqrt{\rho_A} + \sqrt{\rho_B} )\nabla\left[  \Psi(\sqrt{\rho_A\rho_B})(1-\sqrt{\rho_A\rho_B}) \right]\cdot\phi \dx  = 0 .
\end{align*}
Putting the two previous equations together yields
\begin{align}\nonumber
	\int_{\Omega} & \left[\Psi(\sqrt{\rho_A\rho_B})(1-\sqrt{\rho_A\rho_B}) \sqrt{\rho_i} \right]\dv{\phi} \dx \\
	& + \int_{\Omega}\sqrt{\rho_i} \nabla\left[ \Psi(\sqrt{\rho_A\rho_B})(1-\sqrt{\rho_A\rho_B}) \right]\cdot\phi
	\dx  = 0 ,\quad i \in \{A,B\}.\label{ss.1}
\end{align}

Via a density argument the above equation holds for every $\phi\in H^1(\Omega)$. Therefore (thanks to the boundedness of $\sqrt{\rho_A}$, $\sqrt{\rho_B}$ and \eqref{reg.3}) one can choose $\phi = \Psi(\sqrt{\rho_A\rho_B}) (1-\sqrt{\rho_A\rho_B}) \sqrt{\rho_j}\xi\in H^1(\Omega)$ with $\xi\in C^1_c(\Omega)$
arbitrary and obtain by summing in $i,j \in \{A,B\}$, $i\neq j$:
\begin{align*}
	&\int_{\Omega}\left[\Psi(\sqrt{\rho_A\rho_B})^2 
	(1-\sqrt{\rho_A\rho_B})^2 \rrhoAB \right]\dv{\xi}\dx \\
	&\qquad +\int_{\Omega} 2\Psi(\sqrt{\rho_A\rho_B}) (1-\sqrt{\rho_A\rho_B})\rrhoAB\nabla\left[\Psi(\sqrt{\rho_A\rho_B}) (1-\sqrt{\rho_A\rho_B}) \right]\cdot\xi \dx = 0
\end{align*}
which is equivalent to
\begin{align*}
	&\int_\Omega\Psi(\sqrt{\rho_A\rho_B}) (1-\sqrt{\rho_A\rho_B})\nabla\left[\Psi(\sqrt{\rho_A\rho_B}) (1-\sqrt{\rho_A\rho_B})\rrhoAB \right]\cdot\xi \dx \\ 
	&\qquad = \int_\Omega
	\Psi(\sqrt{\rho_A\rho_B}) (1-\sqrt{\rho_A\rho_B})\rrhoAB\nabla\left[\Psi(\sqrt{\rho_A\rho_B}) (1-\sqrt{\rho_A\rho_B}) \right]\cdot\xi \dx,
\end{align*}
for every $\xi\in C^1_c(\Omega; \R^2)$, $\Psi \in Y$ defined in \eqref{class.Y}. 

Subtracting $\int_\Omega\Psi(\sqrt{\rho_A\rho_B}) (1-\sqrt{\rho_A\rho_B})\nabla\left[\Psi(\sqrt{\rho_A\rho_B}) (1-\sqrt{\rho_A\rho_B})\right]\cdot\xi \dx$ from both sides of the above inequality leads to
\begin{align*}
	&\int_\Omega\Psi(\sqrt{\rho_A\rho_B}) (1-\sqrt{\rho_A\rho_B})\nabla\left[\Psi(\sqrt{\rho_A\rho_B}) (1-\sqrt{\rho_A\rho_B})^2\right]\cdot\xi \dx \\ 
	&\qquad = \int_\Omega
	\Psi(\sqrt{\rho_A\rho_B}) (1-\sqrt{\rho_A\rho_B})^2\nabla\left[\Psi(\sqrt{\rho_A\rho_B}) (1-\sqrt{\rho_A\rho_B}) \right]\cdot\xi \dx,
\end{align*}
for every $\xi\in C^1_c(\Omega; \R^2)$, $\Psi \in Y$.
Choosing $\Psi(s)=\frac{1-s}{1+s}$ and arguing by density lead to
\begin{align*}
	\frac{(1-\sqrt{\rho_A\rho_B})^2}{1+\sqrt{\rho_A\rho_B}}\nabla\left[\frac{(1-\sqrt{\rho_A\rho_B})^3}{1+\sqrt{\rho_A\rho_B}}\right] - 
	\frac{(1-\sqrt{\rho_A\rho_B})^3}{1+\sqrt{\rho_A\rho_B}}\nabla\left[\frac{(1-\sqrt{\rho_A\rho_B})^2}{1+\sqrt{\rho_A\rho_B}} \right] = 0
	\quad\mbox{a.e.~in }\Omega .
\end{align*}
Let $w = \frac{(1-\sqrt{\rho_A\rho_B})^3}{1+\sqrt{\rho_A\rho_B}}$.
Since $g(s)\equiv\frac{(1-s)^3}{1+s}$ is strictly monotone and therefore invertible as mapping $\R_+\to (-\infty,1]$, we can define $F(y)\equiv \frac{(1-g^{-1}(y))^2}{1+g^{-1}(y)}$
for $y\leq 1$ and deduce
\begin{align*}
	\nabla\tilde{F}(w) = F(w)\nabla w - w\nabla F(w) = 0\quad\mbox{a.e.~in }\Omega 
\end{align*}
where $\tilde{F}'(y) \equiv F(y)-y F'(y) = -y^2\frac{d}{dy}\left( \frac{F(y)}{y}\right)$ for $y\leq 1$, $y\neq 0$. 
It follows that $\tilde{F}(w)$ is constant.
However $\frac{F(y)}{y} = \frac{1}{1-g^{-1}(y)}$ which is strictly monotone for $y\neq 0$. This means that $\tilde{F}$ is strictly monotone. As a consequence, $w=g(\sqrt{\rho_A\rho_B})$ is constant.
Being $g$ strictly monotone, it follows that 
$\rrhoAB = k_2$ is constant in $\Omega$. We distinguish two cases, according to the value of $k_2$.\medskip
\\
{\em Case 1: $k_2\neq 1$.} In this situation \eqref{ss.1} immediately yields that $\rho_A$, $\rho_B$ are constant, provided that one chooses $\Psi$ such that $\Psi(k_2)\neq 0$.\medskip\\
{\em Case 2: $k_2=1$.} In this case we must consider \eqref{zeta.3}, 
with 
$$\Phi(s) = \begin{cases}
	1 - \cos(\pi s) & 0\leq s \leq 1\\
	2 & s > 1
\end{cases} $$ 
which indeed belongs to the class $Z$ defined in \eqref{class.Z}. We get
\begin{align*}
	\int_{\Omega}\left(\rho_A - \rho_B + \log\frac{\rho_B}{\rho_A}\right)\dv{\phi}
	\dx = 0\qquad\forall \phi\in C^1_c(\Omega),
\end{align*}
which implies that a constant $k_3>0$ exists such that
\begin{align*}
	\rho_A - \rho_B + \log\frac{\rho_B}{\rho_A} = k_3\quad\mbox{a.e.~in }\Omega .
\end{align*}
Since $\rho_A\rho_B=1$ by assumption, it follows
\begin{align*}
	F(\rho_A) \equiv \rho_A - \frac{1}{\rho_A} - 2\log\rho_A = k_3\quad\mbox{a.e.~in }\Omega .
\end{align*}
However, $F'(s) = 1 + \frac{1}{s^2} - \frac{2}{s} = \frac{(s-1)^2}{s^2}>0$ for $s\neq 1$,
which means that $F$ is strictly monotone. We conclude that $\rho_A$ is constant, implying that also $\rho_B$ is constant. This finishes the proof of the Proposition.
\end{proof}
The result might mean that the class of solutions we considered is perhaps too small, as segregated states are ruled out. On the other hand, the definition arises naturally from the weak stability argument and only employs the entropy structure of the equations. It is entirely possible that different analytical tools might yield segregated steady states.

\subsection{Linear stability analysis} \label{S:linearStabilityAnalysis}
We now consider the stability of the steady states of the system \eqref{E:system}-\eqref{E:1_BC}, copied here for reference:
\begin{equation*}
\begin{cases}
\pa_t \rho_A(t,x,y) = \frac{1}{4} \nabla \cdot \left( \nabla \rho_A(t,x,y) + 2 \beta c \rho_A(t,x,y) \nabla \rho_B(t,x,y) \right),\\
\pa_t \rho_B(t,x,y) = \frac{1}{4} \nabla \cdot \left( \nabla \rho_B(t,x,y) + 2 \beta c \rho_B(t,x,y) \nabla \rho_A(t,x,y) \right).
\end{cases}
\end{equation*}
In order to better understand the system, we perform a linear stability analysis around the uniformly distributed steady state,
\begin{align} \label{unif.ss}
\begin{cases}
\bar{\rho}_A(x)= 
N_1, \quad &0 \leq x \leq L,\\
\bar{\rho}_B(x) = 
N_2, \quad &0 \leq x < L,
\end{cases}
\end{align} where $N_1, N_2 >0$ are the equilibrium densities and $L \in \R_+$.

To this end, we consider perturbations of the form $\epsilon = \delta_i e^{\alpha t} e^{ikx}$ with $\delta_i \ll 1$ where $ i \in \{A,B\}$.
\begin{align} \label{perturb}
{\rho}_A(x)&= \bar{\rho}_A + \delta_A e^{\alpha t + i kx},
 \quad 0 \leq x \leq L,\\
{\rho}_B(x)&= \bar{\rho_B} + \delta_B e^{\alpha t + i kx},
 \quad 0 \leq x < L.
\end{align}
\begin{lem}
	The uniform steady state solution \eqref{unif.ss} of system \eqref{E:system}-\eqref{E:1_BC} is linearly stable if the following condition holds true:
	\begin{align}
	\label{lin.stab}
	\beta c \leq \frac{1}{2 \sqrt{\bar{\rho}_A \bar{\rho}_B} }.
	\end{align}
\end{lem}
\begin{proof}
We plug solutions \eqref{perturb} into system \eqref{E:system}-\eqref{E:1_BC}:
\begin{align*}
&\begin{cases}
\frac{\partial }{\partial t} (\bar{\rho}_A + \delta_A e^{\alpha t + i kx}) =  \frac{1}{4} \nabla \cdot \left( \nabla (\bar{\rho}_A + \delta_A e^{\alpha t + i kx}) + 2 \beta c (\bar{\rho}_A + \delta_A e^{\alpha t + i kx}) \nabla (\bar{\rho_B} + \delta_B e^{\alpha t + i kx}) \right),\\
\frac{\partial }{\partial t} (\bar{\rho}_B + \delta_B e^{\alpha t + i kx}) =  \frac{1}{4} \nabla \cdot \left( \nabla (\bar{\rho}_B + \delta_B e^{\alpha t + i kx}) + 2 \beta c (\bar{\rho}_B + \delta_B e^{\alpha t + i kx}) \nabla (\bar{\rho_A} + \delta_A e^{\alpha t + i kx}) \right),
\end{cases}
\end{align*} and obtain 
\begin{align*}
&\begin{cases}
\alpha \delta_A e^{\alpha t + i kx} =  \frac{1}{4} \nabla \cdot \left( (i k \delta_A e^{\alpha t + i kx}) + 2 \beta c (\bar{\rho}_A + \delta_A e^{\alpha t + i kx}) (i k \delta_B e^{\alpha t + i kx}) \right),\\
\alpha \delta_B e^{\alpha t + i kx} =  \frac{1}{4} \nabla \cdot \left( (i k \delta_B e^{\alpha t + i kx}) + 2 \beta c (\bar{\rho}_B + \delta_B e^{\alpha t + i kx}) (i k \delta_A e^{\alpha t + i kx}) \right).
\end{cases}\\
\end{align*}
This implies 
\begin{align*}
&\begin{cases}
\alpha \delta_A =  \frac{-k^2}{4} \delta_A  - \frac{k^2}{2} \delta_B  \beta c \bar{\rho}_A  + \mathcal{O}(\delta_A \delta_B),\\
\alpha \delta_B =  \frac{-k^2}{4} \delta_B  - \frac{k^2}{2}  \delta_A \beta c \bar{\rho}_B  + \mathcal{O}(\delta_A \delta_B).
\end{cases}
\end{align*}
Writing the linear part of this in matrix-vector form $(M -\alpha I) \vec{\delta} = 0$, we have
\begin{equation*}
\begin{bmatrix}
-(\frac{k^2}{4}+\alpha) & - \frac{k^2}{2}  \beta c \bar{\rho}_A\\
- \frac{k^2}{2}  \beta c \bar{\rho}_B & -(\frac{k^2}{4}+\alpha) \\
\end{bmatrix} 
\begin{bmatrix}  \delta_A\\  \delta_B\\  \end{bmatrix} = \begin{bmatrix}  0\\  0\\  \end{bmatrix}.
\end{equation*}
Hence, it follows from the characteristic polynomial 
\[\left(\frac{k^2}{4}+\alpha \right)^2 - \frac{k^4}{4} (\beta c)^2 \bar{\rho}_A \bar{\rho}_B = 0, \]
that
\[\alpha = \frac{k^2}{2} \left(\frac{-1}{2} \pm  \beta c \sqrt{\bar{\rho}_A \bar{\rho}_B} \right). \]
Therefore, the uniform steady state solution will be linearly stable when $\beta c \leq  \frac{1}{2 \sqrt{\bar{\rho}_A \bar{\rho}_B}}$.
\end{proof}
\section{Long-time behavior}\label{S:long_time}
In this section we give our result on the long-time behaviour of solutions. We denote
\[\fint_\Omega\equiv |\Omega|^{-1}\int_\Omega.\]

\begin{thm}[Convergence to steady state]\label{thm:ss_conv}
Let $\rho : \OmT\to\R^2_+$ be a weak solution to \eqref{eq1}--\eqref{bc} according to Definition \ref{def.weaksol}.
Define the constant steady state associated to $\rho$ as 
$$\rho^\infty\equiv (\rho_A^\infty, \rho_B^\infty),\quad
\rho_i^\infty = \fint_\Omega \rho_i(t) \d x = 
\fint_\Omega \rho_i^{in} \d x\quad i\in \{A,B \},\quad t>0,
$$
and assume that $\rho_i^\infty > 0$ for $i\in \{A,B \}$. 
Define the relative entropy functional as
\begin{align*}
\en(\rho\mid \rho^\infty) = 
\int_\Omega h^*(\rho\mid\rho^\infty) \d x,
\end{align*} where 
where 
\begin{align*}
	h^*(\rho\mid\rho^\infty) &=  h(\rho) - h(\rho^\infty) - h'(\rho^\infty)\cdot(\rho-\rho^\infty), \\
	&= \rho_A\log\frac{\rho_A}{\rho_A^\infty} + \rho_B\log\frac{\rho_B}{\rho_B^\infty} + \rho_A^\infty - \rho_A + \rho_B^\infty - \rho_B + (\rho_A - \rho_A^\infty)(\rho_B - \rho_B^\infty),\\
	h(\rho) &= \rho_A\log\rho_A - \rho_A + \rho_B\log\rho_B  - \rho_B + \rho_A\rho_B .
\end{align*}
Then $\en(\rho(t)\mid \rho^\infty)\to 0$ as $t\to\infty$.

Furthermore, if $\rho_A^\infty\leq 1$ and $\rho_B^\infty\leq 1$, then $\rho(t)\to\rho^\infty$ strongly in $L^1(\Omega)$ as $t\to \infty$.
%
\end{thm}

\begin{remark}
In the physical variables, the constraint on the steady state is $\rho_i^\infty\leq (2\beta c)^{-1}$, $i \in \{A,B\}$.
\end{remark}
\begin{proof}
The proof is divided into two parts. First we prove that 
$\lim_{t\to\infty}\en(\rho(t)\mid\rho^\infty)=0$, then we show that,
if both masses are not larger than 1,
then $\rho(t)\to\rho^\infty$ as $t\to\infty$ strongly in $L^1(\Omega)$.\medskip\\
{\bf Step 1: Show that $\lim_{t\to\infty}\en(\rho(t)\mid\rho^\infty)=0$.} 
\begin{remark}
In the following we will identify the quantities $\zeta_A$, $\zeta_B$ with $\nabla\rrhoA + \rrhoAB\nabla\rrhoB$,
$\nabla\rrhoB + \rrhoAB\nabla\rrhoA$, respectively.
Albeit this identification is not known to hold exactly for nondegenerate weak solutions (as the latter expressions are not clearly defined),
the present theorem could be proved also by using the properties \eqref{zeta.1}--\eqref{zeta.2} and proceeding in a similar way as in the proof of Prop.~\ref{prop.ss}.
We chose to omit technical details for the sake of a simple exposition.
\end{remark}
From \eqref{ei.A.int} it follows that
\begin{align*}
\int_0^\infty\int_\Omega\left( 
(1+\sqrt{\rho_A\rho_B})^2|\nabla(\sqrt{\rho_A}+\sqrt{\rho_B})|^2
+ (1-\sqrt{\rho_A\rho_B})^2|\nabla(\sqrt{\rho_A}-\sqrt{\rho_B})|^2
 \right) \mathrm{d}x \mathrm{d}t \leq C .
\end{align*}
As a consequence there exists an increasing sequence of time instants $t_n\to\infty$ such that
\begin{align}\label{cnv.1}
(1+\sqrt{\rho_A\rho_B})\nabla(\sqrt{\rho_A}+\sqrt{\rho_B})
\mid_{t=t_n}\to 0,
\quad
(1-\sqrt{\rho_A\rho_B})\nabla(\sqrt{\rho_A}-\sqrt{\rho_B})
\mid_{t=t_n}\to 0 \quad \mbox{ strongly in }L^2(\Omega),
\end{align}
as $n\to\infty$. 

Define $\rho_i^n\equiv\rho_i(t_n)$ for $i \in \{A,B\}$, $n\in\NN$. 
In particular, $\nabla(\sqrt{\rho_A^n}+\sqrt{\rho_B^n})$ is bounded
in $L^2(\Omega)$. However by mass conservation $\sqrt{\rho_A^n}+\sqrt{\rho_B^n}$ is bounded in $L^2(\Omega)$, and so
$\sqrt{\rho_A^n}+\sqrt{\rho_B^n}$ is bounded in $H^1(\Omega)$. By Sobolev embedding (in 2 space dimensions) $\sqrt{\rho_A^n}+\sqrt{\rho_B^n}$ is bounded in $L^p(\Omega)$ for every $p<\infty$.

From \eqref{cnv.1} we deduce that $\nabla(\sqrt{\rho_A^n}+\sqrt{\rho_B^n})\to 0$ strongly in $L^2(\Omega)$. Poincar\'e-Wirtinger Lemma yields
\begin{align}\label{cnv.2}
\sqrt{\rho_A^n}+\sqrt{\rho_B^n} - \fint_\Omega(\sqrt{\rho_A^n}+\sqrt{\rho_B^n}) \, \mathrm{d}x \to 0
\quad \mbox{ strongly in }L^p(\Omega), \text{ for all } p<\infty.
\end{align}
From \eqref{cnv.1} we also deduce
\begin{align*}
	\|(1-\sqrt{\rho_A^n\rho_B^n})\nabla(\sqrt{\rho_A^n}+\sqrt{\rho_B^n})\|_{L^2(\Omega)}\leq
	\|(1+\sqrt{\rho_A^n\rho_B^n})\nabla(\sqrt{\rho_A^n}+\sqrt{\rho_B^n})\|_{L^2(\Omega)}\to 0,\\
	\|(1-\sqrt{\rho_A^n\rho_B^n})\nabla(\sqrt{\rho_A^n}-\sqrt{\rho_B^n})\|_{L^2(\Omega)}\to 0,
\end{align*}
which immediately implies
\begin{align}\label{cnv.a}
	(1-\sqrt{\rho_A^n\rho_B^n})\nabla\rho_i^n\to 0 \quad 
	\mbox{ strongly in }L^2(\Omega),~~i \in \{A,B\}.
\end{align}
The above relation and the uniform $L^p$ bound for $\rho_i^n$ lead to
(also, thanks to the definition of weak solution \eqref{def.weaksol}, $(\sqrt{\rho_A^n\rho_B^n}-1)^3\sqrt{\rho_i^n}\in H^1(\Omega)$ is an admissible test function, $i \in \{A,B\}$)
\begin{align}\label{cnv.AB}
\frac{1}{5}\nabla\left( (\sqrt{\rho_A^n\rho_B^n}-1)^5 \right)
= (\sqrt{\rho_A^n\rho_B^n}-1)^4\left( 
\sqrt{\rho_A^n}\nabla\sqrt{\rho_B^n} +\sqrt{\rho_B^n}\nabla\sqrt{\rho_A^n} \right)\to 0 \quad \mbox{ strongly in }L^{2-\epsilon}(\Omega),
\end{align}
for every $\epsilon>0$. Again, by Poincar\'e Lemma one deduces that
\begin{align}\label{cnv.3}
(\sqrt{\rho_A^n\rho_B^n}-1)^5 - \fint_\Omega (\sqrt{\rho_A^n\rho_B^n}-1)^5 \mathrm{d }x\to 0
\quad \mbox{ strongly in }L^{p}(\Omega), \text{ for all } p<\infty .
\end{align}
Since $\rho_A^n$, $\rho_B^n$ are bounded in $L^p(\Omega)$ for every $p<\infty$, then the sequences of real numbers 
$\fint_\Omega(\sqrt{\rho_A^n}+\sqrt{\rho_B^n}) \,\mathrm{d}x$,
$\fint_\Omega (\sqrt{\rho_A^n\rho_B^n}-1)^5 \, \mathrm{d}x$ are bounded in $\R$, therefore  up to subsequences 
\begin{align*}
\fint_\Omega(\sqrt{\rho_A^n}+\sqrt{\rho_B^n}) \d x\to c_1,\quad
\fint_\Omega (\sqrt{\rho_A^n\rho_B^n}-1)^5 \d x\to \tilde c_2 .
\end{align*}
for some suitable constants $c_1$, $\tilde c_2\geq 0$.
From the above relation and \eqref{cnv.2}, \eqref{cnv.3} we get
\begin{align*}
\sqrt{\rho_A^n}+\sqrt{\rho_B^n}\to c_1,\quad
(\sqrt{\rho_A^n\rho_B^n}-1)^5\to \tilde c_2 \quad 
\mbox{ strongly in }L^p(\Omega), \text{ for all } p<\infty.
\end{align*}
In particular $(\sqrt{\rho_A^n\rho_B^n}-1)^5\to \tilde c_2$ a.e.~in $\Omega$. However, being $x\in\R\mapsto (x-1)^5\in\R$ globally invertible with continuous inverse $y\in\R\mapsto (y+1)^{1/5}\in\R$, we deduce that 
$\sqrt{\rho_A^n\rho_B^n}\to c_2 := 1+\tilde c_2^{1/5}$ a.e.~in
$\Omega$. Being $\rho_A^n$, $\rho_B^n$ bounded in $L^p(\Omega)$
for every $p<\infty$, it follows
\begin{align}\label{cnv.4}
	\sqrt{\rho_A^n}+\sqrt{\rho_B^n}\to c_1,\quad
	\sqrt{\rho_A^n\rho_B^n}\to c_2 \quad
	\mbox{ strongly in }L^p(\Omega), \text{ for all } p<\infty .
\end{align}
As a consequence
\begin{align*}
\rho_A^n+\rho_B^n = (\sqrt{\rho_A^n}+\sqrt{\rho_B^n})^2 - 2\sqrt{\rho_A^n\rho_B^n}\to c_1^2 - 2 c_2 \quad \mbox{ strongly in }L^p(\Omega),  \text{ for all } p<\infty.
\end{align*}
The above relation and the mass conservation imply
\begin{align}\label{cnv.5}
\rho_A^n+\rho_B^n \to \rho_A^\infty +\rho_B^\infty \quad \mbox{ strongly in }L^p(\Omega), \text{ for all } p<\infty.
\end{align}
Moreover,
\begin{align*}
|\sqrt{\rho_A^n}-\sqrt{\rho_B^n}|^2 = (\sqrt{\rho_A^n}+\sqrt{\rho_B^n})^2 - 4\sqrt{\rho_A^n\rho_B^n}
\to c_1^2 - 4 c_2 \quad \mbox{ strongly in }L^p(\Omega), \text{ for all } p<\infty,
\end{align*}
and so
\begin{align}\label{cnv.6}
|\rho_A^n - \rho_B^n| = (\sqrt{\rho_A^n}+\sqrt{\rho_B^n})|\sqrt{\rho_A^n}-\sqrt{\rho_B^n}|
\to c_3 := c_1\sqrt{c_1^2 - 4 c_2} \quad \mbox{ strongly in }L^p(\Omega), \text{ for all } p<\infty.
\end{align}
In particular, since $2\max\{x,y\} = x + y + |x-y|$, 
$2\min\{x,y\} = x + y - |x-y|$ for every $x,y\geq 0$, it follows
\begin{align}\label{cnv.Mmu}
M^n := \max\{\rho_A^n,\rho_B^n\}\to M,\quad
\mu^n := \min\{\rho_A^n,\rho_B^n\}\to \mu, \quad 
\mbox{ strongly in }L^p(\Omega),\text{ for all } p<\infty ,
\end{align}
and $M$, $\mu$ are constants. Notice that 
$\sqrt{M\mu} = \lim_{n\to\infty}\sqrt{\rho_A^n\rho_B^n}$ a.e.~in $\Omega$.

From \eqref{cnv.a}, \eqref{cnv.AB} it follows
\begin{align*}
\nabla \left[ (\sqrt{\rho_A^n\rho_B^n}-1)^5\rho_i^n \right]\to 0 \quad 
\mbox { strongly in }L^{2-\epsilon}(\Omega),~i \in \{A,B\},
\end{align*} 
that is
\begin{align*}
(\sqrt{\rho_A^n\rho_B^n} -1)^5\rho_i^n\to \theta_i \quad 
\mbox{ strongly in }L^{2-\epsilon}(\Omega),~i \in \{A,B\},
\end{align*}
for some constants $\theta_A$, $\theta_B$.
We distinguish two cases:\medskip\\
{\em \textbf{Case 1}: $M\mu\neq 1$.} 

In this case $\rho_i^n$ is a.e.~convergent in $\Omega$ to a constant which, due to mass conservation and uniform $L^p(\Omega)$ bounds, must be equal to $\rho_i^\infty$. It follows
\begin{align}\label{lt.goal.1}
\rho_i^n\to\rho_i^\infty \quad \mbox{ strongly in }L^q(\Omega),\text{ for all } q<\infty,~i\in \{A,B\}.
\end{align}
%
{\em \textbf{Case 2}: $M\mu = 1$.} 

In this case let us observe that
relation \eqref{cnv.1} can be rewritten as
\begin{align}\label{cnv.1b}
\sqrt{\rho_A^n}\nabla\left( \log\rho_A^n + \rho_B^n \right)\to 0,\quad
\sqrt{\rho_B^n}\nabla\left( \log\rho_B^n + \rho_A^n \right)\to 0,  \quad 
\mbox{ strongly in }L^2(\Omega).
\end{align}
Let $\psi$ like in \eqref{def.psi}. Since
\begin{align*}
\nabla(\psi(\sqrt{\rho_A^n\rho_B^n})(\log\rho_A^n + \rho_B^n)) &= 
\psi'(\sqrt{\rho_A^n\rho_B^n})(\log\rho_A^n + \rho_B^n)\nabla\sqrt{\rho_A^n\rho_B^n} + 
\psi(\sqrt{\rho_A^n\rho_B^n})\nabla(\log\rho_A^n + \rho_B^n)\\
&=\sqrt{\rho_B^n}\frac{\psi'(\sqrt{\rho_A^n\rho_B^n})}{\sqrt{\rho_A^n\rho_B^n}(1-\sqrt{\rho_A^n\rho_B^n})}\sqrt{\rho_A^n}
(\log\rho_A^n + \rho_B^n)
(1-\sqrt{\rho_A^n\rho_B^n})\nabla\sqrt{\rho_A^n\rho_B^n}\\
&\quad + \sqrt{\rho_B^n}
\frac{\psi(\sqrt{\rho_A^n\rho_B^n})}{ \sqrt{\rho_A^n\rho_B^n} }
\sqrt{\rho_A^n}\nabla(\log\rho_A^n + \rho_B^n)
\end{align*}
from \eqref{cnv.1}, \eqref{cnv.1b}
it follows that
\begin{align*}
\nabla(\psi(\sqrt{\rho_A^n\rho_B^n})(\log\rho_A^n + \rho_B^n))\to 0 \quad \mbox{ strongly in }L^{2-\epsilon}(\Omega), \text{ for all }\epsilon>0.
\end{align*}
Being $\int_\Omega \psi(\sqrt{\rho_A^n\rho_B^n})(\log\rho_A^n + \rho_B^n) \, \mathrm{d}x$ bounded, we deduce
\begin{align*}
\psi(\sqrt{\rho_A^n\rho_B^n})(\log\rho_A^n + \rho_B^n)\to c_4 \quad 
\mbox{ strongly in }L^{2-\epsilon}(\Omega),\text{ for all }\epsilon>0.
\end{align*}
In a similar way,
\begin{align*}
\psi(\sqrt{\rho_A^n\rho_B^n})(\log\rho_B^n + \rho_A^n)\to c_5 \quad 
\mbox{ strongly in }L^{2-\epsilon}(\Omega),\text{ for all }\epsilon>0.
\end{align*}
In particular, since $\sqrt{\rho_A^n\rho_B^n}\to 1$ a.e.~in $\Omega$, then
\begin{align*}
\log\frac{\rho_A^n}{\rho_B^n} + \rho_B^n - \rho_A^n\to c_6 := c_4 - c_5\quad\mbox{a.e.~in }\Omega.
\end{align*}
Let $\sigma^n = \frac{\rho_A^n - \rho_B^n}{|\rho_A^n - \rho_B^n|}$
on $\{M^n > \mu^n\}$, $\sigma^n = 0$ on $\{\rho_A^n = \rho_B^n\}$.
We just proved
\begin{align*}
\sigma^n\left( \log\frac{M^n}{\mu^n} + \mu^n - M^n \right)\to c_6 \quad\mbox{a.e.~in }\Omega.
\end{align*}
However, we know from \eqref{cnv.Mmu} that $\log\frac{M^n}{\mu^n} + \mu^n - M^n \to \log\frac{M}{\mu} + \mu - M$ a.e.~in $\Omega$,
with $M$, $\mu$ constants such that
$\sqrt{M\mu} = 1$, so
\begin{align*}
\log\frac{M^n}{\mu^n} + \mu^n - M^n \to 
2\log M + \frac{1}{M} - M\quad\mbox{a.e.~in }\Omega.
\end{align*}
Since the function $x\in (0,\infty)\mapsto 2\log x + \frac{1}{x} - x\in\R$ is one-to-one (strictly decreasing), it vanishes only at $x=1$. If $M=\mu=1$ then $\rho_A^n - \rho_B^n\to 0$ a.e.~in $\Omega$
and so \eqref{lt.goal.1} holds. Let us therefore assume $M>\mu$.
In this case
\begin{align*}
\sigma^n\to c_7 :=
c_6\left( 2\log M + \frac{1}{M} - M \right)^{-1}
 \quad\mbox{a.e.~in }\Omega .
\end{align*}
It follows that $\rho_A^n - \rho_B^n = \sigma^n(M^n - \mu^n)$
is a.e.~convergent in $\Omega$ towards a constant, i.e.~\eqref{lt.goal.1} holds.
The a.e.~convergence $\rho^n\to\rho^\infty$ in $\Omega$ and the continuity of $\rho\in\R^2_+\mapsto h^*(\rho\mid\rho^\infty)\in\R$
imply that $h^*(\rho^n\mid\rho^\infty)\to 0$ a.e.~in $\Omega$,
while the uniform $L^p(\Omega)$ bound for $\rho^n$, valid for every $p<\infty$, implies that $h^*(\rho^n\mid\rho^\infty)$
is bounded (at least) in $L^2(\Omega)$. It follows that 
$h^*(\rho^n\mid\rho^\infty)\to 0$ strongly in $L^1(\Omega)$,
that is (given the definition of $\en(\rho\mid\rho^\infty)$ and $\rho^n = \rho(\cdot,t_n)$) $\lim_{n\to\infty}\en(\rho(t_n)\mid\rho^\infty)=0$. However, since $t\mapsto\en(\rho(t)\mid\rho^\infty)$ is nonincreasing in time, we conclude that $\lim_{t\to\infty}\en(\rho(t)\mid\rho^\infty)=\lim_{n\to\infty}\en(\rho(t_n)\mid\rho^\infty)=0$.
\medskip\\
{\bf Step 2: Show that $\lim_{t\to\infty}\rho(t)=\rho^\infty$ strongly in $L^1(\Omega)$.} 

Assume $\rho_A^\infty$, $\rho_B^\infty\leq 1$. 

We aim to prove that there exists $R>0,~ \gamma>0 $ such that 
\begin{align}\label{claim.h.1}
	\mbox{if $\rho_A+\rho_B\geq R$ then }
	h^*(\rho\mid\rho^\infty) \geq \gamma (\rho_A+\rho_B),\\
	\label{claim.h.2}
	h^*(\rho\mid\rho^\infty) > 0\quad\text{for all }\rho\in\R^2_+\backslash\{\rho^\infty\}.
\end{align}
This strategy is justified by the following
\begin{claim}\label{claim.h}
	If \eqref{claim.h.1}, \eqref{claim.h.2} hold, 
	then $\lim_{t\to\infty}\rho(t)=\rho^\infty$ strongly in $L^1(\Omega)$. 
\end{claim}
\begin{proof}[Proof of Claim \ref{claim.h}]
	In fact, being $\rho\in\R^2\mapsto h^*(\rho\mid\rho^\infty)\in\R$
	continuous, \eqref{claim.h.2} implies that, for every $\epsilon>0$, $h^*(\cdot\mid\rho^\infty)$ is uniformly positive 
	on the compact set $\mathcal{O}_1\equiv\{\rho\in\R^2_+ : ~\rho_A + \rho_B\leq R,~ |\rho-\rho^\infty|\geq\epsilon \}$, while \eqref{claim.h.1} implies
	that $h^*(\cdot\mid\rho^\infty)$ is uniformly positive 
	on $\mathcal{O}_2\equiv\{\rho\in\R^2_+ : ~\rho_A + \rho_B > R,~ |\rho-\rho^\infty|\geq\epsilon \}$. It follows that $h^*(\cdot\mid\rho^\infty)$ is uniformly positive on $\mathcal{O}\equiv\mathcal{O}_1\cup\mathcal{O}_2$, that is,
	\begin{align}\label{meas.conv.h}
		\forall\epsilon>0,~~\exists C_\epsilon>0:~~\rho\in\R^2,~|\rho-\rho^\infty|\geq\epsilon ~~\Rightarrow~~ h^*(\cdot\mid\rho^\infty)\geq C_\epsilon .
	\end{align}
	Given an arbitrary sequence $t_n\to\infty$, from Step 1 we know that $0=\lim_{n\to\infty}\en(\rho(t_n)\mid\rho^\infty)=
	\lim_{n\to\infty}\int_\Omega h^*(\rho(t_n)\mid\rho^\infty)dx$. Since $h^*(\rho\mid\rho^\infty)\geq 0$ a.e.~in $\Omega$ (consequence of \eqref{claim.h.2}), property \eqref{meas.conv.h} implies for $\epsilon>0$ arbitrary
	\begin{align*}
		C_\epsilon\mbox{meas}\{ |\rho(t_n)-\rho^\infty|\geq\epsilon\}\leq 
		\int_{ \{ |\rho(t_n)-\rho^\infty|\geq\epsilon\}} h^*(\rho(t_n)\mid\rho^\infty)\d x\leq 
		\int_\Omega h^*(\rho(t_n)\mid\rho^\infty)\d x\to 0
	\end{align*}
	as $n\to\infty$. This means that $\rho(t_n)\to\rho^\infty$ in measure in $\Omega$ as $n\to\infty$. Given $\rho_A(t_n)\log\rho_A(t_n)$, $\rho_B(t_n)\log\rho_B(t_n)$
	are bounded in $L^1(\Omega)$ thanks to \eqref{ieb.zeta},
	we deduce via dominated convergence that $\rho(t_n)\to\rho^\infty$ strongly in $L^1(\Omega)$. Being the sequence $t_n\to\infty$ arbitrary, we conclude that
	$\rho(t)\to\rho^\infty$ strongly in $L^1(\Omega)$ as $t\to\infty$.
\end{proof}

Let us start with \eqref{claim.h.1}. Let $R>0$ to be fixed later.
Let $\rho\in\R^2_+$ with $\rho_A + \rho_B\geq R$. It holds
\begin{align*}
	\frac{h^*(\rho\mid\rho^\infty)}{\rho_A + \rho_B} &= 
	\frac{\rho_A}{\rho_A + \rho_B}\log\left( \frac{\rho_A}{\rho_A^\infty} \right) 
	+\frac{\rho_B}{\rho_A + \rho_B}\log\left( \frac{\rho_B}{\rho_B^\infty} \right) 
	+ \frac{\rho_A\rho_B}{\rho_A + \rho_B}\\
	&+\frac{\rho_A^\infty - \rho_A + \rho_B^\infty - \rho_B}{\rho_A + \rho_B} 
	+ \frac{\rho_A^\infty\rho_B^\infty - \rho_A\rho_B^\infty
		- \rho_B\rho_A^\infty }{\rho_A + \rho_B}.
\end{align*}
Writing
\begin{align*}
	\log\left( \frac{\rho_i}{\rho_i^\infty} \right) = 
	\log\left( \rho_A + \rho_B \right) +
	\log\left( \frac{\rho_i}{\rho_i^\infty(\rho_A + \rho_B)} \right),
	\quad i \in \{ A,B \},
\end{align*}
and exploiting the fact that $\frac{\rho_A\rho_B}{\rho_A + \rho_B}\geq 0$ we obtain
\begin{align*}
	\frac{h^*(\rho\mid\rho^\infty)}{\rho_A + \rho_B} &\geq
	\log\left( \rho_A + \rho_B \right) + \Xi,
\end{align*}
where we defined
\begin{align*}
	\Xi &\equiv \frac{\rho_A}{\rho_A + \rho_B}\log\left( \frac{\rho_A}{\rho_A^\infty(\rho_A + \rho_B)} \right) 
	+\frac{\rho_B}{\rho_A + \rho_B}\log\left( \frac{\rho_B}{\rho_B^\infty(\rho_A + \rho_B)} \right) \\
	&+\frac{\rho_A^\infty - \rho_A + \rho_B^\infty - \rho_B}{\rho_A + \rho_B} 
	+ \frac{\rho_A^\infty\rho_B^\infty - \rho_A\rho_B^\infty
		- \rho_B\rho_A^\infty }{\rho_A + \rho_B},
\end{align*}
which is clearly bounded. Choosing $R>0$ large enough then yields
\eqref{claim.h.1}.

%

Let us now show \eqref{claim.h.2}. We begin by proving that $h^*(\rho\mid\rho^\infty) > 0$ for $\min\{\rho_A , \rho_B\}=0$. Since $h^*(0\mid\rho^\infty) > 0$ trivially, let us consider the case $\rho_A=0$, $\rho_B>0$ (the complementary case $\rho_B=0$, $\rho_A>0$ is treated in an analogous way). We need to study the function
\begin{align*}
f(\rho_B) \equiv h^*(\rho\mid\rho^\infty)\vert_{\rho_A=0} = 
\rho_A^\infty + \rho_B\log\left( \frac{\rho_B}{\rho_B^\infty} \right) + \rho_B^\infty - \rho_B - \rho_A^\infty(\rho_B - \rho_B^\infty),\quad \rho_B>0.
\end{align*}
Clearly $f$ is a convex function that is positive for $\rho_B=0$ and tends to infinity when $\rho_B\to\infty$. Its point of absolute minimum is
$\rho_B = \rho_B^\infty\exp(\rho_A^\infty)$, which means
\begin{align*}
f(\rho_B) \geq f(\rho_B^\infty\exp(\rho_A^\infty)) = \rho_A^\infty + \rho_B^\infty\left( 1 + \rho_A^\infty  - \exp(\rho_A^\infty) \right),
\quad \rho_B\geq 0.
\end{align*}
The function $s\in [0,1]\mapsto \exp(s) - 1 - 2s\in\R $ is strictly convex, vanishes at zero and equals $e-3<0$ at $1$. It follows that it is negative in $(0,1]$, that is, $\exp(s)<1+2s$ for $0<s\leq 1$. We deduce
\begin{align*}
f(\rho_B) \geq f(\rho_B^\infty\exp(\rho_A^\infty)) > \rho_A^\infty(1-\rho_B^\infty) \geq 0,\quad \rho_B\geq 0.
\end{align*}
Therefore, $h^*(\rho\mid\rho^\infty) > 0$ for $\min\{\rho_A , \rho_B\}=0$.

Let us assume by contradiction that a point $\rho'\in \R^2_+\backslash\{ \rho^\infty \}$ exists such that $h^*(\rho'\mid\rho^\infty)\leq 0$.
From \eqref{claim.h.1} we deduce that $h^*(\rho\mid\rho^\infty)>0$ for 
$\rho_A+\rho_B\geq R$, so $\rho_A' + \rho_B' < R$. Furthermore, since
$h^*(\rho\mid\rho^\infty) > 0$ for $\min\{\rho_A , \rho_B\}=0$, it follows that
$\rho_A'>0$ as $\rho_B' > 0$. We deduce that the function $\rho\mapsto h^*(\rho\mid\rho^\infty)$ achieves local minimum inside the open region 
$\{ \rho_A > 0, \rho_B>0, \rho_A + \rho_B < R \}$ in a point $\tilde{\rho}\neq\rho^\infty$;
in particular $Dh^*(\tilde{\rho}\mid\rho^\infty) = 0$, i.e.~$Dh(\tilde{\rho}) = Dh(\rho^\infty)$. Let us now show that the only solution $\rho\in (0,\infty)^2$
to $Dh(\rho)=Dh(\rho^\infty)$ is $\rho=\rho^\infty$. The equation rewrites as
\begin{align*}
\log\rho_A + \rho_B = \log\rho_A^\infty + \rho_B^\infty,\quad
\log\rho_B + \rho_A = \log\rho_B^\infty + \rho_A^\infty,
\end{align*}
which leads to
\begin{align*}
\rho_A = \rho_A^\infty\exp(\rho_B^\infty - \rho_B),\quad
g(\rho_B) = \log\rho_B^\infty + \rho_A^\infty,\quad
g(s) \equiv \log(s) + \rho_A^\infty\exp(\rho_B^\infty - s).
\end{align*}
Since $\rho_A^\infty\leq 1$ and $\rho_B^\infty\leq 1$ it holds
\begin{align*}
g'(s) = \frac{1}{s} - \rho_A^\infty\exp(\rho_B^\infty - s) \geq 
\frac{1 - s\exp(1-s)}{s}>0\quad\mbox{for }s > 0,~~s\neq 1,
\end{align*}
since $s\mapsto s\exp(1-s)$ achieves its strict maximum as $s=1$. This means that $g$ is strictly increasing and therefore the equation $g(\rho_B) = \log\rho_B^\infty + \rho_A^\infty$ has exactly one solution (i.e.~$\rho_B = \rho_B^\infty$). 

We conclude that \eqref{claim.h.2} holds.
This finishes the proof.	
\end{proof}

\section{Numerical results}\label{S:NR}

In this section we present numerical simulations illustrating Theorem \ref{thm:ss_conv}. In Figure \ref{fig:oc} panel (c) we observe the long-term solutions to system \eqref{E:system}-\eqref{E:1_BC} with initial data $\rho_A(0,x)= .5+e^{-(x-1)^2}$ and $\rho_B(0,x)= .1+e^{-(x+1)^2}$ which is seen in panel (a).  As expected from Theorem \ref{thm:ss_conv}, the solutions converge to the constant equilibrium solutions. 
Note that $\rho_A$ has an initial mass that is larger than the mass of $\rho_B$ and thus the constant equilibriums solution observed at time $t=500$ is larger.  

Figure \ref{fig:oc1} illustrates the time evolution of the two energy functionals.  We observe that they both seem to stabilize at a minimum 
by the time $t=50$.  
\begin{figure}[H]
  \center
         \subfloat[Initial Densities]{\label{fig:1}\includegraphics[width=0.3\textwidth]{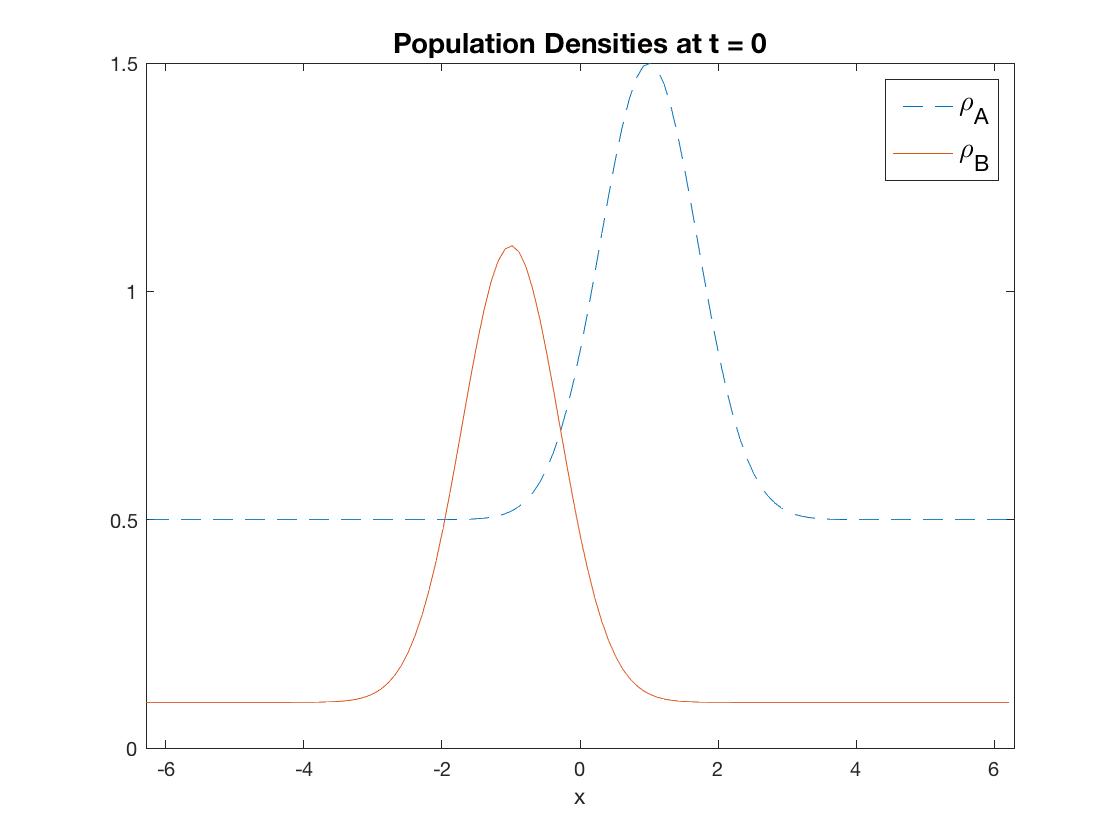}}\;
\subfloat[Densities at $t = 1.245$]{\label{fig:2}\includegraphics[width=0.3\textwidth]{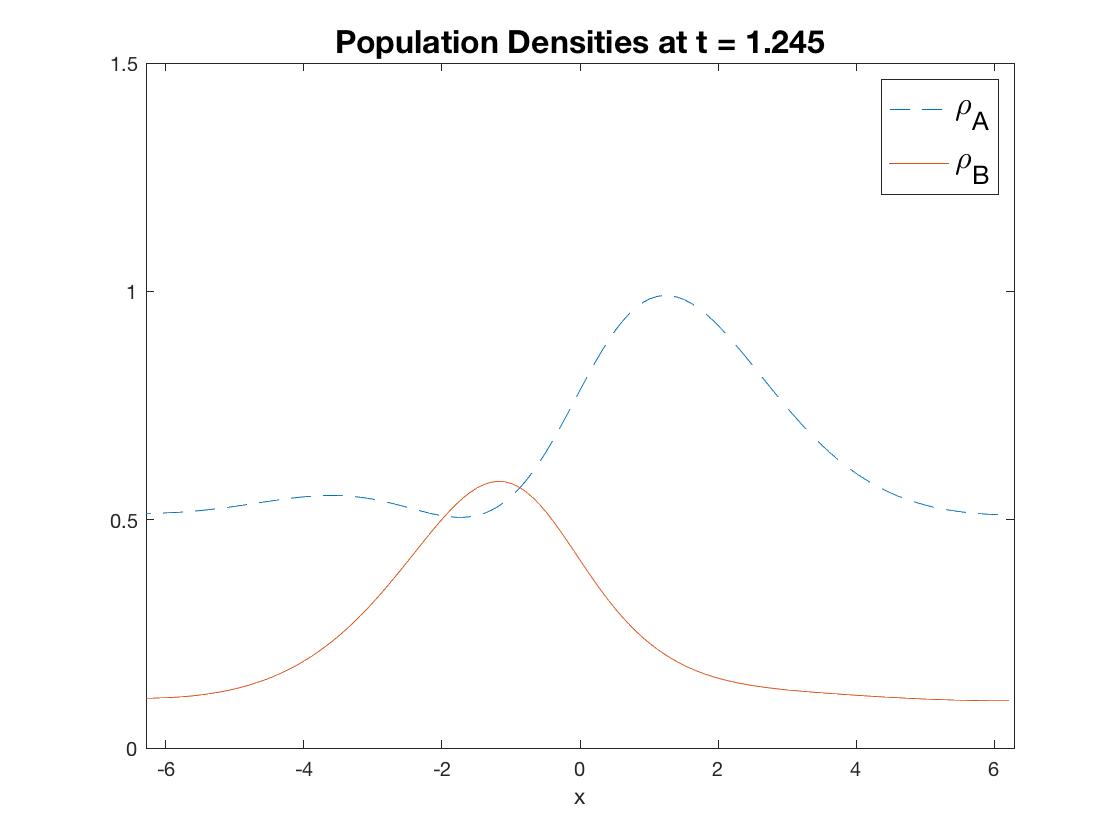}}\;
\subfloat[Densities at $t = 500$]{\label{fig:3}\includegraphics[width=0.3\textwidth]{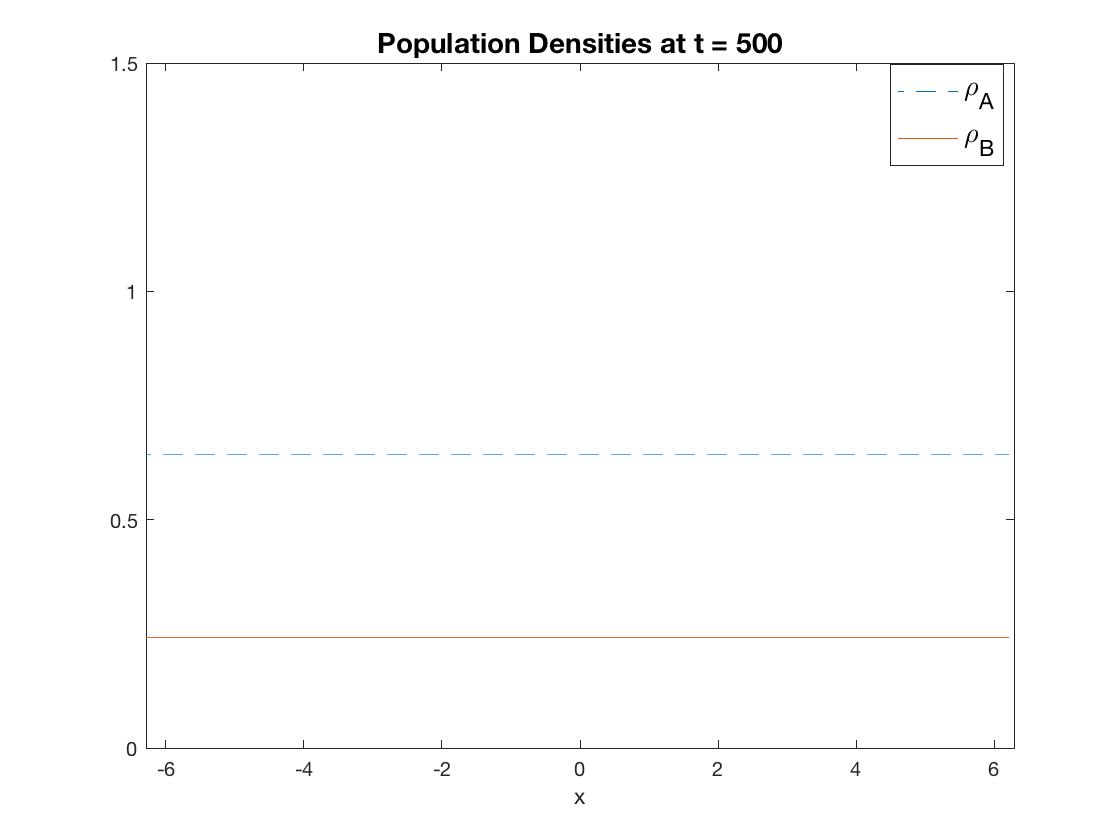}}
   \caption{Numerical solutions to system \eqref{E:system}-\eqref{E:1_BC} with initial densities given by $\rho_A(0,x)= .5+e^{-(x-1)^2}$ and $\rho_B(0,x)= .1+e^{-(x+1)^2}.$
   Panel (b) illustrates transient dynamics and panel (c) the long-time behavior of the solution.}   \label{fig:oc}
\end{figure}

\begin{figure}[H]
  \center
                  \subfloat[Natural Energy]{\label{fig:1}\includegraphics[width=0.4\textwidth]{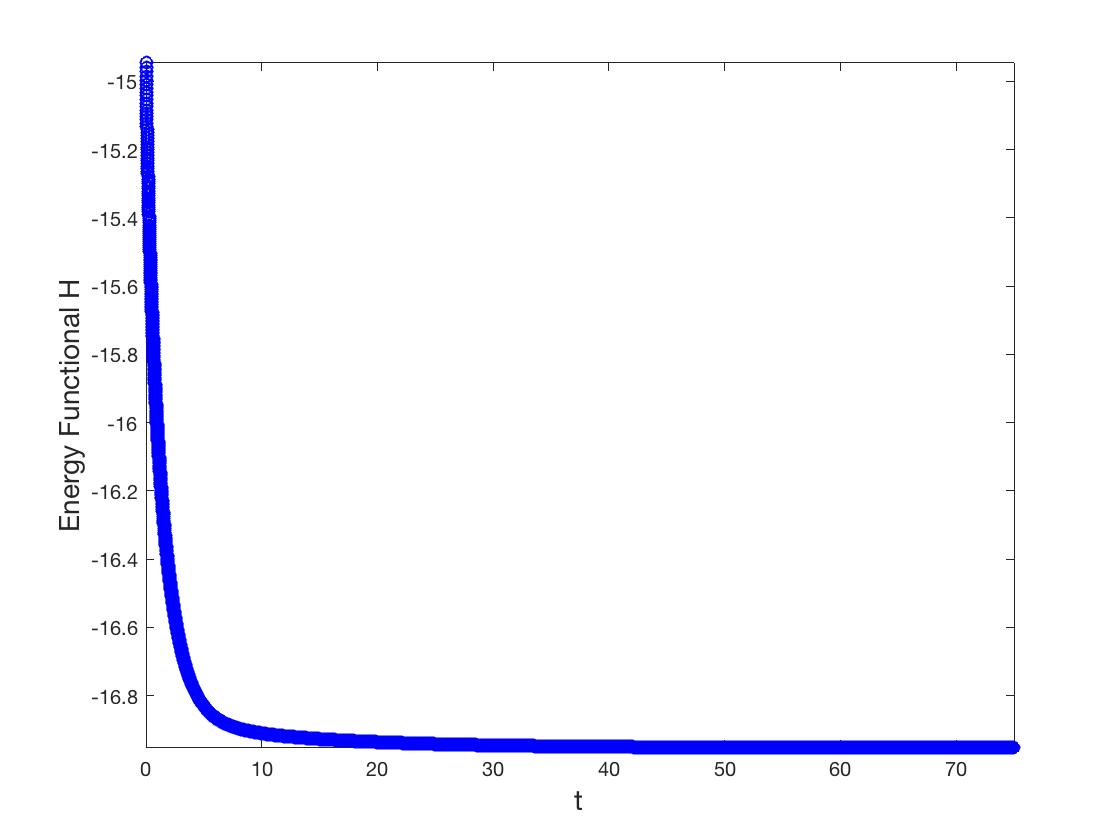}}\;
          \subfloat[Maxwell-Boltzmann Energy]{\label{fig:2}\includegraphics[width=0.4\textwidth]{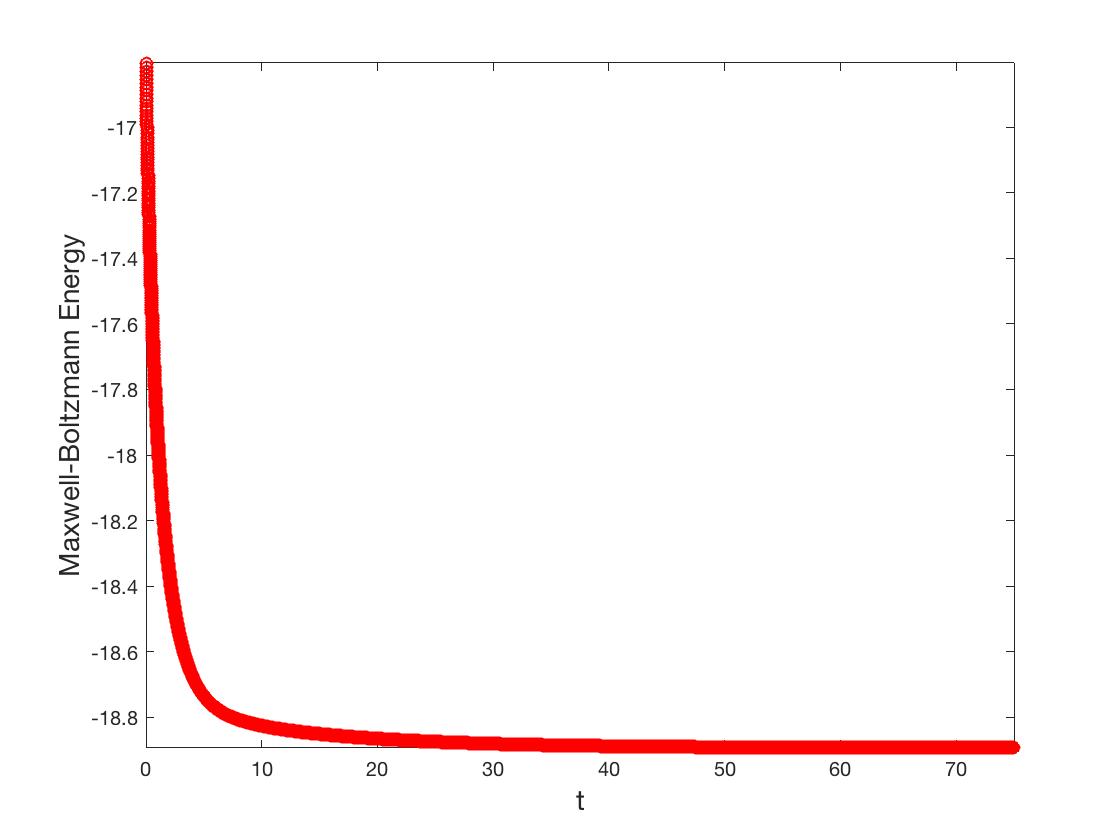}}\;
   \caption{Energy decay with time}   \label{fig:oc1}
\end{figure}

Figure \ref{fig:oc3} illustrates similar results as discussed above.  A difference is that the initial densities 
have a similar mass such that $\rho_A(0,x)= .5+e^{-(x-1)^2}$ and $\rho_B(0,x)= .5+e^{-(x+1)^2}$ as observed in panel (a) of Figure \ref{fig:oc3}.  Thus, the final states of the densities are the same, as seen in panel (c). 
\begin{figure}[H]
  \center
                    \subfloat[Initial Densities]{\label{fig:1}\includegraphics[width=0.3\textwidth]{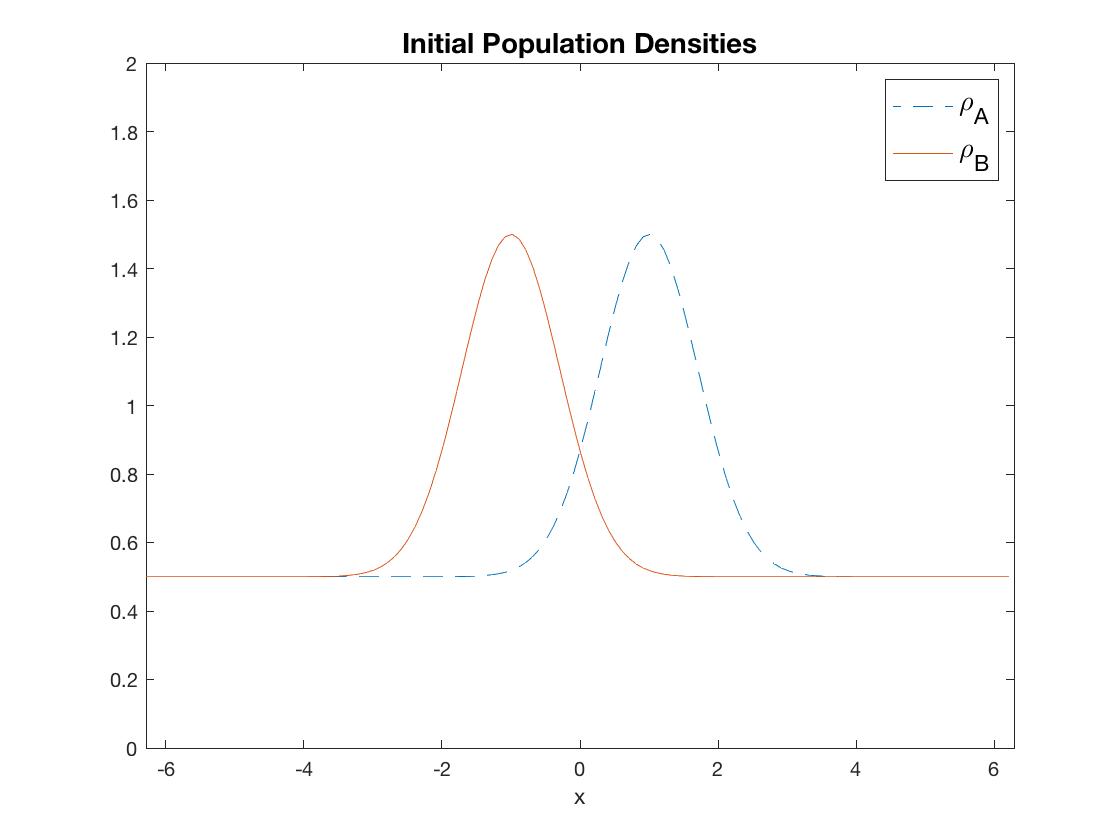}}\;
          \subfloat[Densities at $t = 1.245$]{\label{fig:2}\includegraphics[width=0.3\textwidth]{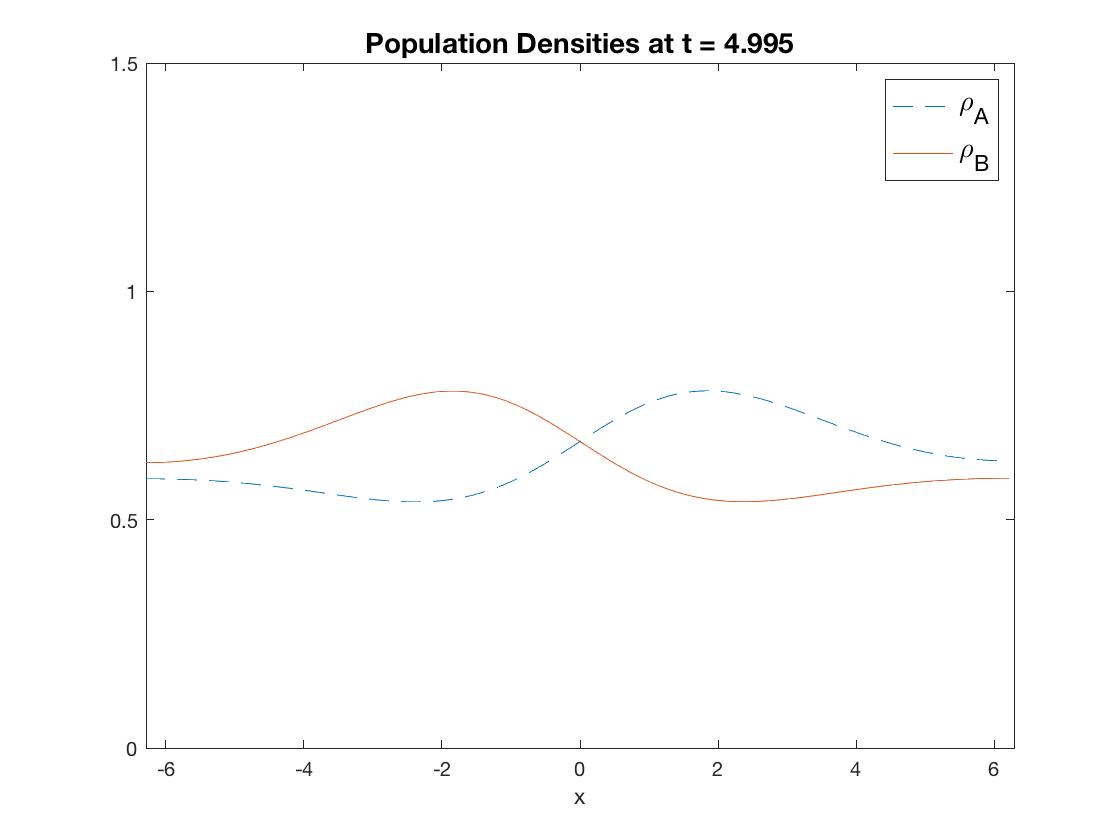}}\;
          \subfloat[Densities at $t = 500$]{\label{fig:3}\includegraphics[width=0.3\textwidth]{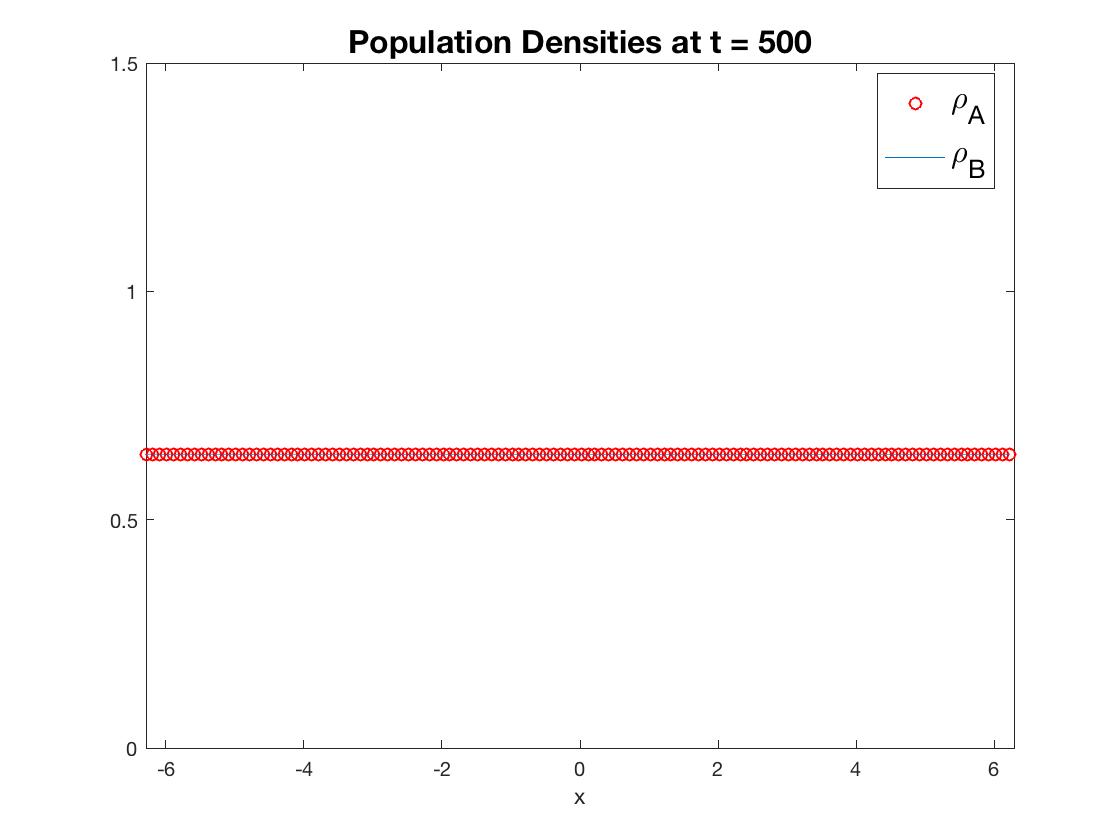}}
   \caption{Numerical solutions with initial densities $\rho_A(0,x)= .5+e^{-(x-1)^2}$ and $\rho_B(0,x)= .5+e^{-(x+1)^2}$}   \label{fig:oc3}
\end{figure}

%

Finally, we illustrate a result in two-dimensions in Figure \ref{fig:oc5}.  These results are consistent with Theorem  \ref{thm:ss_conv}. 
On a final note, the numerical schemes seem to break when initial densities have large mass. 

\begin{figure}[H]
  \center
  \subfloat[$\rho_A(0,x)$]{\label{fig:1}\includegraphics[width=0.32\textwidth]{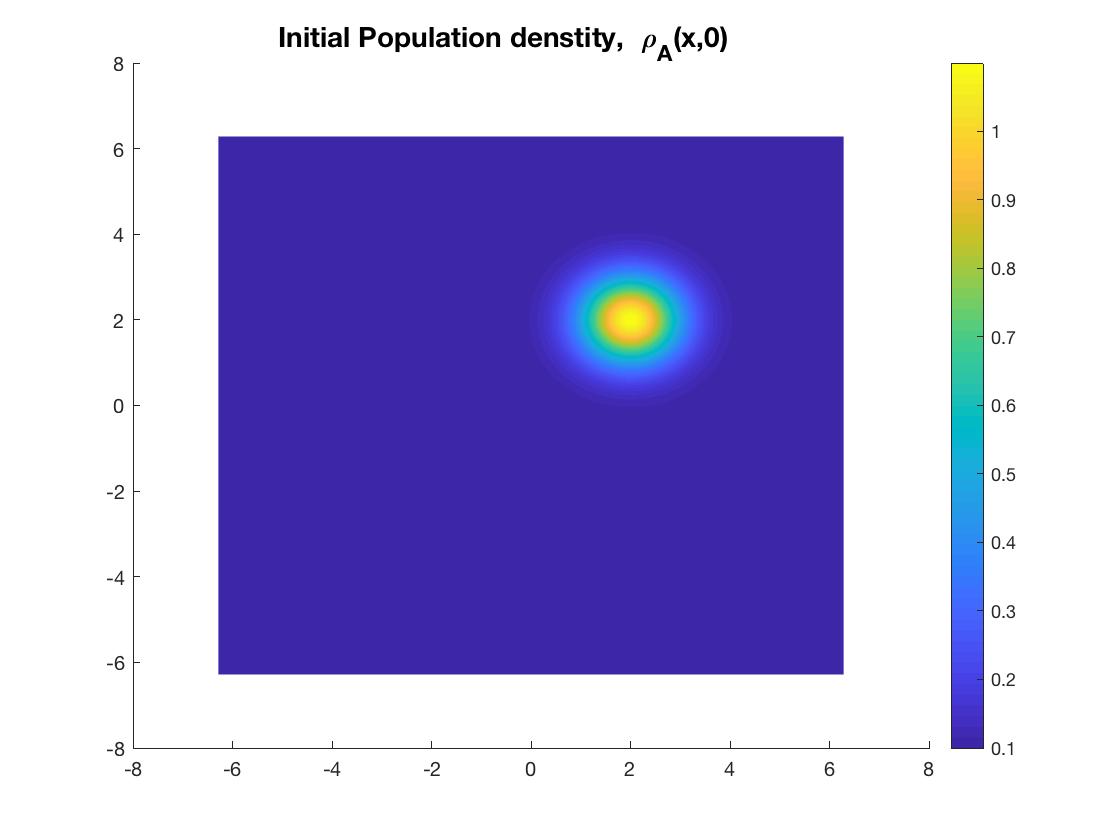}}\;
\subfloat[$\rho_A(1,x)$]{\label{fig:2}\includegraphics[width=0.32\textwidth]{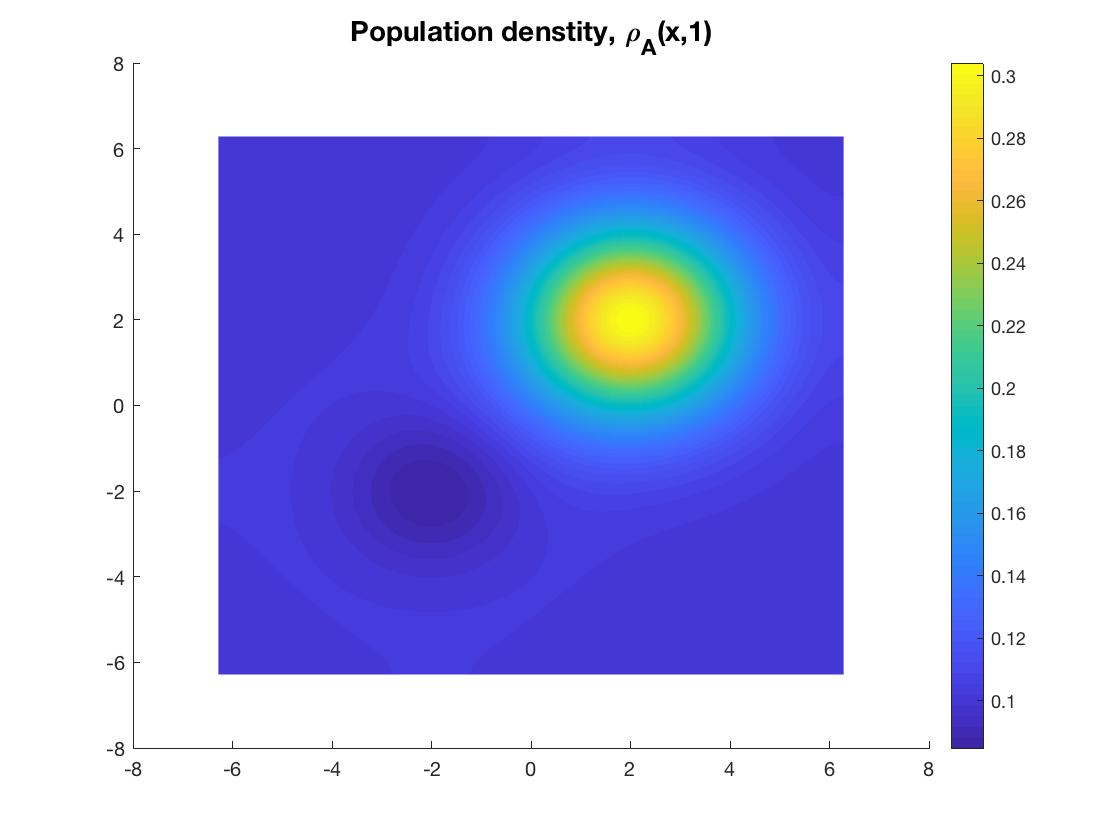}}\;
\subfloat[$\rho_A(30,x)$]{\label{fig:3}\includegraphics[width=0.32\textwidth]{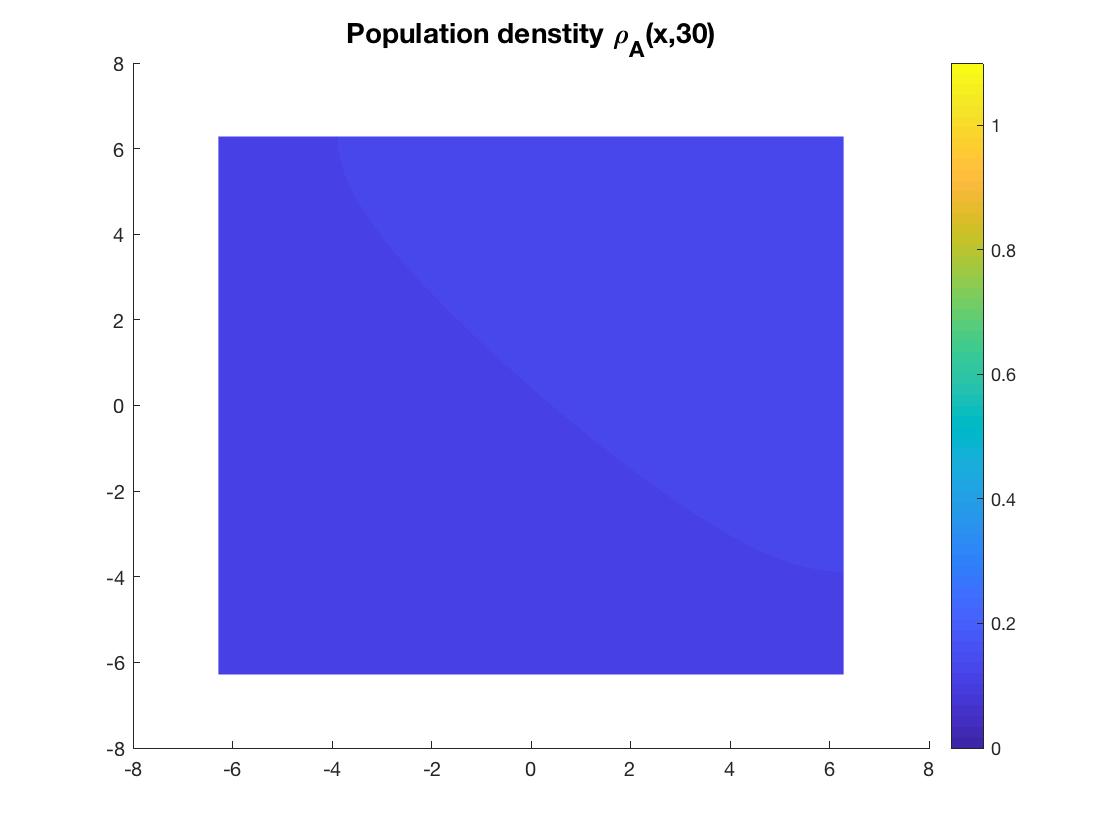}}\\
\subfloat[$\rho_B(0,x)$]{\label{fig:1}\includegraphics[width=0.32\textwidth]{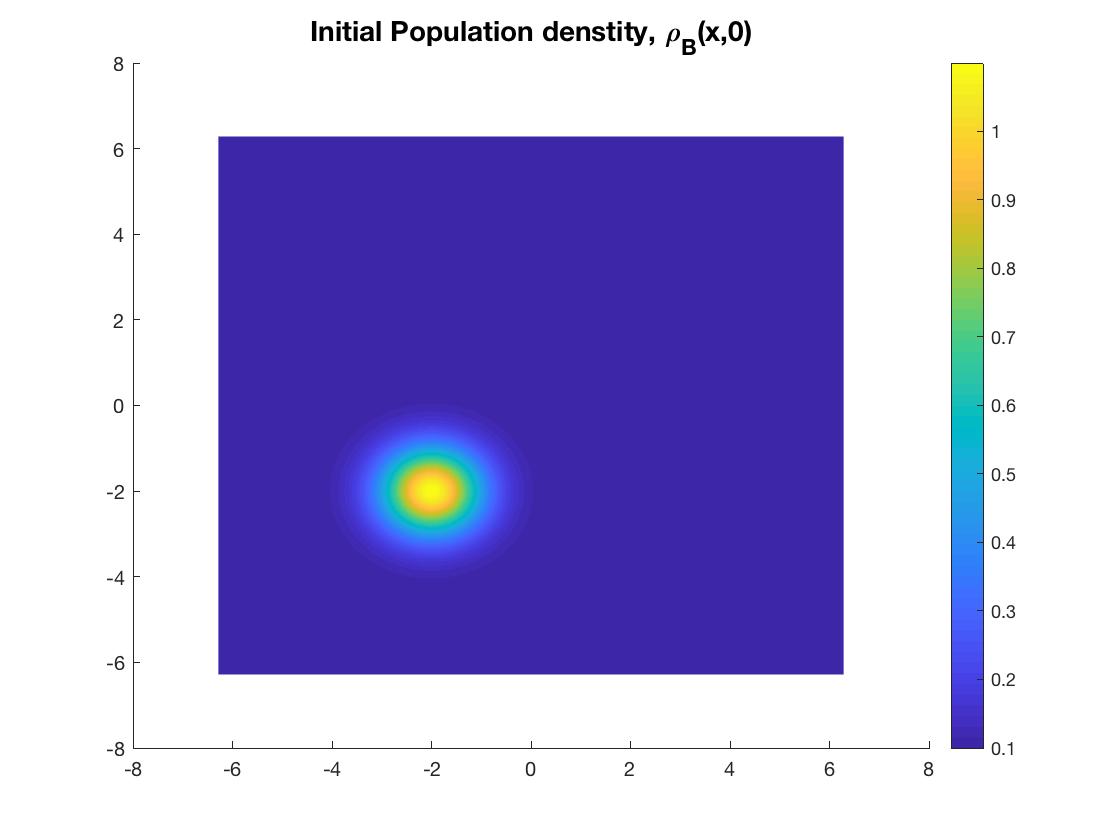}}\;
\subfloat[$\rho_B(1,x)$]{\label{fig:2}\includegraphics[width=0.32\textwidth]{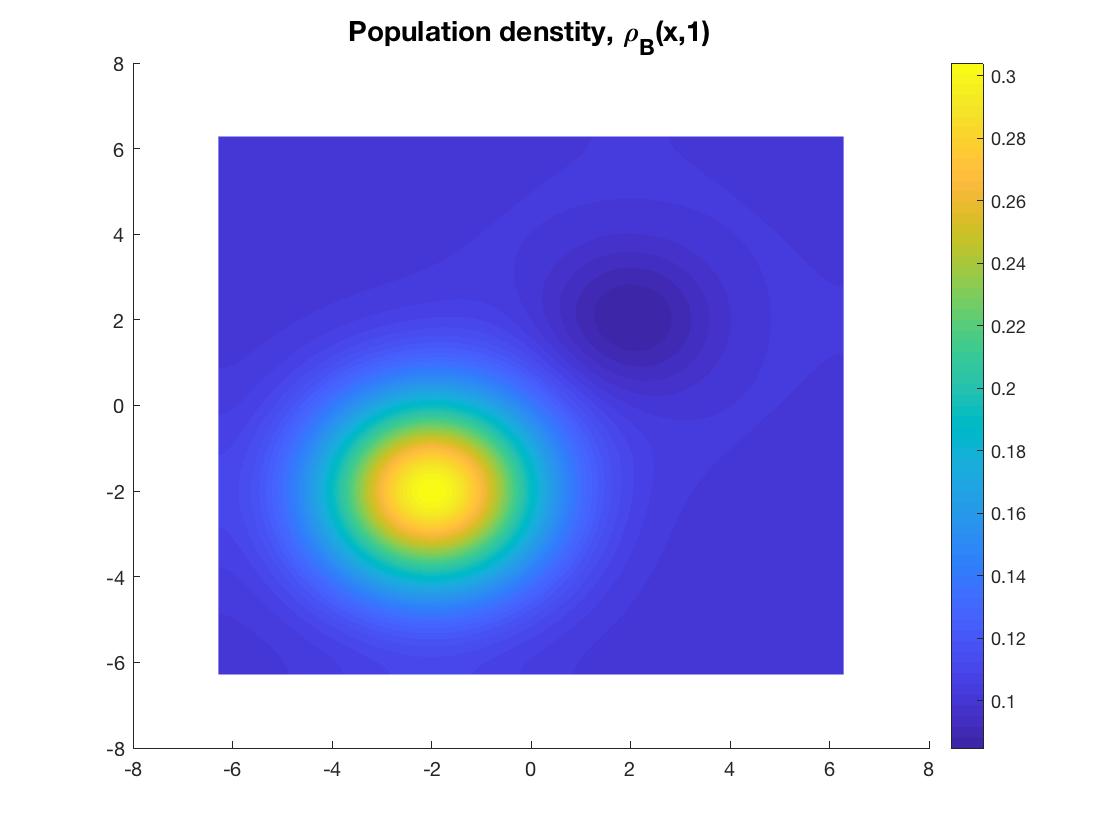}}\;
\subfloat[$\rho_B(30,x)$]{\label{fig:3}\includegraphics[width=0.32\textwidth]{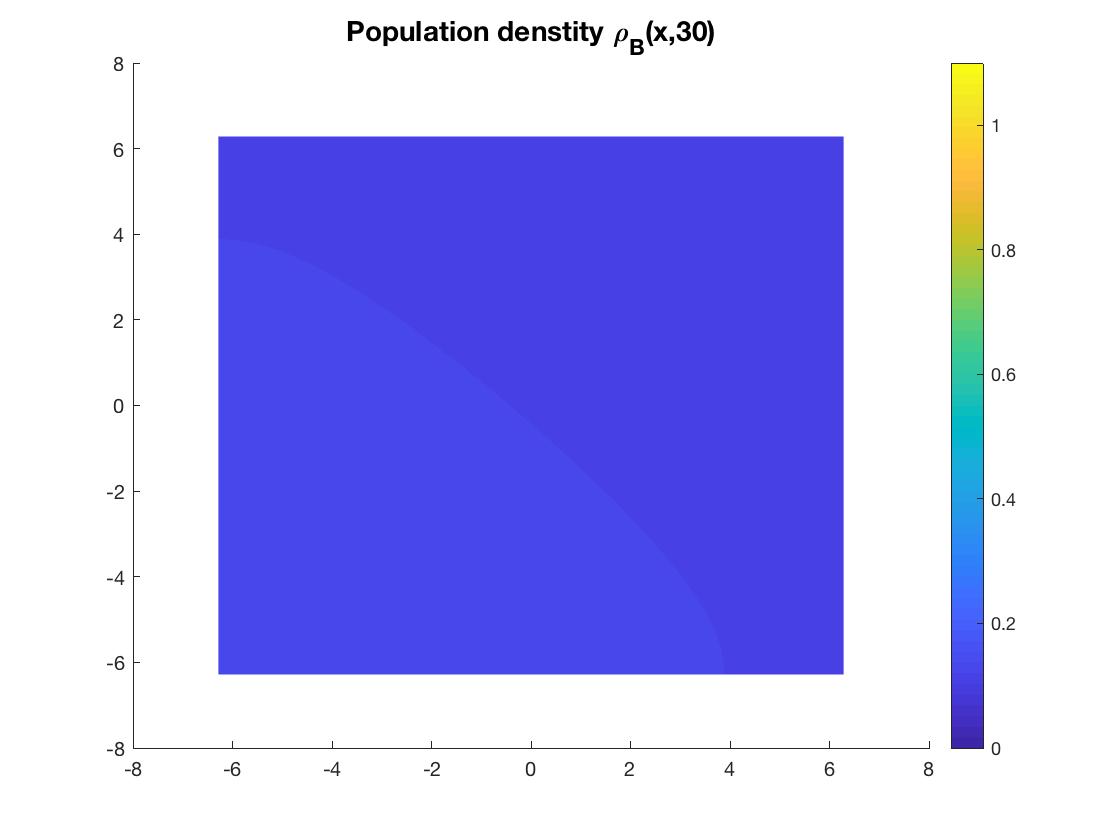}}
   \caption{Numerical solutions in two-dimension with initial densities $\rho_A(0,x)= .1+e^{-\abs{x-2}^2}$ and $\rho_B(0,x)=  .1+e^{-\abs{x+2}^2}$}   \label{fig:oc5}
\end{figure}

\section*{Acknowledgements}

A. Barbaro was supported by the NSF through grant
No. DMS-1319462. N. Rodr\'iguez was partially funded by the NSF DMS-1516778. H. Yolda\c{s} was supported by the European Research Council (ERC) under the European Union’s Horizon 2020 research and innovation programme (grant agreement No 639638). N. Zamponi acknowledges support from the Alexander von Humboldt foundation. The authors gratefully acknowledge the American Institute of Mathematics (AIM), where this project began. 

\bibliography{GangDynamics}
\bibliographystyle{ieeetr}

\end{document}